\documentclass[english]{smfart}
\usepackage{sabbah-yu}
\def\appendixname{Appendix}

\let\oldAfu\Afu
\let\newAfu\CC
\let\Afu\newAfu
\newcommand{\GGm}{\CC_v^*}
\newcommand{\cDXv}{\cD_X[v]\langle\partial_v\rangle}
\newcommand{\cDXu}{\cD_X[u]\langle\partial_u\rangle}

\begin{document}

\frontmatter
\title[On the irregular Hodge filtration]
{On the irregular Hodge filtration of exponentially twisted mixed Hodge modules}

\author[C.~Sabbah]{Claude Sabbah}
\address[C.~Sabbah]{UMR 7640 du CNRS\\
Centre de Mathématiques Laurent Schwartz\\
\'Ecole polytechnique\\
F--91128 Palaiseau cedex\\
France} \email{Claude.Sabbah@polytechnique.edu}
\urladdr{http://www.math.polytechnique.fr/~sabbah}
\thanks{(C.S.) This research was supported by the grants ANR-08-BLAN-0317-01 and ANR-13-IS01-0001-01 of the Agence nationale de la recherche.}

\author[J.-D.~Yu]{Jeng-Daw Yu}
\address[J.-D.~Yu]{Department of Mathematics\\
National Taiwan University\\
Taipei 10617\\
Taiwan}
\email{jdyu@ntu.edu.tw}
\urladdr{http://homepage.ntu.edu.tw/~jdyu/}
\thanks{(J.-D.Y.) This work was partially supported by the NCTS and the MoST of Taiwan.}


\begin{abstract}
Given a mixed Hodge module $\ccN$ and a meromorphic function $f$ on a complex manifold, we associate to these data a filtration (the irregular Hodge filtration) on the exponentially twisted holonomic module $\ccN\otimes\ccE^f$, which extends the construction of \cite{E-S-Y13}. We show the strictness of the push-forward filtered $\cD$\nobreakdash-module through any projective morphism $\pi:X\to Y$, by using the theory of mixed twistor $\cD$-modules of T.\,Mochizuki. We consider the example of the rescaling of a regular function $f$, which leads to an expression of the irregular Hodge filtration of the Laplace transform of the Gauss-Manin systems of $f$ in terms of the Harder-Narasimhan filtration of the Kontsevich bundles associated with $f$.
\end{abstract}


\maketitle
\vspace*{-2\baselineskip}%
{\let\\\relax\tableofcontents}

\mainmatter
\section{Introduction}
\subsection{The irregular Hodge filtration}\label{subsec:irregHodge}
The category of mixed Hodge modules on complex manifolds, as constructed by M.\,Saito \cite{MSaito87}, is endowed with the standard operations (push-forward by projective morphisms, pull-back by holomorphic maps, duality, etc.). In particular, the structure of the Hodge filtration in this category is well-behaved through these operations. For a meromorphic function~$f$ on a complex manifold~$X$, holomorphic on the complement~$U$ of a divisor~$D$ of~$X$, and for a mixed Hodge module with underlying filtered $\cD_X$-module $(\ccN,F_\bbullet\ccN)$, we will define an ``irregular Hodge filtration'', which is a filtration on the exponentially twisted holonomic $\cD_X$-module $\ccN\otimes\ccE^f$, where~$\ccE^f$ denote the $\cO_X$-module $\cO_X(*D)$ equipped with the twisted integrable connection $\rd+\rd f$, that we regard as a left holonomic $\cD_X$-module. We note that, although~$\ccN$ is known to have regular singularities, $\ccN\otimes\ccE^f$ has irregular singularities along the components of the divisor $D$ where~$f$ takes the value~$\infty$, hence cannot underlie a mixed Hodge module. Therefore, the irregular Hodge filtration we define on $\ccN\otimes\ccE^f$, generalizing the definition of Deligne \cite{Deligne8406}, and then \cite{Bibi08}, \cite{Yu12}, \cite{E-S-Y13}, cannot be the Hodge filtration of a mixed Hodge module in the sense of \cite{MSaito87}. There is an algebraic variant of this setting, where we assume that~$f$ is a rational function on a complex smooth variety $X$.

\begin{remarque}\label{rem:ESY}
Such a filtration has been constructed in \cite{E-S-Y13} in the following cases:
\begin{enumeratea}
\item\label{enum:ESY1}
$f$ extends as a morphism $X\to\PP^1$, $D$ is a normal crossing divisor and the filtered $\cD_X$-module $(\ccN,F_\bbullet\ccN)$ is equal to $(\cO_X(*D),F_\bbullet\cO_X(*D))$, where the filtration is given by the order of the pole \cite{Deligne70}. In such a setting, the filtration was denoted $F_\bbullet(\ccE^f(*H))$, where $H$ is the union of the components of $D$ not in $f^{-1}(\infty)$;
\item\label{enum:ESY2}
$X=Y\times\PP^1$, $f$ is the projection to $\PP^1$ and $(\ccN,F_\bbullet\ccN)$ underlies an arbitrary mixed Hodge module.
\end{enumeratea}
\end{remarque}

\begin{definition}\label{def:Rees}
By a \emph{good filtration $F_\bbullet$ indexed by $\QQ$} of a $\cD_X$-module $\ccN$, we mean a finite family $F_{\alpha+\bbullet}\ccN$ of good filtrations\footnote{As usual, this is understood with respect to the filtration by the order $F_\bbullet\cD_X$, and goodness means that $F_{\alpha+p}\ccN=0$ $p\ll0$ locally on $X$, and $\gr^F_{\alpha+\cbbullet}\ccN$ is $\gr^F\!\cD_X$-coherent.} indexed by $\ZZ$, parametrized by $\alpha$ in a finite subset $A$ of $[0,1)\cap\QQ$, such that $F_{\alpha+p}\ccN\subset F_{\beta+q}\ccN$ for all $\alpha,\beta\in A$ and $p,q\in\ZZ$ satisfying $\alpha+p\leq\beta+q$.

We can thus regard it as a single increasing filtration indexed by $\QQ$, such that $F_{\alpha+p}\ccN/F_{<\alpha+p}\ccN=0$ for any $\alpha,p$, except for $\alpha$ in a finite set $A$ of $[0,1)\cap\QQ$.

For each~$\alpha$, the \emph{Rees module} $R_{F_{\alpha+\bbullet}}\ccN$ is the graded module defined as $\sum_pF_{\alpha+p}\ccN\hb^p$, where~$\hb$ is a new (Laurent) polynomial variable. Then we set $R_{F_\bbullet}\ccN:=\bigoplus_{\alpha\in A}R_{F_{\alpha+\bbullet}}\ccN$. We can then regard a (usual) good filtration indexed by $\ZZ$ as a good filtration indexed by $\QQ$.
\end{definition}

\begin{theoreme}\label{th:1main}
Let $f$ be a meromorphic function on $X$, holomorphic on $U=X\setminus\nobreak D$, where $D$ is a divisor in $X$. For each filtered holonomic $\cD_X$-module $(\ccN,F_\bbullet\ccN)$ underlying a mixed Hodge module one can define canonically and functorially a good $F_\bbullet\cD_X$-filtration $F^\irr_\bbullet(\ccN\otimes\ccE^f)$ indexed by $\QQ$ which satisfies the following properties:
\begin{enumerate}
\item\label{th:1main0}
Through the canonical isomorphism $(\ccN\otimes\ccE^f)_{|U}=\ccN_{|U}$, we have $F^\irr_\bbullet(\ccN\otimes\nobreak\ccE^f)_{|U}=F_\bbullet\ccN_{|U}$.
\item\label{th:1main1}
For each morphism $\varphi:(\ccN_1,F_\bbullet\ccN_1)\to(\ccN_2,F_\bbullet\ccN_2)$ underlying a morphism of mixed Hodge modules, the corresponding morphism
\[
\varphi^f:(\ccN_1\otimes\ccE^f,F^\irr_\bbullet(\ccN_1\otimes\ccE^f))\to(\ccN_2\otimes\ccE^f,F^\irr_\bbullet(\ccN_2\otimes\ccE^f))
\]
is strictly filtered.
\item\label{th:1main2}
For each $\alpha\in[0,1)$, the push-forward $\pi_+(\ccN\otimes\nobreak\ccE^f,F^\irr_{\alpha+\bbullet}(\ccN\otimes\nobreak\ccE^f))$ by any projective morphism \hbox{$\pi:X\to Y$} is strict.
\item\label{th:1main3}
Let $\pi:X\to Y$ be a projective morphism and let $h$ be a meromorphic function on $Y$, holomorphic on $V=Y\setminus D_Y$ for some divisor $D_Y$ in $Y$. Assume that $D_X:=\nobreak\pi^{-1}(D_Y)$ is a divisor, and set $U=\pi^{-1}(V)$ and $f=h\circ\pi$. Then the cohomology of the filtered complex $\pi_+(\ccN\otimes\nobreak\ccE^f,F^\irr_\bbullet(\ccN\otimes\nobreak\ccE^f))$, which is strict by \eqref{th:1main2}, satisfies
\[
\cH^j\pi_+R_{F^\irr}(\ccN\otimes\ccE^f)=R_{F^\irr}\big[(\cH^j\pi_+\ccN)\otimes\ccE^h\big].
\]
\item\label{th:1main4}
In cases \ref{rem:ESY}\eqref{enum:ESY1} and \ref{rem:ESY}\eqref{enum:ESY2} above, the filtration $F^\irr_\bbullet$ coincides with the filtration $F^\Del_\bbullet$ constructed in \cite{E-S-Y13}.
\end{enumerate}
\end{theoreme}

The proof of the theorem is given in \S\ref{subsec:proofthmain} and relies much on the theory of mixed twistor $\cD$-modules of T.\,Mochizuki \cite{Mochizuki11}. This theory allows one to simplify and generalize some of the arguments given in \cite{E-S-Y13}, by giving a general framework to treat, from the Hodge point of view, irregular $\cD$-modules like $\ccE^f$. By specializing~\eqref{th:1main2} to the case where $Y$ is a point we obtain:

\begin{corollaire}\label{cor:E1degen}
For $(\ccN,F_\bbullet\ccN)$ underlying a mixed Hodge module on a smooth projective variety $X$, the spectral sequence attached to the hypercohomology of the filtered de~Rham complex $F^\irr_{\alpha+\bbullet}\DR(\ccN\otimes\nobreak\ccE^f)$ degenerates at~$E_1$ for each $\alpha\in[0,1)$.\qed
\end{corollaire}

\begin{remarque}
The assumption that $D:=X\setminus U$ is a divisor is not mandatory, but simplifies the statement. In general, higher cohomology modules supported on $X\setminus U$ may appear for $\ccN\otimes\ccE^f$.
\end{remarque}

\subsection{Rescaling a function}\label{subsec:rescaling}
The case \ref{rem:ESY}\eqref{enum:ESY1} is essentially the only case where we can give an explicit expression for $F_\bbullet^\irr\ccE^f(*H)$ (see \cite{E-S-Y13}, according to \ref{th:1main}\eqref{th:1main4}). Recall that we consider a smooth complex projective variety $X$ together with a morphism \hbox{$f:X\to\PP^1$}. We set $P_\red=f^{-1}(\infty)$ and $P=f^*(\infty)$. We also introduce a supplementary divisor $H$ (which could be empty) having no common components with $P_\red$, and we assume that $D:=P_\red\cup H$ has normal crossings. We set $U=X\setminus D$. We will also consider $X$ as a complex projective manifold equipped with its analytic topology, that we will denote $X^\an$ when the context is not clear.

Our main example in this article, that we consider in Part \ref{part:2}, is that of the rescaling of the function \hbox{$f:X\to\PP^1$}. The rescaled function with rescaling parameter~$v$ is the function \hbox{$vf:U\times\Afu_v\to\Afu$}, defined by $(x,v)\mto vf(x)$. This function does not extend as a morphism to $X\times\PP^1\to\PP^1$.

We consider the projective line $\PP^1_v$ covered by two charts $\Afu_v$ and $\Afu_u$ whose intersection is denoted by~$\GGm$, and we regard $vf$ as a rational function $vf:X\times\PP^1_v\text{ - -}\to\PP^1$. We are therefore in the situation in the beginning of the previous subsection, with underlying space $\ccX:=X\times\PP^1_v$ and reduced pole divisor $\ccP_\red:=(P_\red\times\PP^1_v)\cup (X\times\infty)$, where~$\infty\in\PP^1_v$ denotes the point $u=0$. We will also set $\ccP=(P\times\PP^1)+(X\times\{\infty\})$, $\ccH=H\times\PP^1_v$ and $\ccD=\ccP_\red\cup\ccH$.

We denote by $\ccE^{(v{:}u)f}(*\ccH)$ the $\cO_{X\times\PP^1_v}$-module $\cO_{X\times\PP^1_v}(*\ccD)$ equipped with the connection $\rd+\rd(vf)$ (on the open set $X\times\Afu_v$) and $\rd+\rd(f/u)$ (on the open set $X\times\Afu_u$). We denote its restriction to the corresponding open subsets by $\ccE^{vf}(*\ccH)$ and $\ccE^{f/u}(*\ccH)$ respectively. According to Theorem \ref{th:1main}, it is equipped with an irregular Hodge filtration. We make it partly explicit in Theorem \ref{th:FFirr} (only partly, because around $u=0$ we only make explicit its restriction to the Brieskorn lattice, \cf \S\ref{subsec:computuchart}).

\Subsection{Variation of the irregular Hodge filtration and the Kontsevich bundles}\label{subsec:varirrhodgefiltrkontsevichbundles}

Regarding now $v\in\CC^*$ as a parameter and considering the push-forward by \hbox{$q:X\times\PP^1_v\to\PP^1_v$} of the rescaling $\ccE^{(v:u)f}(*\ccH)$, our aim is to describe the variation with~$v$ of the irregular Hodge filtration $F^\irr_\bbullet\bH^k\big(U,(\Omega_U^\cbbullet,\rd+v\rd f)\big)$ considered in \cite{E-S-Y13}, and its limiting behaviour when $v\to0$ or $v\to\infty$.

The irregular Hodge filtration is conveniently computed with the Kontsevich complex. Recall that M.\,Kontsevich has associated to $f:X\to\PP^1$ as in \S\ref{subsec:rescaling} and to $k\geq0$ the subsheaf~$\Omega_f^k$ of $\Omega_X^k(\log D)$ consisting of logarithmic $k$-forms $\omega$ such that~$\rd f\wedge\omega$ remains a logarithmic $(k+1)$-form, a condition which only depends on the restriction of $\omega$ to a neighbourhood of the reduced pole divisor $P_\red=f^{-1}(\infty)$. For each $\alpha\in[0,1)$, let us denote by $[\alpha P]$ the divisor supported on~$P_\red$ with multiplicity $[\alpha e_i]$ on the component~$P_i$ of $P:=f^*(\infty)$ with multiplicity~$e_i$. One can also define a subsheaf $\Omega_f^k(\alpha)$ of $\Omega_X^k(\log D)([\alpha P])$ by the condition that $\rd f\wedge\omega$ is a section of $\Omega_X^{k+1}(\log D)([\alpha P])$, so that the case $\alpha=0$ is that considered by Kontsevich. Clearly, only those $\alpha$ such that $\alpha e_i\in\ZZ$ for some $i$ are relevant. If $f$ is the constant map, then $\Omega_f^k=\Omega_X^k(\log D)$. One of the main results of \cite{E-S-Y13}, suggested and proved by Kontsevich when $P=P_\red$, is the equality, for each $k$,
\[
\dim\bH^k\big(U,(\Omega_U^\cbbullet,\rd+\rd f)\big)=\sum_{p+q=k}\dim\bH^q(X,\Omega_f^p(\alpha)).
\]
More precisely, for each pair $(u,v)\in\CC^2$ and each $\alpha\in[0,1)$ one can form a complex $(\Omega_f^\cbbullet(\alpha),u\rd+\nobreak v\rd f)$ and it is shown that the dimension of the hypercohomology $\bH^k\big(X,(\Omega_f^\cbbullet(\alpha),u\rd+v\rd f)\big)$ is independent of $(u,v)\in\CC^2$ and $\alpha$ and is equal to the above value. The \emph{irregular Hodge numbers} are then defined~as
\begin{equation}\label{eq:irregHodgenumber}
h_\alpha^{p,q}(f)=\dim\bH^q(X,\Omega_f^p(\alpha)).
\end{equation}
We have $h_\alpha^{p,q}(f)\neq0$ only if $p,q\geq0$ and $p+q\leq2\dim X$. (\cf Remark \ref{rem:G0G0}\eqref{rem:G0G03} for the mirror symmetry motivations related to the irregular Hodge filtration.) If $f$ is the constant map, we recover the results of Deligne \cite{Deligne70,DeligneHII}:
\begin{align*}
\dim\bH^k(U,\CC)=\dim\bH^k\big(U,(\Omega_U^\cbbullet,\rd)\big)&=\dim\bH^k\big(X,(\Omega_X^\cbbullet(\log D),\rd)\big)\\
&=\sum_{p+q=k}\dim H^q\big(X,\Omega_X^p(\log D)\big).
\end{align*}
The Hodge numbers reduce here to $h^{p,q}(X,D)=\dim H^q\big(X,\Omega_X^p(\log D)\big)$.

Following the suggestion of M.\,Kontsevich, let us define the \emph{Kontsevich bundles}~$\cK^k(\alpha)$ on $\PP^1_v$. We~set
\begin{equation}\label{eq:Kontbundles}
\begin{split}
\cK^k_v(\alpha)&:=\bH^k\big(X,(\Omega_f^\cbbullet(\alpha)[v],\rd+v\rd f)\big),\\
\cK^k_u(\alpha)&:=\bH^k\big(X,(\Omega_f^\cbbullet(\alpha)[u],u\rd+\rd f)\big).
\end{split}
\end{equation}
Using the isomorphism $\CC[u,u^{-1}]\isom\CC[v,v^{-1}]$ given by $u\mto v^{-1}$, we have a natural quasi-isomorphism
\begin{equation}\label{eq:varphi}
u^{\cbbullet}:(\Omega_f^\cbbullet(\alpha)[v,v^{-1}],\rd+v\rd f)\isom(\Omega_f^\cbbullet(\alpha)[u,u^{-1}],u\rd+\rd f)
\end{equation}
induced by the multiplication by $u^p$ on the $p$th term of the first complex. Since we know by the above mentioned results that both modules $\cK^k_v(\alpha),\cK^k_u(\alpha)$ are free over their respective ring $\CC[v]$ or $\CC[u]$, the identification
\[
\bH^k(u^{\cbbullet}):\bH^k\big(X,(\Omega_f^\cbbullet(\alpha)[v,v^{-1}],\rd+v\rd f)\big)\simeq\bH^k\big(X,(\Omega_f^\cbbullet(\alpha)[u,u^{-1}],u\rd+\rd f)\big)
\]
allows us to glue these modules as a bundle $\cK^k(\alpha)$ on $\PP^1_v$. The $E_1$-degeneration property can be expressed by the injectivity
\begin{equation}\label{eq:E1degenK}
\begin{aligned}
\bH^k\big(X,\sigma^{\geq p}(\Omega_f^\cbbullet(\alpha)[v],\rd+v\rd f)\big)&\hto\bH^k\big(X,(\Omega_f^\cbbullet(\alpha)[v],\rd+v\rd f)\big),\\
\bH^k\big(X,\sigma^{\geq p}(\Omega_f^\cbbullet(\alpha)[u],u\rd+\rd f)\big)&\hto\bH^k\big(X,(\Omega_f^\cbbullet(\alpha)[u],u\rd+\rd f)\big),
\end{aligned}
\end{equation}
where $\sigma^{\geq p}$ denotes the stupid truncation. Since this truncation is compatible with the gluing $u^{\cbbullet}$, this defines a filtration $\sigma^{\geq p}\cK^k(\alpha)$. When restricted to $\CC^*_v$, this produces the family $F_\alpha^{\irr,p}\bH^k\big(U,(\Omega_U^\cbbullet,\rd+v\rd f)\big)$.

We also notice that the $p$th graded bundle is then isomorphic to $\cO_{\PP^1}(p)^{h_\alpha^{p,k-p}(f)}$, so this filtration is the Harder-Narasimhan filtration $F^\cbbullet\cK^k(\alpha)$ and the Birkhoff-Grothendieck decomposition of $\cK^k(\alpha)$ reads
\begin{equation}\label{eq:bgkontsevich}
\cK^k(\alpha)\simeq\bigoplus_{p=0}^k\cO_{\PP^1}(p)^{h_\alpha^{p,k-p}(f)}.
\end{equation}
In particular, all slopes of $\cK^k(\alpha)$ are nonnegative and we have
\[
\deg\cK^k(\alpha)=\sum_{p=0}^kp\cdot h_\alpha^{p,k-p}(f).
\]

We will show (\cf Lemma \ref{lem:omegaalphau}) that each $\cK^k(\alpha)$ is naturally equipped with a meromorphic connection having a simple pole at $v=0$ and a double pole at most at $v=\infty$. It follows from a remark due to Mochizuki (\cf Remark \ref{rem:HSMochizuki}) that the Harder-Narasimhan filtration satisfies the Griffiths transversality condition with respect to the connection. This is a concrete description of the variation of the irregular Hodge filtration (Corollary \ref{cor:HNfiltr}).

Our main result concerns the limiting behaviour of the variation of the irregular Hodge filtration when $v\to0$, expressed in this model.

\begin{theoreme}\label{th:2main}\mbox{}
\begin{enumerate}
\item\label{th:2main2}
The meromorphic connection $\nabla$ on $\cK^k(\alpha)$ has a logarithmic pole at $v=0$ and the eigenvalues of its residue $\Res_{v=0}\nabla$ belong to $[-\alpha,-\alpha+1)\cap\QQ$.
\item\label{th:2main4}
On each generalized eigenspace of $\Res_{v=0}\nabla$ the nilpotent part of the residue strictly shifts by $-1$ the filtration naturally induced by the Harder-Narasimhan filtration.
\end{enumerate}
\end{theoreme}

The proof of Theorem \ref{th:2main}, which is sketched in \S\ref{sec:intropart2}, does not remain however in the realm of Kontsevich bundles. It is obtained through an identification of the Kontsevich bundles with the bundles $\cH^k(\alpha)$ obtained from the push-forward $\cD$-modules $\cH^k$ of $\ccE^{(v{:}u)f}(*\ccH)$ (\cf\S\ref{subsec:rescaling}) by the projection $q:X\times\PP^1_v\to\PP^1_v$. Recall that $\cH^k:=R^kq_*\DR_{X\times\PP^1_v/\PP^1_v}\ccE^{(v{:}u)f}(*\ccH)$ is a holonomic $\cD_{\PP^1_v}$\nobreakdash-module for each~$k$. It~is equipped with its irregular Hodge filtration $F^\irr_\bbullet\cH^k$ obtained by push-forward, according to Theorem \ref{th:1main}\eqref{th:1main2}. We define the bundles $\cH^k(\alpha)$ by using this filtration, and the main comparison tools with the Kontsevich bundles $\cK(\alpha)$ are provided by Theorems \ref{th:4main} and \ref{th:5main}.

\Subsection{Motivations and open questions}
We have already discussed in \cite[Introduction]{E-S-Y13} the motivation coming from estimating $p$-adic eigenvalues of Frobenius (Deligne) and that coming from mirror symmetry (Kontsevich). We list below some more related questions and possible applications for further investigations.

\subsubsection*{Numerical invariants of mixed twistor $\cD$-modules}
The theory of mixed twistor $\cD$\nobreakdash-modules, as developed by T.\,Mochizuki \cite{Mochizuki11}, is the convenient framework to treat wild Hodge theory. However, this theory produces very few numerical invariants having a Hodge flavor (like Hodge numbers, degrees of Hodge bundles, etc.). The~irregular Hodge filtration, when it does exist, is intended to provide such invariants. Let us emphasize that, contrary to classical Hodge theory, the irregular Hodge filtration is only a by-product of the mixed twistor structure, but is not constitutive of its definition.

Is there a suitable well-behaved category of \emph{wild Hodge $\cD$-modules} with a forgetful functor to the category of mixed twistor $\cD$-modules? What about the expected functorial and degeneration properties? The exponentially twisted Hodge modules should give rise to an object in such a category. Moreover, following the definition due to Simpson of systems of Hodge bundles, we can expect that the objects in this suitable category should carry an internal symmetry (a $\CC^*$-action in the case of tame twistor $\cD$-modules). A possible approach to this question would be to search for the desired category as the category of integrable mixed twistor $\cD$-modules endowed with supplementary structures on the object obtained by rescaling the twistor variable.

\subsubsection*{Analogies with Hodge theory}
Going further in the direction of Hodge theory, one may wonder whether the irregular Hodge filtration, when it exists, shares similar properties with the usual Hodge filtration on mixed Hodge modules. For example, for a morphism $f:X\to\PP^1$, the $\cD_X$-module $\ccE^f$ underlies a pure integrable twistor $\cD$-module (\cf Proposition \ref{prop:EXf}\eqref{prop:EXf3}) and is equipped with an irregular Hodge filtration (\cf Theorem \ref{th:1main} with $\ccN=(\cO_X,\rd)$). Let $\pi:X\to Y$ be a projective morphism. According to the decomposition theorem for pure twistor $\cD$-modules \cite{Mochizuki08}, the push-forward $\pi_+\ccE^f$ decomposes, together with its twistor structure, into a direct sum of possibly shifted simple holonomic $\cD$-modules. One can wonder whether the analogues of Koll\'ar's conjectures (proved by M.\,Saito \cite{MSaito91c}) hold for the irregular Hodge filtration of $\ccE^f$.

Also the question of the limiting behaviour, in the sense of Schmid, of the irregular Hodge filtration raises interesting questions. We treat the case of a tame degeneration (the case of $\ccE^{vf}$ when $v\to0$) in \S\ref{sec:Evf}, but the case of a non-tame degeneration (like $u\to0$ in \S\ref{sec:Euf}) remains unclear in general. We expect that the good behaviour (by definition) of the mixed twistor modules by taking irregular nearby cycles along a holomorphic function should lead to specific limiting properties for the irregular Hodge filtration, when it exists.

\subsubsection*{Extended motivic-exponential $\cD$-modules}
Recall that, following \cite[6.2.4]{B-B-D81}, one defines the notion of a simple regular holonomic $\cD$-module \emph{of geometric origin} on a smooth complex algebraic variety $X$ if it appears as a simple subquotient in a regular holonomic $\cD_X$-module obtained by using only standard geometric functors starting from the case where the variety is a point. In particular, such a simple regular holonomic $\cD_X$-module is a simple summand of a regular holonomic $\cD$-module underlying a polarizable $\QQ$-Hodge module of some weight, as defined by M.\,Saito \cite{MSaito86,MSaito87}. It therefore underlies a simple complex polarizable Hodge module. In other words, there exists an irreducible algebraic closed subvariety $Z\subset X$, a Zariski smooth open set $Z^\circ\subset Z$, and an irreducible local system on $Z^\circ$, underlying a polarizable complex variation of Hodge structure (\cf \cite{Deligne87}), such that this regular holonomic $\cD_X$-module corresponds, via the inverse Riemann-Hilbert correspondence, to the intermediate extension of this local system by the inclusion $Z^\circ\hto X$. In particular, it comes equipped with a good filtration (that induced by the polarizable $\QQ$-Hodge module) and the corresponding filtered $\cD$-module is a direct summand of the filtered $\cD$-module underlying the polarizable $\QQ$-Hodge module.

M.\,Kontsevich \cite{Kontsevich09} has defined the category of \emph{motivic-exponential $\cD$-modules} by adding the twist by~$\ccE^f$ for any rational function~$f$ to the standard permissible operations on regular holonomic $\cD$-modules of geometric origin on algebraic varieties. By \cite{Mochizuki11}, any such motivic-exponential $\cD$-module underlies a pure wild twistor $\cD$-module (\cf\cite{Mochizuki08}).

There is also the category of \emph{extended motivic-exponential $\cD$-modules}, by authorizing extensions of such objects, but we will not consider it here.

One can expect that any motivic-exponential $\cD$-module on a complex alge\-braic \hbox{variety} is endowed with a canonical irregular Hodge filtration, and that this filtration has a good behaviour with respect to the various permissible functors (the six operations of Grothendieck, the nearby and vanishing cycles along a function, and the twist by some $\ccE^f$). Theorem \ref{th:1main} is a step toward this expected result.

\begin{remarque}[Hodge filtration in presence of very irregular singularities]
The holonomic $\cD_Y$\nobreakdash-modules one obtains as $\cH^k\pi_+(\ccN\otimes\ccE^f)$ when $\pi$ is any projective morphism may have irregular singularities much more complicated than an exponential twist of a regular singularity. For example, if $Y$ is a disc, it is shown in \cite{Roucairol07} that any formal meromorphic connection at $0\in Y$ can be produced as the formalization at the origin of a connection obtained by the procedure of Th.\,\ref{th:1main}\eqref{th:1main2} for some suitable~$\ccN$ on $X=Y\times\PP^1$.

However, these $\cD_Y$-modules come equipped with a good filtration $F_\bbullet\cH^k\pi_+(\ccN\otimes\nobreak\ccE^f)$ obtained by pushing-forward $F^\irr_\bbullet(\ccN\otimes\ccE^f)$.
If $Y$ is projective and if for example $\cH^k\pi_+(\ccN\otimes\ccE^f)=0$ except for $k=k_o$ then, according to Corollary \ref{cor:E1degen}, we obtain the degeneration at $E_1$ of the spectral sequence attached to the hypercohomology of the filtered de~Rham complex $F_\bbullet\DR\cH^{k_o}\pi_+(\ccN\otimes\ccE^f)$. Examples of this kind can be obtained by the procedure of \cite{Roucairol07} with arbitrary complicated irregular singularities.
\end{remarque}

\subsubsection*{Acknowledgments}
We thank Maxim Kontsevich for suggesting us the properties stated in Theorem \ref{th:2main} and Takuro Mochizuki for explaining us some of his results on mixed twistor $\cD$-modules and his useful comments. In particular, he suggested various improvements and simplifications to the first version of this article. We owe him the statement of Theorem \ref{th:4main}. Last but not least, we thank Hélène Esnault for the many discussions we had together and for many suggestions and questions on the subject of this article.

\part{Irregular Hodge filtration and twist by \texorpdfstring{$\ccE^f$}{Ef}}

\section{Exponentially regular holonomic \texorpdfstring{$\cD$}{D}-modules}

\subsection{The graph construction}\label{subsec:graph}
We refer to the expository book \cite{H-T-T08} or the expository article \cite{Mebkhout04} for basic properties of regular holonomic $\cD$-modules.

Let $X$ be a complex manifold and let $P_\red$ be a reduced divisor in $X$. We set $U=X\setminus P_\red$. Let $f$ be a meromorphic function on $X$ which is holomorphic on $U$ whose pole divisor $P$ is \emph{exactly} supported by $P_\red$, \ie $f$ takes the value $\infty$ generically on each irreducible component of $P_\red$. By definition, locally analytically on $P_\red$, the function~$f$ can be written as the quotient of two holomorphic functions with no common factor, such that the zero divisor may intersect $P_\red$ in codimension two in $X$ at most. There exists a proper modification $\pi:X'\to X$ with $X'$ smooth, which is an isomorphism over~$U$, and a holomorphic map $f':X'\to\PP^1_t$, such that $f'_{|\pi^{-1}(U)}=f\circ \pi_{|\pi^{-1}(U)}$. The pole divisor $P'$ of~$f'$ satisfies $P'_\red\subset \pi^{-1}(P_\red)=:D'$, and the inclusion may be strict. Let $i_f:U\hto U\times\Afu_t$ denote the graph inclusion of~$f$. The closure $\ov U_f$ of $U_f:=i_f(U)$ in $X\times\PP^1_t$ is a closed analytic set of codimension one, equal to the projection by the proper modification $\pi\times\id:X'\times\PP^1_t\to X\times\PP^1_t$ of the graph $i_{f'}(X')$. The projection $p:X\times\PP^1_t\to X$ induces a proper modification $\ov U_f\to X$, and the pull-back of~$U$ in~$\ov U_f$ maps isomorphically to~$U$. In particular, we have $(X\times\infty)\cap\ov U_f\subset (P_\red\times\PP^1_t)\cap \ov U_f$. We summarize this in the following diagram.
\begin{equation}\label{eq:diagram}
\begin{array}{c}
\xymatrix@C=1cm@R=.6cm{
X'\ar@/_2pc/[rrrddd]_(.2){f'}|!{[ddr];[dr]}\hole|!{[ddr];[ddrr]}\hole
\ar@/^2pc/@{^{ (}->}[rr]_(.6){i_{f'}}\ar[d]_{\pi}\ar[r]^-\sim&i_{f'}(X')\ar[d]_{\pi\times\id}\ar@{^{ (}->}[r]&X'\times\PP^1\ar[d]_{\pi\times\id}\ar@/^1pc/[rddd]^{q'}\\
X&\ov U_f\ar@{^{ (}->}[r]&X\times\PP^1\ar[rdd]^q\\
U\ar@/_2.5pc/[rrrd]^f
\ar@/_2pc/@{^{ (}->}[rr]^(.6){i_f}
\ar@{^{ (}->}[u]\ar[r]^-\sim&U_f\ar@{^{ (}->}[u]\ar@{^{ (}->}[r]&U\times\Afu\ar@{^{ (}->}[u]\\
&&&\PP^1
}
\end{array}
\end{equation}

\smallskip
Let $\ccN$ be a holonomic $\cD_X$-module. We assume that $\ccN$ is equal to its localization $\ccN(*P_\red)$ (if not, replace $\ccN$ with $\ccN(*P_\red)$, which is also a holonomic $\cD_X$-module, by a theorem of Kashiwara). The localized pull-back $\ccN':=\pi^+\ccN(*D')$ consists of a single holonomic $\cD_{X'}$\nobreakdash-module. We then recover $\ccN$ as the push-forward $\pi_+\ccN'=\cH^0\pi_+\ccN'$ (see \eg \cite[Prop.\,8.13]{Bibi10}).

Let us set $\ccM'=i_{f',+}\ccN'$. Then $\ccM'=\ccM'(*(D'\times\PP^1_t))$ and since $\Supp\ccM'\cap(X'\times\nobreak\infty)\subset(D'\times\PP^1_t)$, we also have $\ccM'=\ccM'(*[(D'\times\PP^1_t)\cup(X'\times\infty)])$. We clearly have $\ccN'=p'_+\ccM'=\cH^0p'_+\ccM'$.

We set $\ccM=(\pi\times\id)_+\ccM'=\cH^0(\pi\times\id)_+\ccM'$. Then
\[
\ccM=\ccM(*(P\times\PP^1_t))=\ccM(*[(P\times\PP^1_t)\cup(X\times\infty)]),
\]
and $\ccN=p_+\ccM=\cH^0p_+\ccM$. We notice that $\ccM$ does not depend on the choice of $\pi:X'\to X$. We will use the notation $\ccM=i_{f,\oplus}\ccN$, for which we still have $p_+i_{f,\oplus}=\id$, and which coincides with $i_{f,+}\ccN$ if $f$ extends from $X$ to $\PP^1$ (\ie if we can take $\pi=\id$, so that $f'=f$).

\begin{lemme}\label{lem:MNreg}
If $\ccN$ is regular holonomic, so is $\ccM=i_{f,\oplus}\ccN$.
\end{lemme}

\begin{proof}
Indeed, $\ccN'$ is then regular, hence $\ccM'$ also, and then $\ccM$ too.
\end{proof}

\begin{remarque}[The graph construction for mixed Hodge modules]\label{rem:graphMHM}
Let us now start with a filtered $\cD_X$-module $(\ccN,F_\bbullet\ccN)$ underlying a mixed Hodge module \cite{MSaito87}. We still assume that $\ccN=\ccN(*P_\red)$ (if this is not the case, we use the localization functor in the category of mixed Hodge modules to fulfill the assumption). The construction of \S\ref{subsec:graph} can be done for mixed Hodge modules, by using the corresponding functors in the category of mixed Hodge modules. We therefore get a mixed Hodge module $(\ccM,F_\bbullet\ccM)$ on $X\times\PP^1_t$ such that $p_+(\ccM,F_\bbullet\ccM)=\cH^0p_+(\ccM,F_\bbullet\ccM)=(\ccN,F_\bbullet\ccN)$. If $f$ extends as a morphism $X\to\PP^1$, then $(\ccM,F_\bbullet\ccM)=i_{f,+}(\ccN,F_\bbullet\ccN)$.
\end{remarque}

\subsection{Exponential twist of holonomic \texorpdfstring{$\cD$}{D}-modules}\label{subsec:exptwisthol}
The differential $\rd f$ of the function $f:U\to\Afu_t$ extends as a meromorphic $1$-form on $X$ with poles along $P_\red$. We denote by~$\ccE^f$ the free $\cO_X(*P_\red)$-module of rank one equipped with the connection $\rd+\rd f$. For $\ccN$ as in \S\ref{subsec:graph} (in particular, $\ccN=\ccN(*P_\red)$), we consider the holonomic $\cD_X$-module $\ccN\otimes_{\cO_X}\ccE^f$.

\begin{lemme}\label{lem:if+}
For $\ccM=i_{f,\oplus}\ccN$, we have $\ccM\otimes\ccE^t\simeq i_{f,\oplus}(\ccN\otimes\ccE^f)$.
\end{lemme}

This implies $\ccN\otimes\ccE^f\simeq p_+(\ccM\otimes\ccE^t)=\cH^0p_+(\ccM\otimes\ccE^t)$.

\begin{proof}
Assume first that $f$ extends as a map $X\to\PP^1_t$. We will work in the chart centered at $\infty$ in $\PP^1$, with coordinate $t'$, and we will set $g=(t'\circ f)^{-1}$, so that $f^{-1}(\infty)=g^{-1}(0)$. We denote by $\reg$ the generator of $\ccE^{1/g}$. We have
\[
i_{g,+}(\ccN\otimes\ccE^{1/g})=\bigoplus_k(\ccN\otimes\ccE^{1/g})\otimes\partial_{t'}^k\delta(t'-g)
\]
with its standard $\cD_{X\times\Afu_{t'}}$-module structure. There exists thus a unique $\cO_X[\partial_{t'}]$-linear isomorphism $i_{f,+} (\ccN \otimes\ccE^{1/g})\isom\ccM \otimes\ccE^{1/t'}$ induced by
\[
(n\otimes\reg)\otimes\delta(t'-g)\mto(n\otimes\delta(t'-g))\otimes\retp.
\]
In other words, for each $k$,
\[
(n\otimes\reg)\otimes\partial_{t'}^k\delta(t'-g)\mto\partial_{t'}^k\big[(n\otimes\delta(t'-g))\otimes\retp\big].
\]
By using the same argument as in the proof of \cite[(1.6.5)]{E-S-Y13}, one shows that this isomorphism is $\cD_{X\times\PP^1}$-linear.

Let us now consider the general case. By definition,
\[
i_{f,\oplus}(\ccN \otimes\ccE^f)=(\pi\times\id)_+i_{f',+}\big[\pi^+(\ccN\otimes\ccE^f)(*D')\big].
\]
One then checks that
\[
\pi^+(\ccN\otimes\ccE^f)(*D')=(\pi^+\ccN)(*D')\otimes\ccE^{f'}=\ccN'\otimes\ccE^{f'},
\]
so $i_{f',+}\big[\pi^+(\ccN\otimes\ccE^f)(*D')\big]=\ccM'\otimes\nobreak\ccE^t$ by the argument above. Then, because $\ccE^t=(\pi\times\id)^+\ccE^t$, we have $(\pi\times\nobreak\id)_+(\ccM'\otimes\nobreak\ccE^t)=\ccM\otimes\ccE^t$.
\end{proof}

\Subsection{Exponentially regular holonomic \texorpdfstring{$\cD$}{D}-modules}

\begin{lemme}\label{lem:pasdeHk}
Assume that $\ccM$ is any regular holonomic $\cD_{X\times\PP^1}$-module. Then the push-forward $p_+(\ccM\otimes\ccE^t)$ has holonomic cohomology and satisfies $\cH^kp_+(\ccM\otimes\ccE^t)=0$ for $k\neq0$.
\end{lemme}

\begin{proof}
The first statement follows from the holonomicity of $\ccM\otimes\ccE^t$. We can assume that $\ccM=\ccM(*\infty)$. Let us set $M=p_*\ccM$. Then $M$ is a regular holonomic $\cD_X[t]\langle\partial_t\rangle$-module, and $p_+(\ccM\otimes\ccE^t)$ is the complex
\[
0\to M\To{\partial_t+1}\underset\bbullet M\to0,
\]
where the $\bbullet$ indicates the term in degree zero. Set $K=\cH^{-1}p_+(\ccM\otimes\ccE^t)=\ker(\partial_t+\nobreak1)$. It is $\cD_X$-holonomic, and the $\cD_X$-linear inclusion $K\hto M$ extends as a natural $\cD_X[t]\langle\partial_t\rangle$-linear morphism $K[t]\otimes E^{-t}\to M$. It is clear that $K[t]\otimes E^{-t}$ is purely irregular along $t=\infty$ (this is easily seen on the generic part of the support of $K$) hence, since $M$ is regular, this image is zero, so $K=0$.
\end{proof}

\begin{definition}
We say that a holonomic $\cD_X$-module $\ccN_{\exp}$ is \emph{exponentially regular} if there exists a regular holonomic $\cD_{X\times\PP^1}$-module $\ccM$ such that $\ccN_{\exp}\simeq \cH^0p_+(\ccM\otimes\ccE^t)$.
\end{definition}

\begin{proposition}\label{prop:Nexp}\mbox{}
\begin{enumerate}
\item\label{prop:Nexp1}
If $f$ is meromorphic on $X$ and holomorphic on $U=X\setminus D$, and if $\ccN=\ccN(*D)$ is a regular holonomic $\cD_X$-module, then $\ccN\otimes\ccE^f$ is exponentially regular.
\item\label{prop:Nexp2}
Let $\pi:X\to Y$ be a proper morphism and let $\ccN_{\exp}$ be exponentially regular on~$X$. Then for each $j$, $\cH^j\pi_+\ccN_{\exp}$ is exponentially regular on $Y$.
\end{enumerate}
\end{proposition}

\begin{proof}
The first point follows from Lemma \ref{lem:MNreg} and Lemma \ref{lem:if+}. For the second point, set $\ccN_{\exp}=\cH^0p_+(\ccM\otimes\ccE^t)$ with $\ccM$ regular on $X\times\PP^1$. We have, according to Lemma~\ref{lem:pasdeHk},
\begin{align*}
\cH^j\pi_+\ccN_{\exp}=\cH^j\pi_+(\cH^0p_{X,+}(\ccM\otimes\ccE^t))&=\cH^j(\pi_+p_{X,+}(\ccM\otimes\ccE^t))\\
&=\cH^j(p_{Y,+}(\pi\times\id)_+(\ccM\otimes\ccE^t)).
\end{align*}
Now, $(\pi\times\id)_+(\ccM\otimes\ccE^t)=(\pi_+\ccM)\otimes\ccE^t$, with $\pi_+\ccM$ having regular holonomic cohomology. We thus have $\cH^kp_{Y,+}\cH^j(\pi\times\id)_+(\ccM\otimes\ccE^t)=0$ for $k\neq0$ according to Lemma~\ref{lem:pasdeHk}, hence
\begin{align*}
\cH^j\big(p_{Y,+}(\pi\times\id)_+(\ccM\otimes\ccE^t)\big)&=\cH^0p_{Y,+}\cH^j(\pi\times\id)_+(\ccM\otimes\ccE^t)\\
&=\cH^0p_{Y,+}\big((\cH^j\pi_+\ccM)\otimes\ccE^t\big).\qedhere
\end{align*}
\end{proof}

\section{The mixed twistor \texorpdfstring{$\cD$}{D}-module attached to \texorpdfstring{$\ccE^f$}{E}}\label{sec:MTEf}

If $f$ is a rational function on $X$ with pole divisor $P$, the twist of a holonomic $\cD_X$\nobreakdash-module by $\ccE^f$ consists first in localizing this module along $P_\red$ and then in adding~$\rd f$ to its connection. The main property used is that the localization functor on holonomic $\cD_X$-modules preserves coherence (hence holonomy).

For a filtered holonomic $\cD_X$-module, the stupid localization functor $(*P_\red)$, which consists in localizing both the module and its filtration, does not preserve coherence since the localization of a coherent $\cO_X$-module does not remain $\cO_X$-coherent. In the theory of mixed Hodge modules, there is a localization functor which extends the one at the level of regular holonomic $\cD$-modules. We will now consider the case of $\cR_\cX$-modules and mixed twistor $\cD$-modules, in order to treat the Laplace transform of mixed Hodge modules.

We keep the analytic setting of \S\ref{subsec:graph}. Recall the following notation used in the theory of twistor $\cD$-modules (\cf \cite{Bibi01c,Mochizuki07,Mochizuki08}). For a complex manifold~$X$, we denote by $\cX$ the product $X\times\CC_\hb$ of $X$ with the complex line having coordinate~$\hb$. The ring $\cR_\cX$ is the analytification of the Rees ring $R_F\cD_X:=\bigoplus_{k\in\NN}F_k\cD_X\hb^k$ attached to the ring of differential operators equipped with its standard filtration by the order. It is locally expressed as $\cO_{\cX}\langle\partiall_{x_1},\dots,\partiall_{x_n}\rangle$, where $\partiall_{x_i}:=\hb\partial_{x_i}$.

\subsubsection*{The smooth case}
We denote by $\cE_U^{f/\hb}$ the $\cR_{\cU}$-module $\cO_\cU$ equipped with the $\hb$\nobreakdash-connection $\hb\rd+\rd f$. By using the same argument as in \cite[\S2.2]{Bibi04}, one checks that $\cE_U^{f/\hb}$ underlies a smooth twistor $\cD$-module; equivalently, it corresponds to a harmonic metric on the flat bundle $(\cO_U,\rd+\rd f)$. It follows that $\cE_U^{f/\hb}$ underlies a polarized variation of smooth twistor structure of weight~$0$, equivalently a pure polarized smooth twistor $\cD$-module.

\subsubsection*{The stupid localization}
Similarly, writing for short $\cO_\cX(*P_\red):=\cO_\cX\big({*}(P_\red\times\CC_\hb)\big)$, we consider $\cO_\cX(*P_\red)\cdot\mathrm{e}^{f/\hb}:=(\cO_\cX(*P_\red),\hb\rd+\rd f)$, where we denote the global section~$1$ of $\cO_\cX(*P_\red)$ by~$\mathrm{e}^{f/\hb}$. This is a coherent $\cR_\cX(*P_\red)$-module (however, it~is not necessarily $\cR_\cX$-coherent). Note also that there is a natural action of $\hb^2\partial_\hb$, by setting $\hb^2\partial_\hb(\mathrm{e}^{f/\hb}):=-f\cdot\mathrm{e}^{f/\hb}$ in $\cO_\cX(*P_\red)$. This action commutes with that of the $\hb$-connection. We say that $(\cO_\cX(*P_\red),\hb\rd+\rd f)$ is \emph{integrable} (\cf \cite[Chap.\,7]{Bibi01c}).

\begin{lemme}\label{lem:Efzcoh}
Assume that $f:U\to\Afu$ extends as a holomorphic map $f:X\to\PP^1$. Then $\cO_\cX(*P_\red)\cdot\mathrm{e}^{f/\hb}$ is $\cR_\cX$-coherent.
\end{lemme}

\begin{proof}
The question is local near a point of $P_\red$ and, up to shrinking $X$, we may assume that $f=1/g$ for some holomorphic function $g:X\to\CC$. Then $\cO_\cX(*P_\red)\cdot\nobreak\mathrm{e}^{f/\hb}=\cO_\cX(*\{g=0\})\mathrm{e}^{1/g\hb}$. If $P_\red$ has normal crossings, we choose local coordinates such that $g=x^\bme$, and the relation $x^\bme\partiall_{x_i}\mathrm{e}^{1/x^{\bme}\hb}=(-e_i/x_i)\mathrm{e}^{1/x^{\bme}\hb}$ gives the coherence. If $P_\red$ is arbitrary, let $\pi:X'\to X$ be a projective modification over a neighbourhood of the point of $P_\red$ we consider, such that $\pi^{-1}(P_\red)$ has normal crossings. Set $g'=g\circ\pi$. Then
\[
\pi_+\big(\cO_{\cX'}(*\{g'=0\})\mathrm{e}^{1/g'\hb}\big)=\cH^0\pi_+\big(\cO_{\cX'}(*\{g'=0\})\mathrm{e}^{1/g'\hb}\big)=\cO_\cX(*\{g=0\})\mathrm{e}^{1/g\hb},
\]
since it can be seen that $g$ is invertible on $\cH^0\pi_+\big(\cO_{\cX'}(*\{g'=0\})\mathrm{e}^{1/g'\hb}\big)$. By the properness of $\pi$, $\cO_\cX(*\{g=0\})\mathrm{e}^{1/g\hb}$ is then $\cR_\cX$-coherent.
\end{proof}

\begin{proposition}\label{prop:Epurewild}
If $g:X\to\CC$ is holomorphic, $\cO_\cX(*\{g=0\})\mathrm{e}^{1/g\hb}$ underlies a pure wild twistor $\cD$-module of weight zero.
\end{proposition}

As a consequence, the same property holds for $\cO_\cX(*P_\red)\cdot\mathrm{e}^{f/\hb}$ if $f:U\to\Afu$ extends as $f:X\to\PP^1$.

\begin{proof}[Proof of Proposition \ref{prop:Epurewild}]
This is essentially obvious from the theory of T.\,Mochizuki \cite{Mochizuki08}, but we will make the argument precise. Firstly, one can reduce to the case where $g=0$ has normal crossings, since pure wild twistor $\cD$-module of weight zero are stable by $\cH^0\pi_+$, if $\pi$ is a projective morphism. Here we take $\pi$ as in the proof of Lemma \ref{lem:Efzcoh}.

Set now $U=\{g\neq0\}\subset X$. Let $(\cC^\infty_U,\ov\partial,h)$ be the trivial bundle with its standard holomorphic structure, equipped with its standard metric for which $h(1,1)=1$. Consider it as a harmonic Higgs bundle on $U$ with holomorphic Higgs field $\theta=\rd(1/g)$. Since $g$ is a monomial (in local coordinates), this produces a non-ramified good wild harmonic bundle on $X$, in the sense of \cite[Def.\,7.1.7]{Mochizuki08}.

For a fixed $\hb$ (denoted by $\lambda$ in \loccit), denote by $E^\hb$ the holomorphic bundle $(\cC^\infty_U,\ov\partial+\hb\rd(1/\ov g))$. The extension $\ccP E^\hb$ defined in \cite[Not.\,7.4.1]{Mochizuki08} is nothing but $\cO_X\cdot\exp(\ov\hb/g-\hb/\ov g)$. Together with its natural connection, it is isomorphic to $\ccE^{(1+|\hb|^2)/g\hb}$ (\cf Example 7.4.1.2 in \loccit). Since there is no Stokes phenomenon in rank one, the construction $\ccQ E^\hb$ of \S11.1 in \loccit consists only in dividing the irregular value by $1+|\hb|^2$, so $\ccQ E^\hb\simeq\ccE^{1/g\hb}$ (first point of Th.\,11.1.2 in \loccit). Now, $\ccE^{1/g\hb}$ is the canonical prolongation of $(\cC^\infty_U,\ov\partial,\rd(1/g),h)$ as a coherent $\cR_\cX$-module. It is also equal to the $\cR_\cX$-module denoted by $\mathfrak{E}$ in \loccit\ (see \S12.3.2). Then one concludes by using Prop.\,19.2.1 of \loccit
\end{proof}

\subsubsection*{The twistor localization}
Let $H$ be a divisor in~$X$, locally defined by a holomorphic function $h$ and let $\cN$ be a coherent $\cR_\cX(*H)$-module. According to \cite[Def.\,3.3.1]{Mochizuki11}, one says that $\cN$ is twistor-specializable along $H$ if there exists a coherent $\cR_\cX$-submodule $\cN[*H]\subset\cN$ such that, considering locally the graph inclusion $i_h:X\hto Y:=X\times\CC$ with the coordinate $t$ on $\CC$,
\begin{itemize}
\item
the coherent $\cR_\cY(*\{t=0\})$-module $i_{h,+}\cN$ is strictly specializable along $t=0$, in the sense of \cite[\S3.4.a]{Bibi01c},
\item
$i_{h,+}(\cN[*H])$ is equal to the coherent $\cR_\cY$-submodule of $i_{h,+}\cN$ generated by the~$V^t_1$ term of the $V$-filtration (with the convention taken in this article), denoted by $(i_{h,+}\cN)[*t]$.
\end{itemize}
If $\cN[*H]$ exists locally, it is unique, hence exists globally. The category $\MTM^\mathrm{int}$ is introduced in \S7.2 of \cite{Mochizuki11}, and the results of \loccit imply the following.

\begin{proposition}\label{prop:EXf}
Let $f$ be any meromorphic function on $X$ with pole divisor $P$.
\begin{enumerate}
\item\label{prop:EXf1}
The coherent $\cR_\cX(*P_\red)$-module $\cO_\cX(*P_\red)\cdot\mathrm{e}^{f/\hb}$ is twistor-specializable along~$P_\red$ and defines $\cO_\cX(*P_\red)\cdot\mathrm{e}^{f/\hb}[*P_\red]=:\cE_X^{f/\hb}$.
\item\label{prop:EXf3}
Moreover, $\cE_X^{f/\hb}$ underlies an object of $\MTM^\mathrm{int}(X)$ extending the object of $\MTM^\mathrm{int}(U)$ that~$\cE_U^{f/\hb}$ underlies.
\item\label{prop:EXf4}
If $f$ extends as a morphism $f:X\to\PP^1$, then $\cE_X^{f/\hb}=\cO_\cX(*P_\red)\cdot\mathrm{e}^{f/\hb}$ and the object of $\MTM^\mathrm{int}(X)$ it underlies is pure of weight zero.
\item\label{prop:EXf5}
Let $H$ be any divisor in $X$. Then $\cE_X^{f/\hb}(*H)$ is twistor-specializable along $H$ and the corresponding object $\cE_X^{f/\hb}[*H]$ underlies an object of $\MTM^\mathrm{int}(X)$.
\end{enumerate}
\end{proposition}

\begin{proof}
Let us start by \eqref{prop:EXf4}. Let $g$ be a local equation of $P_\red$. Then
\[
i_{g,+}\big(\cO_\cX(*P_\red)\cdot\mathrm{e}^{1/g\hb}\big)=\big(i_{g,+}\cO_\cX(*P_\red)\big)\otimes\mathrm{e}^{1/t'\hb}
\]
and \cite[Prop.\,2.2.5]{Bibi06b} shows that the $V^{t'}$-filtration is constant. Therefore,
\[
i_{g,+}\big(\cO_\cX(*P_\red)\cdot\mathrm{e}^{1/g\hb}\big)[*t']=i_{g,+}\big(\cO_\cX(*P_\red)\cdot\mathrm{e}^{1/g\hb}\big)
\]
and thus $\cO_\cX(*P_\red)\cdot\mathrm{e}^{f/\hb}[*P_\red]=\cO_\cX(*P_\red)\cdot\mathrm{e}^{f/\hb}$, as wanted. The remaining assertion in \eqref{prop:EXf4} is then given by Proposition \ref{prop:Epurewild}.

Let us prove \eqref{prop:EXf1} and \eqref{prop:EXf3}. If $f$ does not extend as a morphism $X\to\PP^1$, let $\pi:X'\to X$ be as in \eqref{eq:diagram}. Set $D'=P'_\red\cup H'$. Then, according to \cite[Prop.\,11.2.1]{Mochizuki11}, $\cE_{X'}^{f'/\hb}[*H']$ underlies an object of $\MTM^\mathrm{int}(X')$. According to \cite[Prop.\,11.2.6]{Mochizuki11}, its push-forward $\cH^0\pi_+\cE_{X'}^{f'/\hb}[*H']$ underlies an object of $\MTM^\mathrm{int}(X)$. We also have $\cE_{X'}^{f'/\hb}[*H']=(\cO_{\cX'}(*D')\cdot\mathrm{e}^{f'/\hb})[*D']$ and we can apply \cite[Lem.\,3.3.17]{Mochizuki11} (because we work with objects of $\MTM(X')$) to deduce that
\[
\cH^0\pi_+\cE_{X'}^{f'/\hb}[*H']=\cH^0\pi_+(\cO_{\cX'}(*D')\cdot\mathrm{e}^{f'/\hb})[*P_\red].
\]
On the other hand, we have $\cH^0\pi_+(\cO_{\cX'}(*D')\cdot\mathrm{e}^{f'/\hb})=\cO_\cX(*P_\red)\cdot\mathrm{e}^{f/\hb}$. Therefore the latter $\cR_\cX(*P_\red)$-module is twistor-specializable along $P_\red$ and we have $\cE_X^{f/\hb}=\cH^0\pi_+\cE_{X'}^{f'/\hb}[*H']$. This concludes~\eqref{prop:EXf1} and \eqref{prop:EXf3}.

Lastly, \eqref{prop:EXf5} follows from \cite[Prop.\,11.2.1]{Mochizuki11}.
\end{proof}

\subsubsection*{The Laplace twist}
Let $f:U\to\Afu$ be as above and let $\tau$ be a new variable. We now consider the function $\tau f:U\times\Afu_\tau\to\Afu$ as a meromorphic function on~$X\times\Afu_\tau$. Proposition \ref{prop:EXf} implies that $\cE_{X\times\Afu_\tau}^{\tau f/\hb}$ exists and underlies an object of \hbox{$\MTM^\mathrm{int}(X\times\Afu_\tau)$}. In \S\ref{sec:relwithFirr} we will also have to consider another variable $v$ and the object~$\cE_X^{\tau v f/\hb}$ of \hbox{$\MTM^\mathrm{int}(X\times\Afu_v\times\Afu_\tau)$}.

\begin{proposition}\label{prop:Etaufhb}
If $f:U\to\Afu$ extends as a morphism $f:X\to\PP^1$, then $\cE_{X\times\Afu_\tau}^{\tau f/\hb}=\cO_{\cX\times\Afu_\tau}(*P_\red)\cdot\mathrm{e}^{\tau f/\hb}$.
\end{proposition}

\begin{proof}
The question is local near $P_\red$ and, using the notation as above, we have to prove that $\cE_{X\times\Afu_\tau}^{\tau /g\hb}=\cO_{\cX\times\Afu_\tau}(*P_\red)\cdot\mathrm{e}^{\tau /g\hb}$. Equivalently, we should prove that the $V^{t'}$-filtration of $(i_{g,+}\cO_{\cX\times\Afu_\tau})(*\{t'=0\})\cdot\mathrm{e}^{\tau /t'\hb}$ is constant. This is obtained through the equation $\delta(t'-g)\otimes\mathrm{e}^{\tau /t'\hb}=t'\partiall_\tau\delta(t'-g)\otimes\mathrm{e}^{\tau /t'\hb}$.
\end{proof}

\section{Strictness for exponentially twisted regular holonomic \texorpdfstring{$\cD$}{D}-modules}
We will first prove a particular case of Theorem \ref{th:1main}. Let $p:X\times\PP^1\to X$ denote the projection and let $t$ be the coordinate on the affine line $\Afu=\PP^1\setminus\{\infty\}$. Recall that, for $(\ccM,F_\bbullet\ccM)$ underlying a mixed Hodge module on $X\times\PP^1$, we have constructed in \cite[\S3.1]{E-S-Y13} a filtration $F_\bbullet^\Del(\ccM\otimes\ccE^t)$ indexed by $\QQ$ (\cf Definition \ref{def:Rees} for the corresponding Rees construction).

\begin{theoreme}\label{th:Firr}
For $(\ccM,F_\bbullet\ccM)$ underlying a mixed Hodge module, the complex $p_+R_{F_\bbullet^\Del}(\ccM\otimes\ccE^t)$ is strict and has nonzero cohomology in degree zero at most.
\end{theoreme}

In the case where $X$ is a point, this is the statement of \cite[Th.\,6.1]{Bibi08}. If $(\ccM,F_\bbullet\ccM)=i_{f,+}(\ccN,F_\bbullet\ccN)$ for some morphism $f:X\to\PP^1$ and some $(\ccN,F_\bbullet\ccN)$ underlying a mixed Hodge module on $X$, one can adapt the proof given in \cite[Prop.\,1.6.9]{E-S-Y13} for $\ccN=\cO_X(*D)$, where $D$ is a normal crossing divisor, and $f^{-1}(\infty)\subset D$, but this case is not enough for our purposes. The proof that we give below uses the full strength of the theory of mixed twistor $\cD$-modules of T.\,Mochizuki \cite{Mochizuki11}.

\begin{proof}[Proof of Theorem \ref{th:Firr}]
We first note that the second assertion in the theorem (\ie the vanishing of $\cH^j$ for $j\neq0$) follows from the strictness assertion together with Lemma~\ref{lem:pasdeHk}. So let us consider the strictness assertion.

We refer to \cite[\S\S2 \& 3]{E-S-Y13} for the notation and results we use here. Given the filtered $\cD_{X\times\PP^1}$-module $(\ccM,F_\bbullet\ccM)$ underlying a mixed Hodge module, we associate to it the Rees module $\cM:=R_F\ccM=\bigoplus_pF_p\ccM\cdot\hb^p$, which is a graded $R_F\cD_{X\times\PP^1}$-module. Its analytification $\cM^\an$ (with respect to the $\hb$-variable) is part of the data defining an integrable mixed twistor $\cD_{X\times\PP^1}$-module, according to \cite[Prop.\,12.5.4]{Mochizuki11}.

Let us consider the graded $R_F\cD_{X\times\PP^1}$-module $R_{F^\Del}(\ccM\otimes\ccE^t)$. Our aim is to prove the strictness (\ie the absence of $\hb$-torsion) of the push-forward modules $\cH^jp_+R_{F^\Del}(\ccM\otimes\ccE^t)$. Forgetting the grading, $R_{F^\Del}(\ccM\otimes\ccE^t)$ can be obtained by using an explicit expression of the $V$-filtration as in \cite[Prop.\,3.1.2]{E-S-Y13}. It is enough to check the strictness property on the corresponding analytic object, by flatness. Now, the analytification $\big(R_{F^\Del}(\ccM\otimes\ccE^t)\big)^\an$ can be obtained by using the analytic $V$\nobreakdash-filtration, by making analytic the formula of \cite[Prop.\,3.1.2]{E-S-Y13}. We then use that the $V$-filtration behaves well by push-forward for mixed twistor $\cD$-modules, according to results of \cite{Mochizuki11}. This is the main argument for proving Theorem~\ref{th:Firr}.

Let us denote by $\wt\cM$ the (stupidly) localized module $\cM(*\infty)$ and by $\cFcM$ the (not graded) $(R_F\cD_{X\times\PP^1})[\tau]\langle\partiall_\tau\rangle$-module $\wt\cM[\tau]\otimes\cE^{t\tau/\hb}$. By Proposition \ref{prop:Etaufhb}, this is also $\cM[\tau]\otimes\cE^{t\tau/\hb}$. Similarly, $(\cFcM)^\an$ denotes its analytification with respect to both~$\tau$ and~$\hb$. We can use Proposition \ref{prop:EXf} together with \cite[Prop.\,11.3.4]{Mochizuki11} to ensure that $(\cFcM)^\an$ underlies an integrable mixed twistor $\cD$-module.

Let $p:X\times\PP^1\times\Afu_\tau\to X\times\Afu_\tau$ denote the projection. Then $p_+\cFcM^\an$ is strict, each $\cH^jp_+\cFcM^\an$ is strictly specializable along $\tau=0$ and the $V^\tau$-filtration satisfies $V_\bbullet^\tau\cH^jp_+\cFcM^\an=\cH^jp_+(V_\bbullet^\tau\cFcM^\an)$. Indeed, these properties are satisfied according to the main results of \cite{Mochizuki11}.

We will now adapt the proof given in \cite[\S3.2]{E-S-Y13}, which needs a supplementary argument, since we cannot argue with (3.2.2) in \loccit

According to \cite[Prop.\,3.1.2]{E-S-Y13}, we have a long exact sequence
\[
\cdots\to\cH^jp_+V_\alpha^\tau\cFcM^\an\To{\tau-\hb}\cH^jp_+V_\alpha^\tau\cFcM^\an\to\cH^jp_+(R_{F^\Del}(\ccM\otimes\ccE^t))^\an\to\cdots
\]
that we can thus rewrite as
\begin{multline}\label{eq:exseq}
\cdots\to V_\alpha^\tau\cH^jp_+\cFcM^\an\To{\tau-\hb}V_\alpha^\tau\cH^jp_+\cFcM^\an\\
\to\cH^jp_+(R_{F^\Del}(\ccM\otimes\ccE^t))^\an\to\cdots
\end{multline}

Let us first check that $\tau-\hb$ is injective on each $\cH^jp_+\cFcM^\an$. In the case considered in \cite[\S3.2]{E-S-Y13}, we could use (3.2.2) of \loccit, and when $X$ is reduced to a point, the argument in \cite{Bibi08} uses the solution to a Birkhoff problem given by M.\,Saito. We do not know how to extend the argument of \cite{Bibi08} to the case $\dim X\geq1$.

\begin{lemme}\label{lem:injective}
Let $Y$ be a complex manifold, $\cN^\an$ an $\cR_\cY$-module which underlies a mixed twistor $\cD$-module in the sense of \cite{Mochizuki11}. Let~$h$ be a holomorphic function on $Y$. Then the action of $h-\hb$ is injective on $\cN^\an$.
\end{lemme}

\begin{proof}
Since a mixed twistor $\cD$-module is in particular an object of the category $\mathrm{MTW}(Y)$ (\cf \cite[\S\S7.1.1\,\&\,7.2.1]{Mochizuki11}), a simple extension argument with respect to the weight filtration allows us to reduce to the case where $\cN^\an$ underlies a pure wild twistor $\cD$-module (as defined in \cite{Mochizuki08}). Since the question is local on $Y$, we fix some $y_o\in Y$ and work locally near $y_o$.

Assume first that $\cN^\an$ underlies a smooth pure twistor $\cD$-module. Then it is a locally free $\cO_\cY$-module with $\hb$-connection, and the injectivity of $h-\hb$ is clear.

In general, we know that $\cN^\an$ has a decomposition by the strict support (\cf \cite[\S3.5]{Bibi01c}, \cite[\S22.3.4]{Mochizuki08}, \cite[\S1.4]{Bibi06b}) and we can therefore assume that, near $y_o$, $\cN^\an$ has strict support a germ of an irreducible closed analytic subset $Z\subset Y$ at $y_o$. On a dense open set $Z^o$ of the smooth part of $Z$, due by Kashiwara's equivalence for pure twistor $\cD$-modules (see \loccit), we are reduced to the smooth case considered above and the injectivity holds. Therefore, $\ker[(h-\hb):\cN^\an\to\cN^\an]$ is supported on a proper closed analytic subset~$Z'$ of $Z$ in the neighbourhood of~$y_o$. Let $F_\bbullet\cN^\an$ be a good filtration of~$\cN^\an$ as an $(R_F\cD_Y)^\an$-module (which exists since we work locally on $Y$). Then for each $k$, $\ker[(h-\hb):F_k\cN^\an\to F_k\cN^\an]$ is a coherent $\cO_{Y\times\CC_\hb}$-submodule of $\cN^\an$ supported on $Z'$. The $(R_F\cD_Y)^\an$-submodule that it generates is a coherent $(R_F\cD_Y)^\an$-submodule of $\cN^\an$ supported on $Z'$. It is therefore zero since~$\cN^\an$ has strict support equal to $Z$. Since this holds for any $k$, we conclude that $\ker[(h-\hb):\cN^\an\to\cN^\an]=0$.
\end{proof}

Since $\cH^jp_+\cFcM^\an$ underlies a mixed twistor $\cD$-module, we infer from Lemma \ref{lem:injective} that $\tau-\hb$ is injective on each $\cH^jp_+\cFcM^\an$. We conclude that the long exact \hbox{sequence}~\eqref{eq:exseq} splits into short exact sequences and therefore $\cH^jp_+(R_{F^\Del}\ccM)^\an$ is identified with $V_\alpha^\tau\cH^jp_+\cFcM^\an/(\tau-\hb)V_\alpha^\tau\cH^jp_+\cFcM^\an$ for each~$j$. Proving that the later module is strict is a local question, near points with coordinates $(\tau,\hb)$ in the neighbourhood of $(\tau_o,\hb_o)$ with $\tau_o=\hb_o$.
\begin{enumerate}
\item
If $\tau_o=\hb_o=0$, we use that $V_\alpha^\tau\cH^jp_+\cFcM^\an/\tau V_\alpha^\tau\cH^jp_+\cFcM^\an$ is strict, due to the strict specializability of $\cH^jp_+\cFcM^\an$ along $\tau=0$ and it is enough to prove that~$\hb$ is injective on $V_\alpha^\tau\cH^jp_+\cFcM^\an/(\tau-\hb)V_\alpha^\tau\cH^jp_+\cFcM^\an$. Due to the strictness above, if a local section $m$ of $V_\alpha^\tau\cH^jp_+\cFcM^\an$ satisfies $\hb m=\tau m'$ for some local section $m'$ of $V_\alpha^\tau\cH^jp_+\cFcM^\an$, then there exists a local section $m''$ of $V_\alpha^\tau\cH^jp_+\cFcM^\an$ such that $m=\tau m''$, and since~$\tau$ is injective on $V_\alpha^\tau\cH^jp_+\cFcM^\an$ for $\alpha\in[0,1)$, we have $m'=\hb m''$. As a consequence, if a local section $m$ of $V_\alpha^\tau\cH^jp_+\cFcM^\an$ satisfies $\hb m=(\tau-\hb) m_1$ for some local section~$m_1$ of $V_\alpha^\tau\cH^jp_+\cFcM^\an$, then there exists a local section $m''$ of $V_\alpha^\tau\cH^jp_+\cFcM^\an$ such that $m+m_1=\tau m''$ and $m_1=\hb m''$, hence $m=(\tau-\hb)m''$, which gives the desired injectivity.

\item
We now assume that $\tau_o=\hb_o\neq0$. Near such a point, we have $V_\alpha^\tau\cH^jp_+\cFcM^\an=\cH^jp_+\cFcM^\an$. Let us remark, however, that the $V$-filtration of $\cFcM^\an$ along $\tau-\tau_o=0$ satisfies $V_k^{(\tau-\tau_o)}\cFcM^\an=\cFcM^\an$ for $k\geq0$ and $V_k^{(\tau-\tau_o)}\cFcM^\an=(\tau-\tau_o)^{-k}\cFcM^\an$ for $k\leq0$, according to \cite[Prop.\,4.1(iii)]{Bibi05b}. Applying the push-forward argument as above, we conclude that the $V$-filtration of $\cH^jp_+\cFcM^\an$ along $\tau-\tau_o=0$ satisfies the same property. Therefore, setting $\tau'=\tau-\tau_o$ and $\hb'=\hb-\hb_o$, we are reduced to proving the injectivity of $\hb'$ on $V_\alpha^{\tau'}\cH^jp_+\cFcM^\an/(\tau'-\hb') V_\alpha^{\tau'}\cH^jp_+\cFcM^\an$. We can then use the same argument as we used for the case $\tau_o=0$.
\qedhere
\end{enumerate}
\end{proof}

\section{The irregular Hodge filtration}\label{sec:irregHodgefiltration}
In this section, we come back to the setup of Theorem \ref{th:1main}. Let $f$ be a meromorphic function on~$X$ with pole divisor $P$ and let $(\ccN,F_\bbullet\ccN)$ be a filtered $\cD_X$-module underlying a mixed Hodge module such that $\ccN=\ccN(*P_\red)$. Let $(\ccM,F_\bbullet\ccM)$ be the mixed Hodge module on $X\times\PP^1$ associated to $(\ccN,F_\bbullet\ccN)$ by the construction of Remark~\ref{rem:graphMHM}. We know by Theorem \ref{th:Firr} that the complex $p_{X,+}R_{F^\Del_\bbullet}(\ccM\otimes\ccE^t)$ is strict and has cohomology in degree zero at most, hence $\cH^0p_{X,+}R_{F^\Del_\bbullet}(\ccM\otimes\ccE^t)$ is equal to the Rees module of $\ccN\otimes\ccE^f$ with respect to some good filtration, which we precisely define as $F^\irr_\bbullet(\ccN\otimes\ccE^f)$.

\begin{definition}\label{def:Firr}
The filtration $F_\bbullet^\irr(\ccN\otimes\ccE^f)$ is the filtration obtained by push-forward from $F_\bbullet^\Del(\ccM\otimes\ccE^t)$.
\end{definition}

\Subsection{Proof of Theorem \ref{th:1main}}\label{subsec:proofthmain}
\begin{enumerate}
\item
This is clear since it already holds for $F_\bbullet^\Del(\ccM\otimes\ccE^t)$.
\item
Because the category of mixed Hodge modules is abelian, we have an exact sequence of filtered $\cD$-modules underlying mixed Hodge modules:
\[
0\to(\ccN_0,F_\bbullet\ccN_0)\to(\ccN_1,F_\bbullet\ccN_1)\To{\varphi}(\ccN_2,F_\bbullet\ccN_2)\to(\ccN_3,F_\bbullet\ccN_3)\to0
\]
which gives rise to an exact sequence of filtered $\cD$-modules underlying mixed Hodge modules:
\[
0\to(\ccM_0,F_\bbullet\ccM_0)\to(\ccM_1,F_\bbullet\ccM_1)\To{\varphi}(\ccM_2,F_\bbullet\ccM_2)\to(\ccM_3,F_\bbullet\ccM_3)\to0
\]
and therefore, according to \cite[Th.\,3.0.1(2)]{E-S-Y13}, to an exact sequence of filtered $\cD$-modules:
\[
0\ra(\ccM_0\otimes\ccE^t,F^\Del_\bbullet)\ra(\ccM_1\otimes\ccE^t,F^\Del_\bbullet)\ra(\ccM_2\otimes\ccE^t,F^\Del_\bbullet)\ra(\ccM_3\otimes\ccE^t,F^\Del_\bbullet)\ra0.
\]
Applying $\cH^0p_+$ we keep an exact sequence, according to the second statement in Theorem \ref{th:Firr}.

\item
We consider the following diagram:
\[
\xymatrix@C=1.5cm{
X\times\PP^1\ar[d]_{p_X}\ar[r]^-{\pi\times\id_{\PP^1}}&Y\times\PP^1\ar[d]^{p_Y}\\
X\ar[r]^-\pi&Y
}
\]
We thus have
\[
\pi_+R_{F^\irr_\bbullet}(\ccN\otimes\ccE^f)\simeq(\pi\circ p_X)_+R_{F^\Del_\bbullet}(\ccM\otimes\ccE^t)\simeq(p_Y\circ(\pi\times\id))_+R_{F^\Del_\bbullet}(\ccM\otimes\ccE^t).
\]
On the other hand, according to \cite[Prop.\,3.2.3]{E-S-Y13}, $(\pi\times\id)_+R_{F^\Del_\bbullet}(\ccM\otimes\ccE^t)$ is strict and for each $j$,
\[
\cH^j(\pi\times\id)_+R_{F^\Del_\bbullet}(\ccM\otimes\ccE^t)\simeq R_{F^\Del_\bbullet}\big(\cH^j(\pi\times\id)_+\ccM\big)\otimes\ccE^t.
\]
Applying now Theorem \ref{th:Firr} to $\cH^j(\pi\times\id)_+(\ccM,F_\bbullet\ccM)$ we obtain the assertion.

\item
This point is similar to \cite[Prop.\,3.2.3]{E-S-Y13}.

\item
The case \ref{rem:ESY}\eqref{enum:ESY1} follows from \cite[Prop.\,1.6.12]{E-S-Y13}. Let us show the case \ref{rem:ESY}\eqref{enum:ESY2}. If $i:\PP^1_t\hto\PP^1_t\times\PP^1_s$ denotes the diagonal inclusion $t\mto(t,t)$ and \hbox{$p:\PP^1_t\times\PP^1_s\to\PP^1_t$} denotes the projection (and similarly after taking the product with~$X$), we have an isomorphism
\[
\ccM\otimes\ccE^t\simeq \cH^0p_+(i_+(\ccM\otimes\ccE^t))\simeq\cH^0p_+\big((i_+\ccM)\otimes\ccE^s\big).
\]
We claim that, for each $\alpha\in[0,1)$,
\begin{equation}\label{eq:Delirr*}
F_{\alpha+\bbullet}^\irr(\ccM\otimes\ccE^t)=F_{\alpha+\bbullet}^\Del(\ccM\otimes\ccE^t).
\end{equation}
It is enough to check
\begin{equation}\label{eq:Delirr**}
i_+\big(\ccM\otimes\ccE^t,F_{\alpha+\bbullet}^\Del(\ccM\otimes\ccE^t)\big)=\big((i_+\ccM)\otimes\ccE^s,F_{\alpha+\bbullet}^\Del((i_+\ccM)\otimes\ccE^s)\big),
\end{equation}
and the question is non obvious in the charts $t'=1/t$ and $s'=1/s$. Let us set $\delta=\delta(s'-t')$. Let us first recall that, by definition,
\begin{align}
i_+\ccM'&=\bigoplus_{k\geq0}i_*\ccM'\otimes\partial_{s'}^k\delta,\notag\\\label{eq:Fpi+}
F_p(i_+\ccM')&=\bigoplus_{k\geq0}i_*F_{p-k-1}\ccM'\otimes\partial_{s'}^k\delta,\\[-5pt]
\intertext{(the shift by one comes from the left-to-right transformation on $R_F\cD$-modules) and, concerning the $V$-filtration, one checks that\footnotemark}
V_\alpha^{s'}(i_+\ccM')&=\sum_{k\geq0}\partial_{t'}^k(V_\alpha^{t'}\ccM'\otimes\delta).\notag
\end{align}
\footnotetext{The formula in the published version is not correct. We thank Takahiro Saito for noticing the mistake}%
We set $G_p(i_+\ccM')=\bigoplus_{0\leq k\leq p}i_*\ccM'\otimes\partial_{s'}^k\delta$. Let $\sum_{k=0}^\ell \partial_{t'}^k(m_{\alpha,k}\otimes\delta)\in V_\alpha^{s'}(i_+\ccM')$. Then this term is contained in $G_\ell(i_+\ccM')$, and its image in $G_\ell(i_+\ccM')/G_{\ell-1}(i_+\ccM')$ is the class of $m_{\alpha,\ell}\otimes\partial_{s'}^\ell\delta$, hence it is nonzero if and only if $m_{\alpha,\ell}\neq0$. In particular, $V_\alpha^{s'}(i_+\ccM')\cap G_0(i_+\ccM')=V_\alpha^{t'}\ccM'\otimes\delta$.

We recall (\cf\cite[(3.1.1)]{E-S-Y13}):
\[
F_{\alpha+p}^\Del(\ccM'\otimes\ccE^{1/t'})=\sum_{k\geq0}\partial_{t'}^k\tpm\bigl[(F_{p-k}\ccM'\cap V_\alpha^{t'}\ccM')\otimes\ccE^{1/t'}\bigr],
\]
and, by the analogue of \eqref{eq:Fpi+},
\[
i_+\big(F_{\alpha+\bbullet}^\Del(\ccM'\otimes\ccE^{1/t'})\big)_p=F_{\alpha+p-1}^\Del(\ccM'\otimes\ccE^{1/t'})\otimes\delta+\partial_{s'}\big[i_+\big(F_{\alpha+\bbullet}^\Del(\ccM'\otimes\ccE^{1/t'})\big)_{p-1}\big].
\]
On the other hand,
\begin{align*}
F_{\alpha+p}^\Del\big((i_+\ccM')\otimes\ccE^{1/s'}\big)&=\sum_{k\geq0}\partial_{s'}^ks^{\prime-1}\bigl[(F_{p-k}(i_+\ccM')\cap V_\alpha^{s'}(i_+\ccM'))\otimes\ccE^{1/s'}\bigr]\\
&=s^{\prime-1}\big(F_p(i_+\ccM')\cap V_\alpha^{s'}(i_+\ccM')\big)\otimes\reso\\[-3pt]
&\hspace*{3.8cm}+\partial_{s'}F_{\alpha+p-1}^\Del\big((i_+\ccM')\otimes\ccE^{1/s'}\big).
\end{align*}
We prove \eqref{eq:Delirr**} by induction on $p$. Let $p_o$ be such that $F_{p_o-2}\ccM'=0$. Then $F_{p_o}(i_+\ccM')=F_{p_o-1}\ccM'\otimes\delta\subset G_0(i_+\ccM')$, $F_{\alpha+p_o-1}^\Del\big((i_+\ccM')\otimes\ccE^{1/s'}\big)=0$ and
\begin{align*}
F_{\alpha+p_o}^\Del\big((i_+\ccM')\otimes\ccE^{1/s'}\big) &=s^{\prime-1}\big(F_{p_o}(i_+\ccM')\cap V_\alpha^{s'}(i_+\ccM')\big)\otimes\reso\\
&=s^{\prime-1}\big(F_{p_o-1}\ccM'\cap V_\alpha^{t'}\ccM'\big)\otimes(\delta\otimes\reso)\\
&=\big(\tpm(F_{p_o-1}\ccM'\cap V_\alpha^{t'}\ccM')\otimes\retp\big)\otimes\delta\\
&=i_+\big(F_{\alpha+\bbullet}^\Del(\ccM'\otimes\ccE^{1/t'})\big)_{p_o}.
\end{align*}

We now assume that \eqref{eq:Delirr**} holds for $p-1$. Let us first show that
\[
i_+\big(F_{\alpha+\bbullet}^\Del(\ccM'\otimes\ccE^{1/t'})\big)_p\subset F_{\alpha+p}^\Del\big((i_+\ccM')\otimes\ccE^{1/s'}\big).
\]
By induction and the above formula for $i_+\big(F_{\alpha+\bbullet}^\Del(\ccM'\otimes\ccE^{1/t'})\big)_p$, it is enough to check
\[
i_*F_{\alpha+p-1}^\Del(\ccM'\otimes\ccE^{1/t'})\otimes\delta\subset F_{\alpha+p}^\Del\big((i_+\ccM')\otimes\ccE^{1/s'}\big).
\]
We have
\[
F_{\alpha+p-1}^\Del(\ccM'\otimes\ccE^{1/t'})=\tpm(F_{p-1}\ccM'\cap V_\alpha^{t'}\ccM')\otimes\retp+\partial_{t'}\big(F_{\alpha+p-2}^\Del(\ccM'\otimes\ccE^{1/t'})\big).
\]
Then on the one hand, by induction,
\begin{align*}
i_*\partial_{t'}\big(F_{\alpha+p-2}^\Del(\ccM'\otimes\ccE^{1/t'})&\big)\otimes\delta\subset\partial_{t'}\big[i_*\big(F_{\alpha+p-2}^\Del(\ccM'\otimes\ccE^{1/t'})\big)\otimes\delta\big]\\[-3pt]
&\hspace*{3.65cm}+\partial_{s'}\big[i_*\big(F_{\alpha+p-2}^\Del(\ccM'\otimes\ccE^{1/t'})\big)\otimes\delta\big]\\
&\subset\partial_{t'}F_{\alpha+p-1}^\Del\big((i_+\ccM')\otimes\ccE^{1/s'}\big)+\partial_{s'}F_{\alpha+p-1}^\Del\big((i_+\ccM')\otimes\ccE^{1/s'}\big)\\
&\subset F_{\alpha+p}^\Del\big((i_+\ccM')\otimes\ccE^{1/s'}\big).
\end{align*}
On the other hand, $i_*\big[\tpm(F_{p-1}\ccM'\cap V_\alpha^{t'}\ccM')\otimes\retp\big]\otimes\delta$ is the degree zero term (w.r.t.~$G_\bbullet$) in $F_{\alpha+p}^\Del\big((i_+\ccM')\otimes\ccE^{1/s'}\big)$.

Let us now prove the reverse inclusion
\[
F_{\alpha+p}^\Del\big((i_+\ccM')\otimes\ccE^{1/s'}\big)\subset i_+\big(F_{\alpha+\bbullet}^\Del(\ccM'\otimes\ccE^{1/t'})\big)_p.
\]
By induction and the above formula for $F_{\alpha+p}^\Del\big((i_+\ccM')\otimes\ccE^{1/s'}\big)$, it is enough to prove
\[\tag{$*$}
s^{\prime-1}\big(F_p(i_+\ccM')\cap V_\alpha^{s'}(i_+\ccM')\big)\otimes\reso\subset i_+\big(F_{\alpha+\bbullet}^\Del(\ccM'\otimes\ccE^{1/t'})\big)_p.
\]

Let $m=\sum_{j=0}^\ell\partial_{t'}^j(m_{\alpha,j}\otimes\delta)$ be in $F_p(i_+\ccM')\cap V_\alpha^{s'}(i_+\ccM')$. We wish to prove that it belongs to the right-hand side of $(*)$. Assume $m_{\alpha,\ell}\neq0$, so that $m_{\alpha,\ell}\in F_{p-1-\ell}\ccM'\cap V_\alpha^{t'}(\ccM')$. Then $m_{\alpha,\ell}\otimes\delta\in F_{p-\ell}(i_+\ccM')\cap V_\alpha^{s'}(i_+\ccM')$ and $\partial_{t'}^\ell(m_{\alpha,j}\otimes\delta)\in F_p(i_+\ccM')\cap V_\alpha^{s'}(i_+\ccM')$. It follows that each nonzero term in the sum expressing $m$ also belongs to $F_p(i_+\ccM')\cap V_\alpha^{s'}(i_+\ccM')$, and it is enough to prove the desired inclusion for $m=\partial_{t'}^\ell(m_{\alpha,\ell}\otimes\delta)$, that is,
\[\tag{$**$}
\partial_{t'}^\ell \tpm(m_{\alpha,\ell}\otimes\delta\otimes\reso)\in i_+\big(F_{\alpha+\bbullet}^\Del(\ccM'\otimes\ccE^{1/t'})\big)_p.
\]

The left-hand side of $(**)$ reads $\partial_{t'}^\ell \tpm\bigl[(m_{\alpha,\ell}\otimes\retp) \otimes\delta\bigr]$ and is a sum of terms
\[
\bigl[\partial_{t'}^{\ell-j}\tpm(m_{\alpha,\ell}\otimes\retp)\bigr] \otimes\partial_{s'}^j\delta,\quad j=0,\dots,\ell.
\]
We have
\[
\partial_{t'}^{\ell-j}\tpm(m_{\alpha,\ell}\otimes\retp)\in \partial_{t'}^{\ell-j}\tpm(F_{p-1-\ell}\ccM'\cap V_\alpha^{t'}\ccM')\otimes\retp\subset F_{\alpha+p-1-j}^\Del (\ccM'\otimes\ccE^{1/t'}),
\]
and thus
\[
\bigl[\partial_{t'}^{\ell-j}\tpm(m_{\alpha,\ell}\otimes\retp)\bigr] \otimes\partial_{s'}^j\delta\in F_{\alpha+p-1-j}^\Del (\ccM'\otimes\ccE^{1/t'})\otimes\partial_{s'}^j\delta\subset i_+\big(F_{\alpha+\bbullet}^\Del(\ccM'\otimes\ccE^{1/t'})\big)_p,
\]
as wanted.\qed
\end{enumerate}

\begin{remarque}
Let $f:X\to\PP^1$ be a morphism and let $(\ccN,F_\bbullet\ccN)$ underlie a mixed Hodge module. It follows from \ref{th:1main}\eqref{th:1main3} and \eqref{th:1main4} that $i_{f,+}(\ccN\otimes\ccE^f,F_\bbullet^\irr)=(\ccM\otimes\ccE^t,F^\Del_\bbullet)$, if we set as above $(\ccM,F_\bbullet\ccM)=i_{f,+}(\ccN,F_\bbullet\ccN)$.
\end{remarque}

\subsection{The irregular Hodge filtration in terms of \texorpdfstring{$V^\tau_\alpha$}{Va}}\label{subsec:irregV}
With the notation as in the beginning of this section, we consider the pull-back module $\ccN[\tau]$ by the projection $r:X\times\nobreak\Afu_\tau\to X$ and the corresponding Rees object $R_F\ccN[\tau]=:\cN[\tau]$, where $\cN:=R_F\ccN$. Denote by $\cN^\an$ the analytification of $\cN$ and by $r^+\cN^\an$ that of $\cN[\tau]$. We twist $r^+\cN^\an$ by~$\cE^{\tau f/\hb}$ to obtain the object $\cFfcN$, which underlies an object of $\MTM^\mathrm{int}(X\times\nobreak\Afu_\tau)$, according to Propositions \ref{prop:EXf} and \ref{prop:Etaufhb}, and to \cite[Prop.\,11.3.4 \& 12.5.4]{Mochizuki11}.

\begin{proposition}\label{prop:FirrValphatau}
We have $R_{F_\alpha^\irr}(\ccN\otimes\ccE^f)^\an=V^\tau_\alpha(\cFfcN)/(\tau-\hb)V^\tau_\alpha(\cFfcN)$.
\end{proposition}

\begin{proof}
We associate to the mixed Hodge module $(\ccN,F_\bbullet\ccN)$ the mixed Hodge module $(\ccM,F_\bbullet\ccM)$ as in Remark~\ref{rem:graphMHM} (from which we keep the notation). It follows from \eqref{eq:Delirr*} and \cite[Prop.\,3.1.2]{E-S-Y13} that the result holds for $(\ccM,F_\bbullet\ccM)$ on $X\times\PP^1$ and for $f$ equal to the projection to $\PP^1$ (note that it holds without taking~``$\an$''). Applying the same argument as in the proof of Theorem \ref{th:Firr}, we conclude that the operation $V_\alpha^\tau/(\tau-\hb)V_\alpha^\tau$ commutes with $\cH^0p_+$. On the other hand, by definition and according to \eqref{eq:Delirr*}, the operation $R_{F^\irr_\alpha}$ is compatible with $\cH^0p_+$. Therefore, the result holds for $(\ccN,F_\bbullet\ccN)$.
\end{proof}

\begin{remarque}\label{rem:FirrValphatau}
As a consequence, one can recover the graded module $R_{F_\alpha^\irr}(\ccN\otimes\ccE^f)$ from $V^\tau_\alpha(\cFfcN)$ in the following way. As an $\cO_X[\hb]$-module, we have an inclusion $R_{F_\alpha^\irr}(\ccN\otimes\ccE^f)\subset\ccN[\hb,\hb^{-1}]$ and $R_{F_\alpha^\irr}(\ccN\otimes\ccE^f)$ is obtained from $R_{F_\alpha^\irr}(\ccN\otimes\ccE^f)^\an$ as the graded module with respect to the filtration on $R_{F_\alpha^\irr}(\ccN\otimes\ccE^f)^\an$ induced by the $\hb$-adic filtration of $\cO_\cX\otimes_{\cO_X[\hb]}\ccN[\hb,\hb^{-1}]$. By the proposition above, it is thus enough to identify the inclusion as $\cO_\cX$-modules
\begin{starequation}\label{eq:FirrValphatau}
V^\tau_\alpha(\cFfcN)/(\tau-\hb)V^\tau_\alpha(\cFfcN)\subset\cO_\cX\otimes_{\cO_X[\hb]}\ccN[\hb,\hb^{-1}].
\end{starequation}%
By using the strict specializability of $\cFfcN$ along $\tau=0$, one checks as in \cite[Proof of Prop.\,3.1.2]{E-S-Y13} that $(\tau-\hb)V^\tau_\alpha(\cFfcN)=V^\tau_\alpha(\cFfcN)\cap(\tau-\hb)\cFfcN$, so that
\[
V^\tau_\alpha(\cFfcN)/(\tau-\hb)V^\tau_\alpha(\cFfcN)\subset\cFfcN/(\tau-\hb)\cFfcN.
\]
Let us set (recall that $\ccN=\ccN(*P_\red)$)
\begin{align*}
\wt\cN&=(R_F\ccN)(*P_\red)=\bigoplus_pF_p\ccN(*P_\red)\hb^p\subset\ccN[\hb,\hb^{-1}],\\
\wt\cN^\an&=\cO_\cX\otimes_{\cO_X[\hb]}\wt\cN\subset\cO_\cX\otimes_{\cO_X[\hb]}\ccN[\hb,\hb^{-1}].
\end{align*}
As an $\cO_{\cX\times\CC_\tau}$-module we have $\cFfcN=r^*\wt\cN^\an$ and thus as an $\cO_\cX$-module we have $\cFfcN/(\tau-\hb)\cFfcN=\wt\cN^\an$. This gives the desired inclusion \eqref{eq:FirrValphatau}.
\end{remarque}

\part{The case of a rescaled meromorphic function}\label{part:2}

\section{Kontsevich bundles via \texorpdfstring{$\cD$}{D}-modules}\label{sec:intropart2}
In Part \ref{part:2}, we use the setting and notation of \S\ref{subsec:rescaling}. It will also be convenient to work algebraically with respect to $\PP^1$, in which case we will consider the $\cDXv$-module $E^{vf}(*H):=\cO_X(*D)[v]\cdot\revf$ and the $\cDXu$-module $E^{f/u}(*H):=\cO_X(*D)[u,u^{-1}]\cdot\refu$, where $\revf$ and $\refu$ are other notations for $1$ which make clear the twist of the connection.

\subsection{The Laplace Gauss-Manin bundles $\cH^k(\alpha)$}
The bundles $\cH^k(\alpha)$ on $\PP^1_v$ will be obtained by gluing bundles on $\Afu_v$ and on $\Afu_u$, that we describe below.

\subsubsection*{Over the chart $\Afu_v$}
Let us denote by $\cH^k_v$ the restriction of $\cH^k$ (\cf \S\ref{subsec:varirrhodgefiltrkontsevichbundles}) to the $v$\nobreakdash-chart. This is nothing but the Laplace transform of the $(k-\dim X)$th Gauss-Manin system of~$f$. It is known to have a regular singularity at $v=0$ and no other singularity at finite distance (as follows by push-forward from the arguments recalled in \S\ref{subsec:VfiltEv}, or by a general result about Laplace transform of regular holonomic $\cD$-modules in one variable). Moreover, $\cH^k_v$ is equipped with the push-forward filtration $F_\bbullet^\irr\cH^k_v$ by $\cO_{\Afu_v}$\nobreakdash-coherent subsheaves, in a strict way according to Theorem \ref{th:1main}. On $\GGm$ we obtain in such a way a flat bundle $(\cH^k_{|\GGm},\nabla)$ equipped with a filtration $F_\bbullet^\irr\cH^k_{|\GGm}$ indexed by $\QQ$, which satisfies Griffiths transversality condition with respect to $\nabla$ (\cf\S\ref{subsec:deRhamYu}, \cf also Remark \ref{rem:HSMochizuki}). This is the variation with respect to $v$ of the irregular Hodge filtration of $H^k(X,\DR\ccE^{vf}(*H))$.

We consider the limiting filtration (in the sense of Schmid) when $v\to0$. For $\alpha\in[0,1)$, let us denote by $V_\alpha\cH_v^k$ the $\alpha$th term of the Kashiwara-Malgrange \hbox{filtration} of~$\cH_v^k$ at $v=0$. Equivalently, due to the regularity property of the connection at $v=\nobreak0$, $V_\alpha\cH_v^k$ is the Deligne extension of $\cH^k_{|\GGm}$ on which $\nabla$ has a simple pole with residue $\Res_{v=0}\nabla$ having eigenvalues in $[-\alpha,-\alpha+1)$. We set $\gr^V_\alpha\cH_v^k=V_\alpha\cH_v^k/V_{<\alpha}\cH_v^k$.

\begin{theoreme}\label{th:3main}
For each $\alpha\in[0,1)$,
\begin{enumerate}
\item\label{th:3main1}
the jumps $\beta\in\QQ$ of the induced filtration\enlargethispage{\baselineskip}%
\[
F_\bbullet^\irr \gr^V_\alpha\cH_v^k:=\frac{F_\bbullet^\irr\cH^k\cap V_\alpha\cH_v^k}{F_\bbullet^\irr\cH^k\cap V_{<\alpha}\cH_v^k}
\]
belong to $\alpha+\ZZ$,
\item\label{th:3main2}
on each generalized eigenspace of $\Res_{v=0}\nabla$ acting on $V_\alpha\cH_v^k/vV_\alpha\cH_v^k$, the nilpotent part of the residue strictly shifts by one the filtration naturally induced by $F_\bbullet^\irr V_\alpha\cH_v^k$.
\end{enumerate}
\end{theoreme}

Our proof in \S\ref{subsec:nearbymonodromy} is obtained by showing (Proposition \ref{prop:FMHS}) that the conditions needed for applying M.\,Saito's criterion \cite[Prop.\,3.3.17]{MSaito86} are fulfilled. More precisely, we will work with a filtration $F_\bbullet\ccE^{vf}(*\ccH)$ easy to define, and we postpone to~\S\ref{sec:relwithFirr} the proof that this is indeed the irregular Hodge filtration of $\ccE^{vf}(*\ccH)$. It would also be possible, as observed by T\ptbl Mochizuki \cite{Mochizuki15a}, to directly refer to a similar property for twistor $\cD$\nobreakdash-modules.

\subsubsection*{Over the chart $\Afu_u$}

Let us now consider the chart~$\Afu_u$. We denote by $\cH^k_u$ the restriction of $\cH^k$ to this chart. If $j:\Afu_u\moins\{0\}\hto\Afu_u$ denotes the open inclusion, we have $\cH^k_u=j_+\cH^k_{|\GGm}$. There is a natural $\cO_{\Afu_u}$-lattice $G_0\cH_u^k$ of the free $\cO_{\Afu_u}[u^{-1}]$-module $\cH^k_u$ called the Brieskorn lattice in analogy with the construction of Brieskorn in singularity theory \cite{Brieskorn70}. It can be defined in terms of the Hodge filtration of the Gauss-Manin system attached to $f$ (\cf the appendix). It can also be defined (\cf\S\ref{subsec:qEu}) as the push-forward by $q$ in a suitable sense of an $\cO_{X\times\Afu_u}(*\ccP_\red)$-module $G_0\ccE^{f/u}(*\ccH)$ equipped with a $u$-connection $u\rd+\rd f:G_0\ccE^{f/u}(*\ccH)\to G_0\ccE^{f/u}(*\ccH)\otimes\Omega_X^1$ and with a compatible action of $u^2\partial_u$.

The connection on $\cH^k$ has a pole of order at most two at $u=0$ when restricted to $G_0\cH_u^k$ (\cf Remark \ref{rem:actionud2u}). In the context of $\cDXu$-modules $\ccE^{f/u}(*\ccH)$ corresponds to $E^{f/u}(*H)=\cO_X(*D)[u,u^{-1}]\refu$.

\subsubsection*{Gluing}

We can then glue $G_0\cH_u^k$ with $V_\alpha\cH_v^k$ and obtain an $\cO_{\PP^1}$-bundle $\cH^k(\alpha)$ with a connection having a pole of order one at $v=0$ and of order two at $u=0$.

\subsection{The Kontsevich bundles $\cK^k(\alpha)$}\label{subsec:KB}
We now consider the Kontsevich bundles introduced in \S\ref{subsec:varirrhodgefiltrkontsevichbundles}. We can endow them with a natural meromorphic connection having a pole of order one at $v=0$ and of order two at most at $v=\infty$, and no other pole.

In order to do so, we start\footnote{This was suggested to us by T\ptbl Mochizuki.} by considering the morphism of complexes
\[
u^2\partial_u-f:\big(\Omega_f^\cbbullet(\alpha)[u],u\rd+\rd f\big)\to\big(\Omega_f^\cbbullet(\alpha+1)[u],u\rd+\rd f\big).
\]

\begin{lemme}\label{lem:omegaalphau}
For $\alpha\in[0,1)$, the natural inclusion of complexes
\[
(\Omega_f^\cbbullet(\alpha)[u],u\rd+\rd f)\to(\Omega_f^\cbbullet(\alpha+1)[u],u\rd+\rd f)
\]
is a quasi-isomorphism.
\end{lemme}

This lemma allows us to define an action of $u^2\partial_u$ on each $\cK^k(\alpha)_{|\Afu_u}$ and therefore a meromorphic connection $\nabla$ on $\cK^k(\alpha)$ with a pole of order at most two at $u=0$. We will show that $\nabla$ has at most a simple pole at $v=0$.

\begin{remarque}[due to T\ptbl Mochizuki]\label{rem:HSMochizuki}
Let $\cH$ be a vector bundle on $\PP^1$ equipped with a connection $\nabla$ having a simple pole at $v=0$ and a double pole (at most) at $v=\nobreak\infty$. Then the Harder-Narasimhan filtration $F^\bbullet\cH$ satisfies the Griffiths transversality property with respect to $\nabla$.

Indeed, the property is obviously true with respect to the connection $\rd$ on $\cH$ coming from $\rd$ on each summand in a Birkhoff-Grothendieck decomposition. We are thus reduced to proving a similar property for the $\cO$-linear morphism $\nabla-\rd$, and the result follows by noticing that $(\cH/F^{p-1}\cH)\otimes\Omega^1_{\PP^1}(\{v=0\}+2\{u=0\})$ has slopes~$<p$ while $F^p\cH$ has slopes $\geq p$.
\end{remarque}

\begin{proof}[Proof of Lemma \ref{lem:omegaalphau}]
We will show that the quotient complex has zero cohomology. From \cite[Prop.\,1.4.2]{E-S-Y13} we know that the inclusion of complexes
\[
(\Omega_f^\cbbullet(\alpha),\rd f)\to(\Omega_f^\cbbullet(\alpha+1),\rd f)
\]
is a quasi-isomorphism, and thus the quotient complex $(\cQ^\cbbullet,\rd f)$ has zero cohomology. Let $\omega=\sum_{j=0}^k\omega_ju^j$ be a local section of $\cQ^p[u]$ such that $(u\rd+\rd f)(\omega)=0$. Then $\rd f\wedge\omega_0=0$ and therefore there exists $\eta_0\in\cQ^{p-1}$ such that $\omega=\rd f\wedge\eta_0$. By replacing~$\omega$ with $\omega-(u\rd+\rd f)\eta_0$ and iterating the process we can assume that $\omega=\omega_ku^k$ and, dividing by $u^k$, that $\omega\in\cQ^p$. It satisfies then $\rd\omega=0$ and $\rd f\wedge\omega=0$, so $\omega=\rd f\wedge\eta$ for some $\eta\in\cQ^{p-1}$, and therefore $\rd f\wedge\rd\eta=0$. For any representative $\wt\eta\in\Omega_f^{p-1}(\alpha+1)$, we obtain
\[
\rd f\wedge\rd\wt\eta\in\Omega_f^{p+1}(\alpha)\subset\Omega_X^{p+1}(\log D)([\alpha P]).
\]
On the other hand, we note that
\[
\Omega_f^{p-1}(\alpha+1)=\rd f\wedge\Omega_X^{p-2}(\log D)([\alpha P])+\Omega_X^{p-1}(\log D)([\alpha P]),
\]
so we can assume that $\wt\eta\in\Omega_X^{p-1}(\log D)([\alpha P])$. Then $\rd\wt\eta\in\Omega_X^p(\log D)([\alpha P])$, and therefore $\rd\wt\eta\in\Omega_f^p(\alpha)$, that is, $\rd\eta=0$, so $\omega=(u\rd+\rd f)\eta$.
\end{proof}

\begin{proof}[Proof of Theorem \ref{th:2main}]
We will compare the filtered complex $\sigma^{\geq p}(\Omega_f^\cbbullet(\alpha)[v],\rd+\nobreak v\rd f)$ with the filtered relative de~Rham complex of $\ccE^{vf}(*\ccH)$ with respect to the projection to~$\Afu_v$. We introduce in \S\ref{subsec:FalphaE} a filtration $F_\alpha^\cbbullet\ccE^{vf}(*\ccH)$, which will be shown to coincide with $F_\alpha^{\irr,\cbbullet}\ccE^{vf}(*\ccH)$ in Theorem \ref{th:FFirr}.

\begin{theoreme}[\cf\S\ref{subsec:pf4main}, modulo Th.\,\ref{th:FFirr}]\label{th:4main}
There is a natural quasi-isomorphism of filtered complexes
\[
(\cO_{X\times\Afu_v}\otimes_{\cO_X[v]}\Omega_f^\cbbullet(\alpha)[v],\rd+v\rd f,\sigma^{\geq p})\to(\DR_{X\times\Afu_v/\Afu_v}V_\alpha \ccE^{vf}(*\ccH),F_\alpha^{\irr,p})
\]
which is compatible with the meromorphic action of $\nabla_{\partial_v}$.
\end{theoreme}

It follows from \eqref{eq:E1degenK} that applying $\bR q_*$ to the filtered complex on the right-hand side gives a strict complex (\ie we have a similar injectivity statement).

We apply $R^kq_*$ to the quasi-isomorphism of Theorem \ref{th:4main}. The non-filtered statement gives the first point of Theorem \ref{th:2main}, since $V_\alpha$ is compatible with proper push-forward. The second point is then obtained by applying the second point of Theorem~\ref{th:3main}.
\end{proof}

In a way similar to Theorem \ref{th:4main}, but algebraically with respect to $u$, we introduce in \S\ref{subsec:FalphaG0} a filtration $F_\alpha^\cbbullet G_0\ccE^{f/u}(*\ccH)$, which will be shown to coincide with $F_\alpha^{\irr,\cbbullet}G_0\ccE^{f/u}(*\ccH)$ in Theorem \ref{th:FFirr}, and we prove:

\begin{theoreme}[\cf\S\ref{subsec:pf5main}, modulo Th.\,\ref{th:FFirr}]\label{th:5main}
There is a natural quasi-isomorphism of filtered complexes
\[
(\Omega_f^\cbbullet(\alpha)[u],u\rd+\rd f,\sigma^{\geq p})\to(\DR_XG_0E^{f/u}(*H),F_\alpha^{\irr,p})
\]
which is compatible with the action of $\nabla_{\partial_u}$.
\end{theoreme}

As above, it follows from \eqref{eq:E1degenK} that applying $\bR q_*$ to the filtered complex on the right-hand side gives a strict complex.

By applying a degeneration statement similar to that of \cite[Prop.\,3.3.17]{MSaito86} proved in the appendix, we obtain a concrete description of the irregular Hodge filtration of~$\cH^k$.

\begin{corollaire}\label{cor:HNfiltr}
The isomorphism $\cK^k(\alpha)\isom\cH^k(\alpha)$ obtained by pushing forward the quasi-isomorphisms of Theorems \ref{th:4main} and \ref{th:5main} identifies the Harder-Narasimhan filtration of $\cK^k(\alpha)$ (hence of $\cH^k(\alpha)$) with the image on $\cH^k(\alpha)$ of the irregular Hodge filtration $F^\irr\cH^k$.
\end{corollaire}

\begin{remarque}
Another proof of Theorem \ref{th:2main} has recently been given by T.\,Mochizuki \cite{Mochizuki15a}, by showing an analogue of Theorem \ref{th:4main} in the framework of mixed twistor $\cD$-modules, but not referring explicitly to the irregular Hodge filtration.
\end{remarque}

\section{The \texorpdfstring{$\cD_{X\times\Afu_v}$}{DXAv}-module \texorpdfstring{$\ccE^{vf}(*\ccH)$}{EvfH}}\label{sec:Evf}

\subsection{Setting}\label{subsec:setting}
We will use the local setting and notation similar to that of \cite[\S1.1]{E-S-Y13} that we recall now, together with the notation introduced in \S\ref{subsec:rescaling}. In the local analytic setting, the space $X^\an$ is the product of discs $\Delta^\ell\times\Delta^m\times\Delta^{m'}$ with coordinates $(x,y)=(x_1,\dots,x_\ell,y_1,\dots,y_m,y'_1,\dots,y'_{m'})$ and we are given a multi-integer $\bme=\nobreak(e_1,\dots,e_\ell)\in(\ZZ_{>0})^\ell$ for which we set:
\begin{itemize}
\item
$f(x,y)=x^{-\bme}:=\prod_{i=1}^\ell x_i^{-e_i}$,
\item
$g(x,y)=1/f(x,y)=x^{\bme}$
\item
$P_\red=\{\prod_{i=1}^\ell x_i=0\}$, $H=\{\prod_{j=1}^m y_j=0\}$, $D=P_\red\cup H$.
\end{itemize}

Set $\cO=\CC\{x,y,y'\}$ and $\cD=\cO\langle\partial_x,\partial_y,\partial_{y'}\rangle$ is the ring of linear differential operators with coefficients in $\cO$, together with its standard increasing filtration $F_\cbbullet\cD$ by the total order w.r.t.\,$\partial_x,\partial_y,\partial_{y'}$:
\begin{equation*}\label{eq:FpD}
F_p\cD=\sum_{|\alpha|+|\beta|+|\gamma|\leq p}\cO\partial_x^\alpha\partial_y^\beta\partial_{y'}^\gamma,
\end{equation*}
where we use the standard multi-index notation with $\alpha\in\NN^\ell$, etc. Similarly we will denote by $\cO[t']$ the ring of polynomials in $t'$ with coefficients in $\cO$ and by $\cD[t']\langle\partial_{t'}\rangle$ the corresponding ring of differential operators.

Consider the left $\cD$-modules
\[
\cO(*P_\red)=\cO[x^{-1}],\quad
\cO(*H)=\cO[y^{-1}],\quad
\cO(*D)=\cO[x^{-1},y^{-1}]
\]
with their standard left $\cD$-module structure. They are generated respectively by $1/\prod_{i=1}^\ell x_i$, $1/\prod_{j=1}^m y_j$ and $1/\prod_{i=1}^\ell x_i\prod_{j=1}^m y_j$ as $\cD$-modules. We will consider on these $\cD$-modules the increasing filtration $F_\bbullet$ defined as the action of $F_\bbullet\cD$ on the generator:
\begin{align*}
F_p\cO(*P_\red)&=\sum_{|\bma|\leq p}\cO\cdot\partial_x^{\bma}(1/\prod\nolimits_{i=1}^\ell x_i)=\sum_{|\bma|\leq p}\cO x^{-\bma-\bf1},\\
F_p\cO(*H)&=\sum_{|\bmc|\leq p}\cO\cdot\partial_y^{\bmc}(1/\prod\nolimits_{j=1}^m y_j)=\sum_{|\bmc|\leq p}\cO y^{-\bmc-\bf1},\\
F_p\cO(*D)&=\sum_{|\bma|+|\bmc|\leq p}\cO\cdot\partial_x^{\bma}\partial_y^{\bmc}(1/\prod\nolimits_{i=1}^\ell x_i\prod\nolimits_{j=1}^m y_j)=\sum_{|\bma|+|\bmc|\leq p}\cO x^{-\bma-\bf1}y^{-\bmc-\bf1},
\end{align*}
so that $F_p=0$ for $p<0$. These are the ``filtrations by the order of the pole'' in \cite[(3.12.1) p.\,80]{Deligne70}, taken in an increasing way. Regarding $\cO(*H)$ as a $\cD$-submodule of $\cO(*D)$, we have $F_p\cO(*H)=F_p\cO(*D)\cap\cO(*H)$ and similarly for $\cO(*P_\red)$. On the other hand it clearly follows from the formulas above that
\begin{equation*}\label{eq:FpPH}
F_p\cO(*D)=\sum_{q+q'\leq p}F_q\cO(*H)\cdot F_{q'}\cO(*P_\red),
\end{equation*}
where the product is taken in $\cO(*D)$.

\subsection{The $V$-filtration of the \texorpdfstring{$\cD_{X\times\Afu_v}$}{DXAv}-module \texorpdfstring{$\ccE=\ccE^{vf}(*\ccH)$}{EvfH}}\label{subsec:VfiltEv}
On $X\times\Afu_v$ we consider the holonomic $\cD_{X\times\Afu_v}$-module that we denote by $\ccE^{vf}(*\ccH)$. It is defined by the formula
\[
\ccE^{vf}(*\ccH):=\big(\cO_{X\times\Afu_v}(*(D\times\Afu_v)),\rd+\rd(vf)\big).
\]
For the sake of simplicity, we will set $\ccE=\ccE^{vf}(*\ccH)$. Then $\ccE$ has a global section, equal to~$1$, that we denote by $\revf$ on $X\times\Afu_v$. Similarly, we will consider the $v$-algebraic version of the same object, regarded as a $\cDXv$-module:
\[
E=E^{vf}(*H):=(\cO_X(*D)[v],\rd+\rd(vf))=\cO_X(*D)[v]\cdot\revf.
\]

It is standard that the $\cD_{X\times\Afu_v}$-module $\ccE$ is holonomic. However, it is not of exponential type as considered in \cite{E-S-Y13} since $vf$ is only a rational function, but is exponentially regular according to Proposition \ref{prop:Nexp}\eqref{prop:Nexp1}, hence it enters the frame considered in \S\ref{subsec:exptwisthol}. It is however known to have regular singularities along~$v=0$ (in~a~sense made precise in \cite{Bibi05b}) which has been thoroughly analyzed in \cite{Bibi96a}. On the other hand, it is easy to check that $F_0\cO_{X\times\Afu_v}(*\ccH)\revf$ generates $\ccE$ as a $\cD_{X\times\Afu_v}$-module.

Let us recall the definition of the $V$-filtration (considered increasingly) along $v=\nobreak0$ over $\Afu_v$. For each $\alpha\in[0,1)$ and $k\in\ZZ$, $V_{\alpha+k}\ccE$ is a coherent $\cD_{X\times\Afu_v/\Afu_v}$-module (by~\hbox{regularity}) equipped with an action of $v\partial_v$, and the minimal polynomial of $v\partial_v$ on $V_{\alpha+k}\ccE/V_{\alpha+k-1}\ccE$ has roots contained in $[-\alpha-k,-\alpha-k+1)$. We have by definition, for $k\geq1$,
\[
V_{\alpha-k}\ccE=v^k V_\alpha\ccE\quad\text{and}\quad V_{\alpha+k}\ccE=\begin{cases}
\sum_{j=1}^k\partial_v^{j-1} V_1\ccE&\text{if }\alpha=0,\\[5pt]
V_k\ccE+\partial_v^k V_\alpha\ccE&\text{if }\alpha\in(0,1).
\end{cases}
\]
Since $\ccE[v^{-1}]=j_{0_+}j_0^+\ccE$ is also holonomic, it also has a $V$-filtration. It satisfies, for any $k\in\ZZ$,
\[
V_{\alpha-k}(\ccE[v^{-1}])=v^k V_\alpha\ccE.
\]
There is also a notion of $V$-filtration for holonomic $\cDXv$-modules and we have $V_{\alpha+k}\ccE=\cO_{X\times\Afu_v}\otimes_{\cO_X[v]}V_{\alpha+k}E$.

\begin{lemme}[Description of $V_\bbullet E$]\label{lem:Vfiltration}
Let us fix $\beta\geq0$ and $\alpha\in[0,1)$.
\begin{enumerate}
\item\label{lem:Vfiltration1}
Near a point of $(X\moins P_\red)$, we have
\[
V_{\alpha+k}E=v^{\max(-k,0)}E\quad\text{for }k\in\ZZ.
\]

\item\label{lem:Vfiltration2}
Near a point of $P_\red$, in the local setting of \S\ref{subsec:setting}, we have
\[
V_\beta E=\sum_{\bma\geq0}\big(F_0\cO(*P_\red)\big)([\beta P])(*H)[v\partial_v]x^{-\bma}P_{\bma,\beta}(v\partial_v)\cdot\revf,
\]
with (convention: $\prod_{k\in\emptyset}\star_k=1$)
\[
P_{\bma,\beta}(s):=\prod_{i=1}^\ell\prod_{k=1}^{a_i}\Big(s+\frac{[\beta e_i]+k}{e_i}\Big).
\]
\end{enumerate}
\end{lemme}

\begin{proof}
The first point follows from the relation $\partial_v\revf=f\revf$. The second point is a reformulation of \cite[Lem.\,4.9]{Bibi96a}.
\end{proof}

For each $\bma\geq0$, let us set $I(\bma)=\{i\mid a_i=0\}\subset\{1,\dots,\ell\}$ and $x_{I(\bma)}=(x_i)_{i\in I(\bma)}$. For $\lambda\geq0$ we also set $P_{\bma,\lambda,\beta}(s)=(s+\beta)^\lambda P_{\bma,\beta}(s)$. Then near a point of $P_\red$, each local section of $V_\beta E$ has a \emph{unique} decomposition
\begin{equation}\label{eq:uniquevdv}
\sum_{\bma\geq0}\sum_{\lambda\geq0} h_{\bma,\lambda,\beta}(x_{I(\bma)},y,y^{-1},y')x^{-[\beta\bme]-{\bf1}}x^{-\bma}P_{\bma,\lambda,\beta}(v\partial_v)\revf,
\end{equation}
with $h_{\bma,\lambda,\beta}(x_{I(\bma)},y,y^{-1},y')\in\CC\{x_{I(\bma)},y,y'\}[y^{-1}]$. Moreover (\cf \loccit), a section \eqref{eq:uniquevdv} belongs to $V_{<\beta}E$ if and only if
\begin{equation}\label{eq:grV}
\forall\bma,\lambda\geq0,\quad h_{\bma,\lambda,\beta}\neq0\implique\begin{cases}
\lambda\geq\#\{i\mid\beta e_i\in\ZZ\}&\text{if }\beta>0,\\
\lambda\geq\ell+1&\text{if }\beta=0.
\end{cases}
\end{equation}
Let us make explicit the action of $\CC[v]$ on a section \eqref{eq:uniquevdv}. For $j\geq1$ we have
\[
v^j\cdot h_{\bma,\lambda,\beta}x^{-[\beta\bme]-{\bf1}}x^{-\bma}P_{\bma,\lambda,\beta}(v\partial_v)\revf=h_{\bma,\lambda,\beta}x^{-[\beta\bme]-{\bf1}}x^{-\bma+j\bme}P_{\bma,\lambda,\beta}(v\partial_v-j)v^j\partial_v^j\revf.
\]
Set $\bma-j\bme=\bma'-\bma''$, with $a'_i=\max(a_i-je_i,0)$, $a''_i=\max(je_i-a_i,0)$. The polynomial
\[
P_{\bma,\lambda,\beta}(s-j)=(s+\beta-j)^\lambda\prod_i\prod_{k=1}^{a_i}\big(s+([(\beta-j)e_i]+k)/e_i\big)
\]
is a multiple of $P_{\bma',0,\beta}(s)$ and there is a polynomial $R_{\bma,\lambda,j,\beta}(s)\in\QQ[s]$ such that
\[
P_{\bma,\lambda,\beta}(v\partial_v-j)v^j\partial_v^j=R_{\bma,\lambda,j,\beta}(v\partial_v)P_{\bma',0,\beta}(v\partial_v)=\sum_{\mu\geq0}c_\mu P_{\bma',\mu,\beta}(v\partial_v)
\]
with $c_\mu\in\QQ$. We thus obtain
\begin{multline}\label{eq:vjuniquevdv}
v^j\cdot h_{\bma,\lambda,\beta}x^{-[\beta\bme]-{\bf1}}x^{-\bma}P_{\bma,\lambda,\beta}(v\partial_v)\revf\\
=\sum_{\mu\geq0}c_\mu x^{\bma''}h_{\bma,\lambda,\beta}x^{-[\beta\bme]-{\bf1}}x^{-\bma'}P_{\bma',\mu,\beta}(v\partial_v)\revf,
\end{multline}
and since $I(\bma')=\{i\mid a_i-je_i\leq0\}\supset I(\bma)$, we obtain the result in the form of \eqref{eq:uniquevdv}.

\begin{lemme}\label{lem:PQ}
For any monic polynomial $P(s)$ of degree $p$, there exists a monic polynomial $Q(s)$ of degree $p$ such that $P(v\partial_v)\revf=Q(v/x^\bme)\revf$ in $\ccE$.\qed
\end{lemme}

Let us then denote $Q_{\bma,\lambda,\beta}(s)$ the polynomial associated with $P_{\bma,\lambda,\beta}(s)$ by Lem\-ma~\ref{lem:PQ}. We thus have
\begin{equation}\label{eq:Valphav}
V_\beta E=\sum_{\bma\geq0}\sum_{\lambda\geq0} \big(F_0\cO_X(*P_\red)\big)([\beta P])(*H)x^{-\bma}Q_{\bma,\lambda,\beta}(vf)\cdot\revf,
\end{equation}
and we note that $\deg Q_{\bma,\lambda,\beta}=|\bma|+\lambda$. Moreover, each local section has a unique decomposition
\begin{equation}\label{eq:uniquev}
\sum_{\bma\geq0}\sum_{\lambda\geq0} h_{\bma,\lambda,\beta}(x_{I(\bma)},y,y^{-1},y')x^{-[\beta\bme]-{\bf1}}x^{-\bma}Q_{\bma,\lambda,\beta}(vf)\revf.
\end{equation}

\begin{corollaire}\label{cor:Vfiltration}
Let us denote by $\gr^{[v]}V_\beta E$ the grading of $V_\beta E$ with respect to the degree in $v$. Then, in a neighbourhood of a point of $P_\red$, we have
\[
\gr_p^{[v]}V_\beta E\simeq\big(F_p\cO_X(*P_\red)\big)([(\beta+p)P])(*H)\cdot v^p.
\]
\end{corollaire}

\Subsection{The filtration \texorpdfstring{$F_{\alpha+\bbullet}E$}{FE}}\label{subsec:FalphaE}

Although the function $vf$ does not extend as a map $X\times\CC_v\to\PP^1$, we can nevertheless adapt in a natural way the definition given in \cite[(1.6.1) \& (1.6.2)]{E-S-Y13} for the case of the map $f:X\to\PP^1$.

\begin{definition}[The filtration]\label{def:Fexpressionv}
For $\alpha\in[0,1)$ we set, over $\Afu_v$,
\begin{align*}
F_{\alpha+p}E^{vf}&=\Big(\sum_{k\leq p}F_k\cO_X(*P_\red)([(\alpha+p)P])v^k\Big)[v]\cdot\revf,\\
F_{\alpha+p}E&=\sum_{q+q'\leq p}F_q\cO_X(*H)\cdot F_{\alpha+q'}E^{vf}.
\end{align*}
The analytification of these filtrations with respect to $v$ are denoted $F_{\alpha+p}\ccE^{vf}$ and $F_{\alpha+p}\ccE$ respectively.
\end{definition}

\begin{lemme}\label{lem:Fexpressionv}\mbox{}
For each $\alpha\in[0,1)$, the filtration $F_{\alpha+\bbullet}E$ is an $F_\bbullet\cDXv$-filtration which satisfies the following properties.
\begin{enumerate}
\item\label{lem:Fexpressionvcomp}
$F_{\alpha+p_1}E\subset F_{\beta+p_2}E$ for all $p_1,p_2\in\ZZ$ and $\beta\in[0,1)$ such that $\alpha+p_1\leq\beta+p_2$. Moreover, $F_{\alpha+p}E=0$ for $p<0$.
\item\label{lem:Fexpressionuv}
When restricted to $\GGm$, the filtration $F_{\alpha+p}\ccE$ is equal to $F_{\alpha+p}^\Del\ccE_{X\times\GGm}$ as defined in \cite[(1.6.2)]{E-S-Y13} for the map $vf:X\times\GGm\to\PP^1$, and for each $v_o\in\CC^*$, we have
\[
F_{\alpha+p}\ccE/(v-v_o)F_{\alpha+p}\ccE=F^\Del_{\alpha+p}\ccE^{v_of}(*H).
\]
\item\label{lem:Fexpressionvgood}
The filtration $F_{\alpha+\bbullet}E$ satisfies
\[
F_{\alpha+p}E=F_p\cD_{X\times\Afu_v}\cdot F_\alpha E\,;
\]
in particular, it is good with respect to $F_\bbullet\cDXv$.
\end{enumerate}
\end{lemme}

\begin{proof}
The exhaustivity is clear from the expression of Definition \ref{def:Fexpressionv}, and the first two points are straightforward. Let us check \eqref{lem:Fexpressionvgood}. It is enough to check it locally analytically on~$X$.

\subsubsection*{Near a point of $X\setminus P_\red$}
If $H=\emptyset$ and $p\geq0$, we have $F_{\alpha+p}E=\cO[v]\revf$ and the generation by $F_\alpha E$ is clear.

If $H=\{y_1\cdots y_m=0\}$ and $p>0$, we have
\[
F_{\alpha+p}E=\sum_{|\bma|\leq p}y^{-\bma-\bf1}\cO[v]\revf.
\]
Since $\partial_y^{\bma}(y^{-\bf1}\cO[v]\revf)=\star y^{-\bma-\bf1}\cO[v]\revf\bmod F_{\alpha+p-1}E$ if $|\bma|=p$, we get the generation by $F_\alpha E$ near such a point.

\subsubsection*{Near a point of $P_\red$}
From the equalities (for some nonzero constants~$\star$):
\begin{align*}
\partial_{x_i}\big(x^{-([\alpha\bme]+\bf1)}\cdot\revf\big)&=\star\,x_i^{-1}x^{-([\alpha\bme]+\bf1)}\cdot\revf+\star\,x_i^{-1}x^{-([(\alpha+1)\bme]+\bf1)}v\cdot\revf\\
\partial_v\big(x^{-([\alpha\bme]+\bf1)}\revf\big)&=x^{-([(\alpha+1)\bme]+\bf1)}\cdot\revf,
\end{align*}
we conclude
\[
F_1\cDXv\cdot F_\alpha E^{vf}=\big(F_0\cO(*P_\red)+F_1\cO(*P_\red)v\big)[v]\cdot x^{-[(\alpha+1)\bme]}\revf
\]
and by iterating the argument we get the generation property. The corresponding property for $F_{\alpha+\bbullet}E$ is proved similarly.
\end{proof}

We will rely on computations made in \cite{Bibi96a} and we will first express differently the filtration $F_{\alpha+p}E$. Let us define $G_pE$ as the filtration by $\cO_X$-modules (but not $\cO_X[v]$-modules), defined as
\[
G_pE=\bigoplus_{k=0}^p(F_{p-k}\cO_X(*H))(*P_\red) v^k\cdot\revf.
\]
The filtration $G_\bbullet E$ clearly satisfies
\begin{equation}\label{eq:Gp}
\begin{aligned}
p<0&\implique G_pE=0,\\
q\geq0&\implique G_pE\cap v^qE=v^q G_{p-q}E,\\
p-q<0&\implique G_pE\cap v^qE=0.
\end{aligned}
\end{equation}
For $\alpha\in[0,1)$ and $p\in\ZZ$, we set
\begin{equation}\label{eq:Falphap}
F'_{\alpha+p}E:=\sum_{k+j\leq p}(G_kE\cap V_{\alpha+j}E).
\end{equation}
Then $F'_{\alpha+p}E$ is an $\cO_X[v]$-module. Note also that $F'_{\alpha+\bbullet}E$ is an $F_\bbullet\cDXv$-filtration. Indeed, $G_kE$ is stable by $\partial_{y'},\partial_v$, and we have $\partial_{x_i},\partial_{y_j}G_kE\subset G_{k+1}E$; moreover, $V_{\alpha+j}E$ is stable by $\partial_x,\partial_y,\partial_{y'}$, and we have $\partial_v V_{\alpha+j}E\subset V_{\alpha+j+1}E$.

Recall that, for $j\geq0$, we have by definition $V_{\alpha-j}E=v^jV_\alpha E$, so that for $k\geq0$,
\begin{equation}\label{eq:multvjGV}
v^j:G_kE\cap V_\alpha E\isom G_{k+j}E\cap\nobreak V_{\alpha-j}E.
\end{equation}
Therefore, we can also write
\begin{equation}\label{eq:Fprimealphapv}
F'_{\alpha+p}E
:=\CC[v](G_pE\cap V_\alpha E)+\sum_{j=1}^p(G_{p-j}E\cap V_{\alpha+j}E).
\end{equation}
It follows that $F'_{\alpha+p}E=0$ for $p<0$ (and $\alpha\in[0,1)$).

\begin{lemme}\label{lem:FF'v}
For each $\alpha\in[0,1)$ and $p\in\ZZ$ we have
\[
F'_{\alpha+p}E=F_{\alpha+p}E.
\]
\end{lemme}

\begin{proof}
Let us first consider an analytic neighbourhood of a point of $X\setminus P_\red$. Due to the relation $\partial_v\revf=f\revf$, we have, near such a point, $V_{\alpha+j}E=v^{\max(-j,0)}E$ for any $j\in\ZZ$, and $G_pE=\bigoplus_{k=0}^pF_k\cO(*H)v^{p-k}\revf$. Then, near such a point, \eqref{eq:Fprimealphapv} reads
\begin{align*}
F'_{\alpha+p}E&:=\CC[v](G_pE\cap V_0E)\\
&=\CC[v]G_pE=F_p\cO(*H)[v]\revf=F_{\alpha+p}E.
\end{align*}

We now argue locally at a point of $P_\red$. We refine \eqref{eq:uniquevdv} in order to take into account the pole order along $H$, so a section of $V_\beta E$ can be written in a unique way~as
\begin{equation}\label{eq:uniquevdvH}
\sum_{\bma\geq0}\sum_{\lambda\geq0}\sum_{\bmc\geq0} h_{\bma,\bmc,\lambda,\beta}(x_{I(\bma)},y_{J(\bmc)},y')x^{-[\beta\bme]-{\bf1}}x^{-\bma}y^{-\bmc-{\bf1}}P_{\bma,\lambda,\beta}(v\partial_v)\revf,
\end{equation}
with $J(\bmc)=\{j\mid c_j=0\}$ and $h_{\bma,\bmc,\lambda,\beta}(x_{I(\bma)},y_{J(\bmc)},y')\in\CC\{x_{I(\bma)},y,y'\}$. Arguing as in the proof of \cite[Lemma 4.11]{Bibi96a}, we obtain that, for $\beta\geq0$, a section \eqref{eq:uniquevdvH} belongs to $G_kE\cap V_\beta E$ if and only if the coefficients $h_{\bma,\bmc,\lambda,\beta}$ are zero whenever $\deg P_{\bma,\lambda,\beta}+|\bmc|=|\bma|+|\bmc|+\lambda$ is $>k$ (note that this condition clearly defines an increasing filtration with respect to~$k$).

We will first show that $F_\alpha E=F'_\alpha E$ for $\alpha\in[0,1)$. We have
\begin{align*}
F'_\alpha E&=\CC[v](G_0E\cap V_\alpha E)\\
\tag*{and} F_\alpha E&=F_0\cO(*D)([\alpha P])[v]\revf.
\end{align*}
Here we are considering the case $k=0$ and $\beta=\alpha$. Let us consider a section \eqref{eq:uniquevdvH} in $G_0E\cap V_\alpha E$. The only possible term occurs for $\bma=0$, $\lambda=0$ and $\bmc=0$, so $G_0E\cap V_\alpha E=F_0\cO(*D)([\alpha P])\revf$. Therefore,
\begin{equation}\label{eq:FalphaValpha}
F_\alpha E:=F_0\cO(*D)([\alpha P])[v]\revf=\CC[v](G_0E\cap V_\alpha E)=F'_\alpha E.
\end{equation}
Since $F_{\alpha+p}E=F_p\cD[v]\langle\partial_v\rangle\cdot F_\alpha E$ (Lemma \ref{lem:Fexpressionv}\eqref{lem:Fexpressionvgood}), and since $F'_{\alpha+p}E$ is an $F_\bbullet\cDXv$-filtration, it follows that $F_{\alpha+p}E\subset F'_{\alpha+p}E$ for all $p$.

Let us now show the reverse inclusion $F'_{\alpha+p}E\subset F_{\alpha+p}E$. Let us consider a term in $G_kE\cap V_{\alpha+p-k}E$ ($0\leq k\leq p$) of the form
\[
h(x_{I_\beta(\bma)},y_{J(\bmc)},y')x^{-[\beta\bme]-{\bf1}}x^{-\bma}y^{-\bmc-{\bf1}}P_{\bma,\lambda,\beta}(v\partial_v)\revf,
\]
with $\beta=\alpha+p-k$, $\bma\geq0$, $\lambda+|\bma|+|\bmc|\leq k$. Let us rewrite $P_{\bma,\lambda,\beta}(v\partial_v)$ in terms of the monomials $v^j\partial_v^j$. For $j\leq \lambda+|\bma|\leq k-|\bmc|$, the result of the action of $v^j\partial_v^j$ on $h(x_{I_\beta(\bma)},y_{J(\bmc)},y')x^{-[\beta\bme]-{\bf1}}x^{-\bma}y^{-\bmc-{\bf1}}\revf$ is
\begin{multline*}
h(x_{I_\beta(\bma)},y_{J(\bmc)},y')x^{-[(\beta+j)\bme]-{\bf1}}x^{-\bma}y^{-\bmc-{\bf1}}v^j\revf\\
=h(x_{I_\beta(\bma)},y_{J(\bmc)},y')x^{-[(\alpha+p)\bme]-{\bf1}}x^{-\bma+(k-j)\bme}y^{-\bmc-{\bf1}}v^j\revf.
\end{multline*}
For $\bma'\in\ZZ^\ell$, let us set $|\bma'|_+=\sum_i\max(0,|a'_i|)$. Since $|\bma|_+=|\bma|\leq k$, the reverse inclusion follows from the lemma below.
\end{proof}

\begin{lemme}\label{lem:precision}
For $\bma'\in\ZZ^\ell$ and $k\geq0$, assume that $|\bma'|_+\leq k$. Then for $j$ such that $0\leq j\leq k$, we have $|\bma'-(k-j)\bme|_+\leq j$.
\end{lemme}

\begin{proof}
We argue by decreasing induction on $j$ and the result is true if $j=k$ by assumption. We~are reduced to proving that, if $|\bma'|_+\geq1$, then $|\bma'-\bme|_+\leq|\bma'|_+-1$. There exists~$i_o$ such that $a'_{i_o}\geq1$, so $\max(a'_{i_o},0)=a'_{i_o}\geq1$ and $\max(a'_{i_o}-e'_{i_o},0)\leq a'_{i_o}-1$. Since $\max(a'_i-e'_i,0)\leq \max(a'_i,0)$ for any $i$, we get $|\bma'-\bme|_+\leq |\bma'|_+-1$, as wanted.
\end{proof}

\subsection{Filtration on the nearby cycles of \texorpdfstring{$\ccE$}{E} along \texorpdfstring{$v=0$}{v0}}\label{subsec:psivE}

In this subsection, we analyze the filtration induced by $F_{\alpha+\bbullet}\ccE$ on the nearby cycles of $\ccE$ along $v=0$. Our objective is to show that M.\,Saito's criterion \cite[Prop.\,3.3.17]{MSaito86} applies.

\begin{proposition}\label{prop:FMHS}\mbox{}
\begin{enumerate}
\item\label{prop:FMHS3}
For each $\alpha\in[0,1)$, the filtered module $(\ccE^{vf}(*\ccH),F_{\alpha+\bbullet}\ccE^{vf}(*\ccH))$ is strictly specializable and regular along $v=0$, in the sense of \cite[(3.2.1)]{MSaito86}.
\item\label{prop:FMHS4}
Let $V_\bbullet\ccE^{vf}(*\ccH)$ be the $V$-filtration of $\ccE^{vf}(*\ccH)$ along $v=0$ and, for each $\alpha\in [0,1)$, let us set
\[
\psi_{v,\exp(-2\pi\ri\alpha)}\ccE^{vf}(*\ccH):=\gr_\alpha^V\ccE^{vf}(*\ccH)=V_\alpha\ccE^{vf}(*\ccH)/V_{<\alpha}\ccE^{vf}(*\ccH).
\]
For each $\alpha\in[0,1)$ the jumps of the induced filtration (considered as a filtration indexed by $\QQ$)
\[
F_\bbullet\psi_{v,\exp(-2\pi\ri\alpha)}\ccE^{vf}(*\ccH):=\frac{F_\bbullet\cap V_\alpha\ccE^{vf}(*\ccH)}{F_\bbullet\cap V_{<\alpha}\ccE^{vf}(*\ccH)}
\]
occur at $\alpha+\ZZ$ at most, and the filtration
\[
F_p\psi_v\ccE^{vf}(*\ccH):=\bigoplus_{\alpha\in[0,1)}F_{\alpha+p}\psi_{v,\exp(-2\pi\ri\alpha)}\ccE^{vf}(*\ccH)
\]
is (up to a shift by $\dim X-1$ on $\psi_{v,\neq1}\ccE^{vf}$ and by $\dim X$ on $\psi_{v,1}\ccE^{vf}$) the Hodge filtration of a mixed Hodge module.
\item\label{prop:FMHS5}
If moreover $H=\emptyset$, this mixed Hodge module is polarized by the nilpotent part of the monodromy naturally acting on $\psi_v\ccE^{vf}$, induced by the action of $\exp-2\pi\ri v\partial_v$.
\end{enumerate}
\end{proposition}

The latter statement means that the weight filtration of the corresponding mixed Hodge module is, up to a shift which depends on whether $\alpha=0$ or $\alpha\neq0$, the monodromy filtration of the nilpotent endomorphism induced by $v\partial_v+\alpha$ on $\psi_{v,\exp(-2\pi\ri\alpha)}\ccE^{vf}$.

\begin{proof}[Proof of Proposition \ref{prop:FMHS}\eqref{prop:FMHS4} and \eqref{prop:FMHS5}]
It is enough to work in the algebraic setting with respect to $v$. Recall that we set $E=E^{vf}(*H)$ for short and that $F'$ was defined by \eqref{eq:Falphap}. Let $\beta\in\nobreak[0,1)$. We claim that
\begin{equation}\label{eq:FcapV}
F'_{\alpha+p}E\cap V_\beta E=\begin{cases}
\CC[v](G_pE\cap V_\alpha E)+(G_{p-1}E\cap V_\beta E)&\text{if }\beta>\alpha,\\[5pt]
v\CC[v](G_pE\cap V_\alpha E)+(G_pE\cap V_\beta E)&\text{if }\beta\leq\alpha.
\end{cases}
\end{equation}
This implies that
\begin{equation}\label{eq:FgrV}
F'_{\alpha+p}\gr_\beta^VE=\begin{cases}
G_{p-1}\gr_\beta^VE=F'_{<\alpha+p}\gr_\beta^VE&\text{if }\beta>\alpha,\\[5pt]
G_p\gr_\beta^VE=F'_{\beta+p}\gr_\beta^VE&\text{if }\beta\leq\alpha.
\end{cases}
\end{equation}

Let us prove \eqref{eq:FcapV}. We have $F'_{\alpha+p}E\subset V_{\alpha+p}E$ and
\begin{align*}
F'_{\alpha+p}E\cap V_{\alpha+p-1}E&=\CC[v](G_pE\cap V_\alpha E)+\sum_{\ell=1}^{p-1}(G_{p-\ell}E\cap V_{\alpha+\ell}E)+(G_0E\cap V_{\alpha+p-1}E)\\
&=\CC[v](G_pE\cap V_\alpha E)+\sum_{\ell=1}^{p-1}(G_{p-\ell}E\cap V_{\alpha+\ell}E),
\end{align*}
and by decreasing induction one eventually finds
\[
F'_{\alpha+p}E\cap V_{\alpha+1}E=\CC[v](G_pE\cap V_\alpha E)+(G_{p-1}E\cap V_{\alpha+1}E).
\]
Intersecting now with $V_\beta E$ gives the first line of \eqref{eq:FcapV}. The second one is obtained by showing in the same way
\[
F'_{\alpha+p}E\cap V_\alpha E=\CC[v](G_pE\cap V_\alpha E)=v\CC[v](G_pE\cap V_\alpha E)+(G_pE\cap V_\alpha E).
\]

Lemma \ref{lem:FF'v}, together with \eqref{eq:FgrV} and \cite[Th.\,4.3]{Bibi96a}, proves \ref{prop:FMHS}\eqref{prop:FMHS4} and~\eqref{prop:FMHS5}.
\end{proof}

\begin{proof}[Proof of \ref{prop:FMHS}\eqref{prop:FMHS3}]
Continuing the proof of \eqref{eq:FcapV} gives, for $\beta\in[0,1)$ and $\ell\geq1$,
\[
F'_{\alpha+p}E\cap V_{\beta-\ell}E=\begin{cases}
v^\ell\CC[v](G_pE\cap V_\alpha E)+(G_{p+\ell-1}E\cap V_{\beta-\ell}E)&\text{if }\beta>\alpha,\\[5pt]
v^{\ell+1}\CC[v](G_pE\cap V_\alpha E)+(G_{p+\ell}E\cap V_{\beta-\ell}E)&\text{if }\beta\leq\alpha,
\end{cases}
\]
which amounts to
\[
F'_{\alpha+p}E\cap V_{\beta-\ell}E=\begin{cases}
v^\ell\Big[\CC[v](G_pE\cap V_\alpha E)+(G_{p-1}E\cap V_\beta E)\Big]&\text{if }\beta>\alpha,\\[8pt]
v^\ell\Big[v\CC[v](G_pE\cap V_\alpha E)+(G_pE\cap V_\beta E)\Big]&\text{if }\beta\leq\alpha,
\end{cases}
\]
that is, in any case,
\[
F'_{\alpha+p}E\cap V_{\beta-\ell}E=v^\ell(F'_{\alpha+p}E\cap V_\beta E),
\]
which is \cite[(3.2.1.2)]{MSaito86} (up to changing the convention for the $V$-filtration), since~$v$ acts in an injective way on $V_\beta E$.

We now wish to prove that Property (3.2.1.3) of \cite{MSaito86} holds, that is, for each $\beta>0$ and for each $\alpha\in[0,1)$ and $p\in\ZZ$, the morphism
\[
\partial_v:F'_{\alpha+p}\gr_\beta^VE\to F'_{\alpha+p+1}\gr_{\beta+1}^VE
\]
is an isomorphism. Assume first that there exists $p_o\geq0$ such that $\alpha+p_o\leq\beta<\alpha+p_o+1$ (otherwise, $0<\beta<\alpha$, a case which will be treated separately). Then a computation similar to that for proving \eqref{eq:FcapV} gives
\[
F'_{\alpha+p}E\cap V_\beta E=\begin{cases}\dpl
\CC[v](G_pE\cap V_\alpha E)+\sum_{\ell=1}^{p_o}(G_{p-\ell}E\cap V_{\alpha+\ell}E)+(G_{p-p_o-1}E\cap V_\beta E)\\[-5pt]
\hspace*{4.5cm}\text{if }\alpha+p_o<\beta<\alpha+p_o+1,\\[5pt]
\dpl\CC[v](G_pE\cap V_\alpha E)+\sum_{\ell=1}^{p_o}(G_{p-\ell}E\cap V_{\alpha+\ell}E)\quad\text{if }\beta=\alpha+p_o.
\end{cases}
\]
As a consequence, we find
\[
F'_{\alpha+p}\gr^V_\beta E=\begin{cases}
G_{p-p_o-1}\gr^V_\beta E&\text{if }\alpha+p_o<\beta<\alpha+p_o+1,\\[5pt]
G_{p-p_o}\gr^V_\beta E&\text{if }\beta=\alpha+p_o.
\end{cases}
\]
If $0<\beta<\alpha$, we also get $F'_{\alpha+p}\gr^V_\beta E=G_p\gr^V_\beta E$. So we are reduced to proving that, for any $\beta>0$ and any $k$, the morphism
\begin{equation}\label{eq:dvGpgr}
\partial_v:G_k\gr^V_\beta E\to G_k\gr^V_{\beta+1}E
\end{equation}
is an isomorphism.

Away from $P_\red$, we have $\gr_\beta^VE=0$ for each $\beta>0$, so the assertion is empty. Let us prove the assertion in the neighbourhood of a point of $P_\red$. The left action of $\partial_v$ on a term of the sum \eqref{eq:uniquevdvH} gives, since $\partial_v\revf=x^{-\bme}\revf$ and due to standard commutation rules,
\[
P_{\bma,\lambda,\beta}(v\partial_v+1)h_{\bma,\bmc,\lambda,\beta}(x_{I(\bma)},y_{J(\bmc)},y')x^{-[(\beta+1)\bme]-\bf1}x^{-\bma}y^{-\bmc-\bf1}\revf.
\]
We have $P_{\bma,\lambda,\beta}(v\partial_v+1)=P_{\bma,\lambda,\beta+1}(v\partial_v)$ and we use \eqref{eq:grV} for $\beta>0$ to conclude that \eqref{eq:dvGpgr} is an isomorphism.

It remains to be checked that $F'_{\alpha+\bbullet}\gr_\beta^VE$ is a good filtration for any $\beta\in\RR$. The previous arguments reduces us to checking this for $\beta\in[0,1]$, and we are reduced to proving that, for any such $\beta$, $G_\bbullet\gr_\beta^VE$ is a good filtration. This follows from \cite[4.14]{Bibi96a}, since this filtration is identified (after grading by a finite filtration) to a filtration which is already known to be good (and which is the Hodge filtration of a mixed Hodge module).
\end{proof}

\subsection{Computation of \texorpdfstring{$F_{\alpha+\bbullet}E\cap V_\alpha E$}{FaVa}}\label{subsec:computFV}

The previous section shows that, for $\alpha,\beta\in[0,1)$, the computation of $F_{\beta+\bbullet}E\cap V_\alpha E$ is interesting mainly when $\beta=\alpha$. Note that $F_{\alpha+p}E=0$ for $p<0$ and that, away from $P_\red$, we have $F_{\alpha+p}E^{vf}\cap\nobreak V_\alpha E^{vf}=E^{vf}=\cO_X\revf$.

\begin{lemme}\label{lem:grFalphaVE}
For $\alpha\in[0,1)$ and $p\geq0$, we have
\begin{starequation}\label{eq:grFalphaVE}
\begin{split}
\gr_p^{F_\alpha}V_\alpha E&=\bigoplus_{k\geq0}\Big(\gr_{p+k}^GV_{\alpha-k}E/\gr_{p+k}^GV_{\alpha-k-1}E\Big)\\
&\simeq\CC[v]\otimes_\CC\Big(\gr_p^GV_\alpha E/\gr_p^GV_{\alpha-1}E\Big).
\end{split}
\end{starequation}%
\end{lemme}

Note that, for $p=0$, we have $\gr_0^GV_\alpha E=G_0E\cap V_\alpha E$.

\begin{proof}
On the one hand, the natural map $G_pE\cap V_\alpha E\to\gr_p^{F_\alpha}V_\alpha E$ has kernel equal to $G_pE\cap V_\alpha E\cap\big(\sum_{k\geq0}(G_{p-1+k}E\cap V_{\alpha-k}E)\big)$, according to \eqref{eq:Falphap}. The latter space is contained in $G_pE\cap (G_{p-1}E\cap V_\alpha E+V_{\alpha -1}E)$, that is, in \hbox{$(G_{p-1}E\cap V_\alpha E)+(G_pE\cap V_{\alpha -1}E)$}; but clearly the converse inclusion is also true. We thus have an inclusion
\begin{equation}\label{eq:injGVgrFV}
\frac{G_pE\cap V_\alpha E}{(G_{p-1}E\cap V_\alpha E)+(G_pE\cap V_{\alpha-1}E)}\hto \gr_p^{F_\alpha}V_\alpha E.
\end{equation}
On the other hand,
\[
G_pE\cap V_\alpha E\cap\Big(\sum_{k\geq1}G_{p+k}E\cap V_{\alpha-k}E\Big)\subset G_pE\cap V_{\alpha-1}E\subset F_{\alpha+p-1}E,
\]
so \eqref{eq:injGVgrFV} is a direct summand in $\gr_p^{F_\alpha}V_\alpha E$, and one can continue to get the first expression in \eqref{eq:grFalphaVE}. For the second expression, we use \eqref{eq:multvjGV}.
\end{proof}

\begin{lemme}\label{lem:FValphav}
For $\alpha\in[0,1)$, $p\geq0$, near a point of $P_\red$, the following holds:
\begin{enumerate}
\item\label{lem:FValphav1}
The natural morphism, induced by the inclusion of each summand in $\cO_X(*P_\red)[v]$:
\begin{starequation}\label{eq:FValphav}
\bigoplus_{\substack{\bma\geq0\\|\bma|\leq p}}\cO_{I(\bma)}x^{-[\alpha\bme]-\bf1}x^{-\bma}Q_{\bma,p-|\bma|,\alpha}(v/x^\bme)\to \gr_p^GV_\alpha E^{vf}
\end{starequation}%
is an isomorphism.

\item\label{lem:FValphav2}
For each $i=1,\dots,\ell$, the morphism $\partial_i:\gr_p^GV_\alpha E^{vf}\to\gr_{p+1}^GV_\alpha E^{vf}$ induced by $-\partial_{x_i}/e_i$ is given by
\bgroup\multlinegap0pt
\begin{multline*}
\partial_i h_{\bma,p-|\bma|,\alpha}(x_{I(\bma)},y')x^{-[\alpha\bme]-\bf1}x^{-\bma}Q_{\bma,p-|\bma|,\alpha}(v/x^\bme)\\
=
\begin{cases}
h_{\bma,p-|\bma|,\alpha}(x_{I(\bma)},y')x^{-[\alpha\bme]-\bf1}x^{-(\bma+1_i)}Q_{\bma+1_i,p+1-|\bma+1_i|,\alpha}(v/x^\bme)&\text{if }i\notin I(\bma),\\[5pt]
h_{\bma,p-|\bma|,\alpha}(0_i,y')x^{-[\alpha\bme]-\bf1}x^{-(\bma+1_i)}Q_{\bma+1_i,p+1-|\bma+1_i|,\alpha}(v/x^\bme)\\
\hspace*{1.5cm}{}+h^{(i)}_{\bma,p-|\bma|,\alpha}(x_{I(\bma)},y')x^{-[\alpha\bme]-\bf1}x^{-\bma}Q_{\bma,p+1-|\bma|,\alpha}(v/x^\bme)&\text{if }i\in I(\bma),
\end{cases}
\end{multline*}
\egroup
where, for $i\in I(\bma)$, we set
\[
h_{\bma,p-|\bma|,\alpha}(x_{I(\bma)},y')=h_{\bma,p-|\bma|,\alpha}(0_i,y')+x_ih^{(i)}_{\bma,p-|\bma|,\alpha}(x_{I(\bma)},y'),
\]
and $0_i$ means that $x_i$ is set to be $0$ in $x_{I(\bma)}$.
\end{enumerate}
\end{lemme}

\begin{proof}
The first point follows from \cite[Lemma 4.11]{Bibi96a}. For the second point we have, modulo $G_pE^{vf}\cap\nobreak V_\alpha E^{vf}$,
\begin{multline*}
\partial_i h_{\bma,p-|\bma|,\alpha}(x_{I(\bma)},y')x^{-[\alpha\bme]-\bf1}x^{-\bma}Q_{\bma,p-|\bma|,\alpha}(v/x^\bme)\\
=h_{\bma,p-|\bma|,\alpha}(x_{I(\bma)},y')x^{-[\alpha\bme]-\bf1}x^{-(\bma+1_i)}Q_{\bma+1_i,p-|\bma|,\alpha}(v/x^\bme).
\end{multline*}
However, this is possibly not written in the form above if $i\in I(\bma)$ (\ie $a_i=0$) and we modify this expression as indicated in the statement to obtain, modulo $G_pE^{vf}\cap\nobreak V_\alpha E^{vf}$, the desired formula.
\end{proof}

\Subsection{The filtered relative de~Rham complex and the~rescaled Yu complex}\label{subsec:deRhamYu}
We consider the relative de~Rham complex $\DR_{X\times\Afu_v/\Afu_v}\ccE$ which is nothing but the complex of $\cO_{\Afu_v}$-modules
\[
0\to\ccE\To{\rd+v\rd f}\Omega^1_X[v]\otimes\ccE\to\cdots
\]
and the action of $\partial_v$ by $\partial/\partial v+f$ induces a $\CC[v]\langle\partial_v\rangle$-structure on each term compatible with the differentials.

We filter this complex as usual by subcomplexes of $\CC[v]$-modules:
\[
F_{\alpha+p}\DR_{X\times\Afu_v/\Afu_v}\ccE=\{F_{\alpha+p}\ccE\To{\rd+v\rd f}\Omega^1_X[v]\otimes F_{\alpha+p+1}\ccE\to\cdots\}.
\]
As usual we set $F_\alpha^p=F_{\alpha-p}$. The action of $\partial_v$ on $\DR_{X\times\Afu_v/\Afu_v}\ccE$ induces a morphism
\[
\partial_v:F_{\alpha+p}\DR_{X\times\Afu_v/\Afu_v}\ccE\to F_{\alpha+p+1}\DR_{X\times\Afu_v/\Afu_v}\ccE.
\]
We will use the following notation, as in \cite{E-S-Y13}:
\[
\Omega^k_{X\times\Afu_v/\Afu_v}(\log\ccD)([(\alpha+j)\ccP])_+=\begin{cases}
0&\text{if }j<0,\\
\Omega^k_{X\times\Afu_v/\Afu_v}(\log\ccD)([(\alpha+j)\ccP])&\text{if }j\geq0.
\end{cases}
\]
We define the \emph{rescaled Yu complex} as being the filtered complex ($\alpha\in[0,1)$ and $p\in\ZZ$):
\bgroup
\multlinegap0pt
\begin{multline*}
F_{\alpha+p}^\Yu\DR_{X\times\Afu_v/\Afu_v}\ccE\\
:=\cO_{X\times\Afu_v}([(\alpha+p)\ccP])_+\To{\rd+v\rd f}\Omega^1_{X\times\Afu_v/\Afu_v}(\log\ccD)([(\alpha+p+1)\ccP])_+\to\cdots
\end{multline*}
\egroup
which is a complex of $\cO_{\Afu_v}$-modules. The connection $\partial/\partial v+f$ induces a morphism $\nabla_{\partial_v}:F_{\alpha+p}^\Yu\DR_{X\times\Afu_v/\Afu_v}\ccE\to F_{\alpha+p+1}^\Yu\DR_{X\times\Afu_v/\Afu_v}\ccE$.

\begin{proposition}\label{prop:FFYu}
The natural morphism
\[
F_{\alpha+p}^\Yu\DR_{X\times\Afu_v/\Afu_v}\ccE\to F_{\alpha+p}\DR_{X\times\Afu_v/\Afu_v}\ccE
\]
is a quasi-isomorphism for each $\alpha\!\in\![0,1)$ and $p\!\in\!\ZZ$ compatible with the action of~$\partial_v$.
\end{proposition}

\begin{proof}
The existence of a natural morphism follows from Lemma \ref{lem:Fexpressionv}. The compatibility with respect to the action of $\partial_v$ is then clear. The proof is then similar to that of \cite[Prop.\,1.7.4]{E-S-Y13}. We note that, away from $\ccP_\red$, we use that the morphism
\begin{equation}\label{eq:awayP}
(\Omega^\bbullet_{X\times\Afu_v/\Afu_v}(\log\ccH),\rd+v\rd f)\to(\Omega^\bbullet_{X\times\Afu_v/\Afu_v}(*\ccH),\rd+v\rd f),
\end{equation}
is a filtered quasi-isomorphism. Here the analytic version of $\ccE$ is needed in order to write $\rd+v\rd f=e^{-vf}\circ\rd\circ e^{vf}$.
\end{proof}

\subsection{Proof of Theorem \ref{th:4main}}\label{subsec:pf4main}

We consider the complex $(\Omega_f^\cbbullet(\alpha)[v],\rd+v\rd f)$. We have a natural connection
\[
\nabla_{\partial_v}:(\Omega_f^\cbbullet(\alpha)[v],\rd+v\rd f)\to(\Omega_f^\cbbullet(\alpha+1)[v],\rd+v\rd f)
\]
induced by the action of $f+\partial/\partial v$ on each term of the complex.

\begin{lemme}\label{lem:omegafV}
For $\alpha\in[0,1)$, there is a natural filtered morphism
\[
\big(\cO_{X\times\Afu_v}\otimes_{\cO_X[v]}\Omega_f^\cbbullet(\alpha)[v],\rd+v\rd f,\sigma^{\geq p}\big)\to(\DR_{X\times\Afu_v/\Afu_v}V_\alpha\ccE,F_\alpha^p\DR_{X\times\Afu_v/\Afu_v}V_\alpha\ccE)
\]
which makes the following diagram commutative:
\[
\xymatrix{
\big(\cO_{X\times\Afu_v}\otimes_{\cO_X[v]}\Omega_f^\cbbullet(\alpha)[v],\rd+v\rd f\big)\ar[d]_{\nabla_{\partial_v}}\ar[r]&\DR_{X\times\Afu_v/\Afu_v}V_\alpha\ccE\ar[d]^{\nabla_{\partial_v}}\\
\big(\cO_{X\times\Afu_v}\otimes_{\cO_X[v]}\Omega_f^\cbbullet(\alpha+1)[v],\rd+v\rd f\big)\ar[r]&\DR_{X\times\Afu_v/\Afu_v}V_{\alpha+1}\ccE
}
\]
\end{lemme}

\begin{proof}
Once the morphism is defined, the commutativity of the diagram is straightforward: let $\rd$ denote the differential with respect to $X$ and $\rd_v$ that with respect to $v$; the verification reduces to checking that $e^{-vf}\circ\rd\circ e^{vf}$ commutes with $e^{-vf}\circ\rd_v\circ e^{vf}$, a statement which follow from the commutation of $\rd$ with $\rd_v$.

Away from $P_\red$, the morphism is given by \eqref{eq:awayP}. At a point of $P_\red$, we will use the algebraic version $E$ of $\ccE$ for simplicity. For each $k\geq0$, we have a natural inclusion
\[
(F_0\cO(*P_\red)\big)([\alpha P])(*H)v^k\revf\subset V_\alpha E.
\]
Indeed, it is enough to prove the inclusion
\[
x^{k\bme}(F_0\cO(*P_\red)\big)([\alpha P])(*H)\cdot(v/x^\bme)^k\revf\subset V_\alpha E
\]
and then the inclusion
\[
x^{k\bme}(F_0\cO(*P_\red)\big)([\alpha P])(*H)Q_{{\bf0},j,\alpha}(v/x^{\bme})\subset V_\alpha E\quad\text{with }P_{{\bf0},j,\alpha}(s)=(s+\alpha)^j,
\]
by expressing $(v/x^{\bme})^k$ in terms of the $Q_{{\bf0},j,\alpha}(v/x^{\bme})$ with $j\leq k$. The assertion is then clear by taking the term with $\bma=0$ in the expression of Lemma \ref{lem:Vfiltration}. We thus obtain the desired morphism.

In order to prove that it is filtered, we note that for each $k\geq0$, the natural inclusion morphism $\Omega_f^k(\alpha)[v]\to\Omega^k_X\otimes_{\cO_X}E$ factorizes through the subsheaf $\Omega^k_X\otimes_{\cO_X}F_\alpha V_\alpha E$. Indeed, according to \eqref{eq:FalphaValpha}, we have $F_\alpha V_\alpha E=F_\alpha E=F_0\cO_X(*D)([\alpha P])[v]\cdot\revf$, so we are reduced to proving the inclusion $\Omega_f^k(\alpha)\subset\Omega_X^k\otimes F_0\cO_X(*D)([\alpha P])$. This is clear away from $P_\red$ since this reduces to $\Omega_X^k(\log H)\subset \Omega_X^k\otimes F_0\cO_X(*H)$. In the local setting of \S\ref{subsec:setting} near a point of $P_\red$, the conclusion follows from \cite[Formula (1.3.1)]{E-S-Y13} for $\alpha=0$, and the same formula multiplied by $x^{-[\alpha\bme]}$ if $\alpha\in(0,1)$.
\end{proof}

We will show Theorem \ref{th:4main} with the filtration $F_\alpha^\cbbullet$ introduced in Definition \ref{def:Fexpressionv}. That this is the filtration $F_\alpha^{\irr,\bbullet}$ will be shown in Theorem \ref{th:FFirr}. Near a point of $(X\moins P_\red)\times\CC_v$ we can write $\rd+v\rd f=e^{-vf}\circ\rd\circ e^{vf}$ to reduce the statement to the standard result proved by Deligne \cite{Deligne70}.

We will thus focus on $P_\red\times\CC_v$, and it will be enough to consider the $v$-algebraic version of the statement. It is also enough to prove that for each $p\geq0$, the $p$th graded morphism is a quasi-isomorphism. We are thus lead to proving that for $p\geq0$ the following vertical morphism is a quasi-isomorphism:
\begin{equation}\label{eq:grOmegafV}
\begin{array}{c}
\xymatrix@C=.6cm{
0\ar[r]&\Omega_f^p(\alpha)[v]\ar[d]\ar[r]&0\ar[d]\ar[r]&\cdots\\
0\ar[r]&F_\alpha V_\alpha E\otimes\Omega^p\ar[r]&\gr_1^{F_\alpha}V_\alpha E \otimes\Omega^{p+1}\ar[r]&\cdots
}
\end{array}
\end{equation}
Since the question is local, we can treat separately the variables $x$ and $y$, and the main problem remains the case of the $x$ variables, so that we will assume $H=\emptyset$. We will use the computations of \S\ref{subsec:computFV}, from which we keep the notation.

\begin{lemme}\label{lem:QIgrpV}
Near a point of $P_\red$, for $q\in\ZZ$ and $\alpha\in[0,1)$, the relative de~Rham complex
\[
0\to\gr_q^GV_\alpha E^{vf}\to\gr_{q+1}^GV_\alpha E^{vf}\otimes\Omega^1\to\cdots\to\gr_{q+n}^GV_\alpha E^{vf}\otimes\Omega^n\to0
\]
has zero cohomology in degrees $\geq-q+1$ (recall that $\gr_q^GV_\alpha E^{vf}=0$ for $q<0$).
\end{lemme}

\begin{proof}[Sketch of proof]
We will forget the variables $y'$ and work with the variables $x\in\CC^\ell$, so we will replace $n$ with $\ell$ in the de~Rham complex above. We then note that this complex is the simple complex associated with the hypercubical complex built on the cube in $\RR^\ell$ with vertices $\epsilong\in\{0,1\}^\ell$, whose vertex at $\epsilong$ is $\gr_{q+|\epsilong|}^GV_\alpha E^{vf}$ and whose arrows $(\epsilon_i=0)\to(\epsilon_i=1)$ are the derivatives $\partial_{x_i}$. We may as well replace the arrow~$\partial_{x_i}$ with $\partial_i=-\partial_{x_i}/e_i$.

The formula of Lemma \ref{lem:FValphav}\eqref{lem:FValphav2} shows that, if $\epsilon_i=0$, the arrow $\partial_i:\epsilong\to\epsilong+1_i$ is injective, with cokernel identified with
\[
\bigoplus_{\substack{\bma'\geq0,\,a'_i=0\\|\bma'|=q+1 }}\cO_{I(\bma')}x^{-[\alpha\bme]-\bf1}x^{-\bma'}Q_{\bma',0,\alpha}(v/x^e).
\]
We use the convention that a sum indexed by the empty set is zero, a case which occurs if $q+1<0$.
\begin{itemize}
\item
If $\ell=1$, we only need to consider the case $q\geq0$. The cokernel of $\partial_1$ is then equal to zero, showing that $\partial_1$ is bijective in this case, which implies the desired assertion.

\item
If $\ell\geq2$, we replace (with a shift) the hypercubical $\ell$-complex with the $(\ell-\nobreak1)$\nobreakdash-complex made of the cokernels of the injective arrows $\partial_1$, and the formula for the induced arrows $\partial_2,\dots,\partial_\ell$ is then that of the case $i\notin I(\bma)$ in the formula of Lemma \ref{lem:FValphav}\eqref{lem:FValphav2}. Now, the maps induced by $\partial_2$ are injective, with cokernel
\[
\bigoplus_{\substack{\bma''\geq0,\,a''_1=a''_2=0\\|\bma''|=q+2}}\cO_{I(\bma'')}x^{-[\alpha\bme]-\bf1}x^{-\bma''}Q_{\bma'',0,\alpha}(v/x^e),\quad\text{etc.}\qedhere
\]
\end{itemize}
\end{proof}

From Lemma \ref{lem:QIgrpV} we conclude that for $\alpha\in[0,1)$ and each $p\geq0$, the de~Rham complex
\refstepcounter{equation}\label{eq:grFVOmega}
\[\tag*{(\ref{eq:grFVOmega})$_\alpha$}
0\to\cdots\to0\to\gr_0^GV_\alpha E^{vf}\otimes\Omega^p\to\gr_1^GV_\alpha E^{vf}\otimes\Omega^{p+1}\to\cdots
\]
has zero cohomology in degrees $\geq p+1$, while the de~Rham complex
\[\tag*{(\ref{eq:grFVOmega})$_{\alpha-1}$}
0\to\cdots\to0=\gr_0^GV_{\alpha-1} E^{vf}\otimes\Omega^p\to\gr_1^GV_{\alpha-1} E^{vf}\otimes\Omega^{p+1}\to\cdots
\]
has zero cohomology in degrees $\geq p+2$ since $\gr_k^GV_{\alpha-1}E^{vf}\simeq\gr_{k-1}^GV_\alpha E^{vf}$, according to \eqref{eq:multvjGV}. Therefore, the quotient complex $\eqref{eq:grFVOmega}_\alpha/\eqref{eq:grFVOmega}_{\alpha-1}$ has zero cohomology in degrees $\geq p+1$. It follows then from Lemma \ref{lem:grFalphaVE} that the bottom line of \eqref{eq:grOmegafV} has zero cohomology in degrees $\geq p+1$. It remains to identify the degree $p$ cohomology of this bottom line. As noted above, we have
\[
F_\alpha V_\alpha E^{vf}=F_\alpha E^{vf}=F_0\cO(*P_\red)([\alpha P])[v]\revf,
\]
so the cohomology consists of sections of $F_0\cO(*P_\red)([\alpha P])[v]\otimes\Omega^p$ whose image by $\rd+v\rd f$ belong to $F_0\cO(*P_\red)([\alpha P])[v]\otimes\Omega^{p+1}$. This cohomology is then contained in $\Omega^p(\log P_\red)([\alpha P])[v]$, according to Lemma \ref{lem:caractOmegalog} below, and it is then easy to identify it with $\Omega^p_f(\alpha)[v]$.\qed

\begin{lemme}\label{lem:caractOmegalog}
For $k\geq0$, a section of $F_0\cO(*P_\red)\Omega^k$ belongs to $\Omega^k(\log P_\red)$ if and only if its exterior product by $\sum_{i=1}^\ell e_i\rd x_i/x_i$ belongs to $F_0\cO(*P_\red)\Omega^{k+1}$.\qed
\end{lemme}

\subsection{Some properties of the filtration \texorpdfstring{$F_\alpha^{\protect\cbbullet}\cH^k_v$}{FH}}

Recall that the $\cD_{\Afu_v}$-module $\cH^k_v$ is defined in \S\ref{subsec:rescaling}. For $\alpha\in[0,1)$, we denote by $V_\alpha\cH^k_v$ the free $\CC[v]$-lattice of~$\cH^k_v$ on which the connection $\nabla$ induced by the $\cD_{\Afu_v}$-module structure has a simple pole, with residue as in Theorem \ref{th:2main}\eqref{th:2main2}. This is also the part of indices in $[0,1)$ of the Kashiwara-Malgrange $V$-filtration of $\cH^k_v$, which exists since it is a holonomic $\cD_{\Afu_v}$-module.

By a standard result on the strictness of the Kashiwara-Malgrange $V$-filtration with respect to proper push-forward, we have:
\[
V_\alpha\cH^k_v=\image\Big[R^kq_+(\DR_{X\times\Afu_v/\Afu_v}V_\alpha\ccE)\to R^kq_+(\DR_{X\times\Afu_v/\Afu_v}\ccE)\Big]
\]
and the latter morphism is injective. We obtain, as a consequence of Proposition \ref{prop:FMHS}:

\begin{corollaire}[{of \cite[Prop.\,3.3.17]{MSaito86}}]\label{cor:propertiesF}
For each $k,\alpha,p$, $F_\alpha^\bbullet\cH^k_v$ satisfies the properties \cite[(3.2.1)]{MSaito86}.\qed
\end{corollaire}

\pagebreak[2]
Let us consider the restriction $j^*F_\alpha^\bbullet\cH_v^k$ ($j:\GGm\hto\Afu_v$).

\begin{corollaire}\label{cor:Floc}
For each $\alpha\in[0,1)$ and $p\in\ZZ$, we have an isomorphism of $\cO_{\GGm}$\nobreakdash-modules:
\[
j^*F_\alpha^p\cH^k_v\simeq\cO_{\GGm}\otimes_\CC F_\alpha^{\Yu,p} H^k_\dR(U,\nabla).
\]
In particular, $j^*F_\alpha^p\cH^k_v=\cH^k_{|\GGm}$ for $p\leq0$ and $j^*F_\alpha^p\cH^k_v=0$ for $p>k$.
\end{corollaire}

\begin{proof}
The first part follows from Lemma \ref{lem:Fexpressionv}\eqref{lem:Fexpressionuv} and the second part follows from the property $\gr_{F_\alpha^{\Yu}}^pH^k_\dR(U,\nabla)=0$ for $p\notin[0,k]$, which is a consequence of \cite[Cor.\,1.5.6]{E-S-Y13}.
\end{proof}

Recall that the irregular Hodge numbers $h_\alpha^{p,q}(f)$ are defined by \eqref{eq:irregHodgenumber}. As a consequence of Corollary \ref{cor:Floc} we have
\[
h_\alpha^{p,q}(f)=\rk\gr_{F_\alpha}^p\!\!j_*\cH^{p+q}_v.
\]

\begin{corollaire}\label{cor:F0=V}
For $\alpha\in[0,1)$, we have $F_\alpha^0\cH^k_v\supset V_\alpha\cH^k_v$.
\end{corollaire}

\begin{proof}
We have seen that both $\cO_{\Afu_v}$-modules coincide with $\cH^k_v[v^{-1}]$ after tensoring with $\cO_{\Afu_v}[v^{-1}]$ (by Corollary \ref{cor:Floc} for the first one, and by a standard property of the $V$-filtration for the second one). Hence for any $m\in V_\alpha\cH^k_v$ there exists $\ell\geq0$ such that $v^\ell m\in F_\alpha^0\cH^k_v$. Let $p\in\ZZ$ be such that $m\in F_\alpha^p\cH^k_v$. Corollary \ref{cor:propertiesF} implies that Property \cite[(3.2.1.1)]{MSaito86} holds for the filtration $F_\alpha^\bbullet\cH^k_v$, and thus the morphism $v^\ell:(F_\alpha^q\cH^k_v\cap V_\alpha\cH^k_v)\to (F_\alpha^q\cH^k_v\cap V_{\alpha-\ell}\cH^k_v)$ is an isomorphism for each $q$. It follows that
\[
m\in F_\alpha^p\cH^k_v\cap V_\alpha\cH^k_v\text{ and }v^\ell m\in F_\alpha^0\cH^k_v\cap V_{\alpha-\ell}\cH^k_v\Longrightarrow m\in F_\alpha^0\cH^k_v\cap V_\alpha\cH^k_v,
\]
as was to be proved.
\end{proof}

\subsection{Nearby cycles and the monodromy filtration}\label{subsec:nearbymonodromy}
We now consider the functor $\psi_{v,\exp(-2\pi\ri\beta)}$ ($\beta\in[0,1)$). The result of \cite[Prop.\,3.3.17]{MSaito86} implies then that, for each $\beta\in[0,1)$, the filtration naturally induced by the $\QQ$-indexed filtration $F^\bbullet\cH^k_v$ on $\psi_{v,\exp(-2\pi\ri\beta)}\cH^k_v$ is equal to
\begin{equation}\label{eq:FpsicHalpha}
F^\bbullet\bH^k(X,\DR\psi_{v,\exp(-2\pi\ri\beta)}\ccE):=\bH^k(X,F^\bbullet\DR\psi_{v,\exp(-2\pi\ri\beta)}\ccE)
\end{equation}
and therefore has jumps at $\beta+\ZZ$ at most. It is then enough to consider the filtration induced by $F_\beta^\cbbullet\cH^k_v$ on $\psi_{v,\exp(-2\pi\ri\beta)}\cH^k_v$. Then, according to the previous results, we have
\[
F_\beta^p\psi_{v,\exp(-2\pi\ri\beta)}\cH^k_v=\begin{cases}
\psi_{v,\exp(-2\pi\ri\beta)}\cH^k_v&\text{if }p\leq0,\\
0&\text{if }p>k.
\end{cases}
\]

\begin{definition}\label{def:specF}
For $\alpha\in[0,1)$ and $k\geq0$, the \emph{spectral multiplicity function} is the function
\[
\ZZ\ni p\mto\mu_\alpha^k(p):=\dim\gr^p_{F_\alpha}\psi_v\cH^k_v:=\sum_{\beta\in[0,1)}\dim\gr^p_{F_\alpha}\psi_{v,\exp(-2\pi\ri\beta)}\cH^k_v.
\]
\end{definition}

\begin{lemme}\label{lem:munu}
For each $\alpha\in[0,1)$, $k\in\NN$ and $p\in\ZZ$, we have
\[
\mu_\alpha^k(p)=h_\alpha^{p,k-p}.
\]
In particular, $\mu_\alpha^k(p)=0$ for $p\notin[0,k]$.
\end{lemme}

\begin{proof}
For $\beta\in[0,1)$, we have an isomorphism (\cf Corollary \ref{cor:propertiesF}):
\begin{equation}\label{eq:vFV}
v:(F_\alpha^p\cH^k\cap V_\beta\cH^k)\isom (F_\alpha^p\cH^k\cap V_{\beta-1}\cH^k).
\end{equation}
Therefore,
\begin{align*}
\sum_{\beta\in[0,1)}\dim F_\alpha^p\psi_{v,\exp(-2\pi\ri\beta)}\cH^k_v&=\sum_{\beta\in[0,1)}\dim F_\alpha^p\gr_\beta^V\cH^k_v=\sum_{\beta\in(\alpha-1,\alpha]}\dim F_\alpha^p\gr_\beta^V\cH^k_v\\
&=\dim\frac{F_\alpha^p\cH^k_v\cap V_\alpha\cH^k_v}{F_\alpha^p\cH^k_v\cap V_{\alpha-1}\cH^k_v}\\[5pt]
&=\dim\frac{F_\alpha^p\cH^k_v\cap V_\alpha\cH^k_v}{v(F_\alpha^p\cH^k_v\cap V_\alpha\cH^k_v)}\quad\text{(Corollary \ref{cor:propertiesF})}.
\end{align*}
Since $V_\alpha\cH^k_v$ is $\cO_{\Afu_v}$-free for $\alpha\in[0,1)$, the $\cO_{\Afu_v}$-module $F_\alpha^p\cH^k_v\cap V_\alpha\cH^k_v$ is $\cO_{\Afu_v}$\nobreakdash-torsionfree, hence $\cO_{\Afu_v}$-free, and the latter term is equal to $\rk(F_\alpha^p\cH^k_v\cap V_\alpha\cH^k_v)$, hence to $\rk(F_\alpha^p\cH^k_v)[v^{-1}]$, that is, $\dim F_\alpha^{\Yu,p}H^k_\dR(U,\nabla)$, according to Corollary \ref{cor:Floc}. The result follows from \cite[Cor.\,1.4.8]{E-S-Y13}.
\end{proof}

\begin{proof}[Proof of Theorem \ref{th:3main}]
By Lemma \ref{lem:FF'v} and \eqref{eq:FgrV}, we can apply \cite[Th.\,5.3]{Bibi96a} to the filtration given by \eqref{eq:FpsicHalpha}. It remains to identify the latter with the irregular Hodge filtration. This follows from Theorem \ref{th:FFirr} below.
\end{proof}

\section{The \texorpdfstring{$\cD_{X\times\Afu_u}$}{DAu}-module \texorpdfstring{$\ccE^{f/u}(*\ccH)$}{EfuH}}\label{sec:Euf}

We now focus on the $u$-chart. In this section, we will consider the $\cDXu$-module $E^{f/u}(*H):=(\cO_X(*D)[u,u^{-1}],\rd+\rd(f/u))$ and we use the identification
\[
E^{f/u}(*H)=\cO_X(*D)[u,u^{-1}]\cdot\refu,
\]
which makes clear the twist of the $\cDXu$-structure. We will denote for short $E=E^{f/u}(*H)$.

\subsection{The Brieskorn lattice of the $\cDXu$-module $E^{f/u}(*H)$}\label{subsec:Brieskornlattice}
Let $F_\bbullet\cD_X$ denote the filtration of~$\cD_X$ by the order of differential operators, and consider the Rees ring $R_F\cD_X:=\bigoplus_kF_k\cD_X\cdot u^k$, which can be expressed in local coordinates as $\cO_X[u]\langle u\partial_x,u\partial_y,u\partial_{y'}\rangle$. It will be useful to extend it by adding the action of $u^2\partial u$. We obtain in this way a sheaf of rings $R_F\cD_X\langle u^2\partial_u\rangle$, that we will denote by $G_0\cDXu$. It is naturally filtered by the order with respect to the partials, a filtration that we denote by $F_\bbullet G_0\cDXu$.

\begin{remarque}\label{rem:uhb}
The Rees construction is the same as that used in \S\ref{sec:MTEf}. However the notation for the extra variable used here is not the same as in \S\ref{sec:MTEf} since it will not play the same role. We will use both in \S\ref{sec:relwithFirr}.
\end{remarque}

The \emph{Brieskorn lattice} $G_0E$ defined in \cite[\S1]{Bibi97b} is the $\cO_X(*P_\red)[u]$-module
\begin{equation}\label{eq:G0E}
G_0E:=\bigoplus_j(F_j\cO_X(*H))(*P_\red)\cdot u^j\refu,
\end{equation}
and we set, for each $p\in\ZZ$, $G_pE=u^{-p}G_0E$. Then $G_\bbullet E$ is an increasing filtration of $E$ indexed by $\ZZ$. Note that, if $\nablaf$ denotes the relative connection on~$E$ induced by the $\cD_X$-module structure, then $G_0E$ is preserved by~$u\nablaf $. It is also preserved by the action of $u^2\partial_u$. In other words, $G_0E$ is a $G_0\cDXu$-module. For example, if $H=\emptyset$, we have
\[
G_0E^{f/u}=\cO_X(*P_\red)[u]\cdot\refu.
\]
Using the Rees module notation, we can also write
\[
\big(G_0E,u\nabla)=\big((R_F\cO_X(*H))(*P_\red), u\rd+\rd f\big).
\]

\Subsection{The filtration \texorpdfstring{$F_{\alpha+\bbullet}G_0E^{f/u}(*H)$}{FE}}\label{subsec:FalphaG0}

Although the function $f/u$ does not extend as a map $X\times\CC_u\to\PP^1$, we can nevertheless adapt in a natural way the definition given in \cite[(1.6.1) \& (1.6.2)]{E-S-Y13} for the case of the map $f:X\to\PP^1$.

\begin{definition}[The filtration]\label{def:Fexpressionu}
For $\alpha\in[0,1)$ we set
\begin{align*}
F_{\alpha+p}G_0E^{f/u}&=F_p\cO_X(*P_\red)\big([(\alpha+p)P]\big)[u]\cdot\refu,\\
F_{\alpha+p}G_0E&=\sum_{q+q'\leq p}u^qF_q\cO_X(*H)\cdot F_{\alpha+q'}G_0E^{f/u}.
\end{align*}
\end{definition}

\begin{lemme}\label{lem:Fexpressionu}\mbox{}
For each $\alpha\in[0,1)$, the filtration $F_{\alpha+\bbullet}G_0E$ is an $F_\bbullet G_0\cDXu$-filtration which satisfies the following properties.
\begin{enumerate}
\item\label{lem:Fexpressionucomp}
$F_{\alpha+p_1}G_0E\subset F_{\beta+p_2}G_0E$ for all $p_1,p_2\in\ZZ$ and $\beta\in[0,1)$ such that $\alpha+p_1\leq\beta+p_2$. Moreover, $F_{\alpha+p}G_0E=0$ for $p<0$.
\item\label{lem:Fexpressionugood}
The filtration $F_{\alpha+\bbullet}G_0E$ satisfies
\[
F_{\alpha+p}G_0E=F_pG_0\cDXu\cdot F_\alpha G_0E\,;
\]
in particular, it is good with respect to $F_\bbullet G_0\cDXu$.
\end{enumerate}
\end{lemme}

\begin{proof}
Similar to that of Lemma \ref{lem:Fexpressionv}.
\end{proof}

We will give an expression of $F_{\alpha+p}G_0E$ in terms of the $V$-filtration. We set $G_pE=u^{-p}G_0E$ and we identify $E[u^{-1}]$ with $E^{vf}(*H)[v^{-1}]$, so that we can define the filtration $V_{\alpha+k}(E[u^{-1}])$ as being the filtration $V_{\alpha+k}(E^{vf}(*H)[v^{-1}])$ considered in \S\ref{subsec:VfiltEv}. Note that, since $v$ is invertible on $E^{vf}(*H)[v^{-1}]$ and since $V_\alpha(E^{vf}(*H)[v^{-1}])=V_\alpha(E^{vf}(*H))$ for $\alpha\in[0,1)$, we have
\[
V_{\alpha+k}(E[u^{-1}])=V_{\alpha+k}(E^{vf}(*H)[v^{-1}])=v^{-k}V_\alpha(E^{vf}(*H))=u^kV_\alpha E.
\]

For $\alpha\in[0,1)$, we set
\begin{equation}\label{eq:Fprimealphapu}
F'_{\alpha+p}G_0E:=u^p\CC[u](G_pE\cap V_\alpha E)=\CC[u](G_0E\cap u^pV_\alpha E),
\end{equation}
where the intersection is taken in $E[u^{-1}]$. This is an $F_\bbullet(G_0\cDXu)$-filtration since $u\partial_{x_i}$ sends $V_\alpha E$ to $uV_\alpha E$, and so does $u^2\partial_u=-\partial_v$.

\begin{lemme}\label{lem:FF'u}
For each $\alpha\in[0,1)$ and $p\in\ZZ$ we have
\[
F'_{\alpha+p}G_0E=F_{\alpha+p}G_0E.
\]
\end{lemme}

\begin{proof}
It will be similar to that of Lemma \ref{lem:FF'v}. In the neighbourhood of a point of $X\setminus P_\red$, Definition \ref{def:Fexpressionu} gives $F_{\alpha+p}G_0E=\sum_{q=0}^pu^qF_q\cO(*H)[u]\cdot\refu$, while a computation similar to that at the beginning of the proof of Lemma \ref{lem:FF'v} gives \hbox{$G_pE\cap V_\alpha E=\bigoplus_{0\leq q\leq p} F_q\cO(*H)u^{q-p}\cdot\refu$}, hence the result by multiplying the latter term by $u^p\CC[u]$.

In the neighbourhood of a point of $P_\red$, the inclusion $\supset$ is proved exactly as in Lemma \ref{lem:FF'v}. For the inclusion $\subset$, we use \eqref{eq:uniquevdvH} with $\beta=\alpha$. Using similarly $v^j\partial_v^j$ instead of $(v\partial_v)^j$, and replacing $v^j$ with $u^{-j}$, a term of $G_pE\cap V_\alpha E$ in the sum \eqref{eq:uniquevdvH} can be written as
\[
h_{\bma,\bmc,\lambda,j,\alpha}(x_{I(\bma)},y_{J(\bmc)},y')x^{-[(\alpha+j)\bme]-\bf1}x^{-\bma}y^{-\bmc-\bf1}u^{-j}\refu,
\]
with $j\leq q':=\lambda+|\bma|$ and $q:=|\bmc|\leq p-q'$. Note that $p-q-j\geq0$. So each term in $u^p(G_p\cap V_\alpha)$ is a sum of terms
\[
h_{\bma,\bmc,\lambda,j,\alpha}(x_{I(\bma)},y_{J(\bmc)},y')\cdot x^{-[(\alpha+q')\bme]-\bf1}x^{-(\bma-(q'-j)\bme)}u^{p-q-j}\cdot(u^qy^{-\bmc-\bf1})\refu
\]
which all belong to $F_{\alpha+p}G_0E$ (Definition \ref{def:Fexpressionu}), since $\bma\geq0$, $|\bma|\leq q'$ hence, according to Lemma \ref{lem:precision}, $|\bma-(q'-j)\bme|_+\leq j\leq q'$ for any $j$ such that $0\leq j\leq q'$.
\end{proof}

\begin{remarque}\label{rem:FF'u}
We conclude from the lemma that the filtration $F_{\alpha+p}\gr_0^GE$ induced by $F_{\alpha+p}G_0E$ is nothing but the filtration induced by $u^pV_\alpha(E[u^{-1}])=V_{\alpha+p}(E[u^{-1}])$. Indeed, recalling that $uG_0=G_{-1}$, we have
\bgroup
\multlinegap0pt
\begin{multline*}
\frac{\CC[u](G_0E\cap u^pV_\alpha(E[u^{-1}]))}{uG_0E\cap\big(\CC[u](G_0E\cap u^pV_\alpha(E[u^{-1}]))\big)}\\
=\frac{G_0E\cap u^pV_\alpha(E[u^{-1}])}{(uG_0E\cap u^pV_\alpha(E[u^{-1}]))+\big[G_0E\cap u^pV_\alpha(E[u^{-1}])\cap u\CC[u](G_0E\cap u^pV_\alpha(E[u^{-1}]))\big]}\\
=\frac{G_0E\cap u^pV_\alpha(E[u^{-1}])}{G_{-1}E\cap u^pV_\alpha(E[u^{-1}])}
\end{multline*}
\egroup
It follows that
\[
\frac{F_{\alpha+p}\gr_0^GE}{F_{<\alpha+p}\gr_0^GE}\underset{\sim}{\To{u^{-p}}}\gr_\alpha^V\gr_p^GE=\gr_p^G\gr_\alpha^VE^{vf}(*\ccH),
\]
and we conclude that $\gr_{\alpha+p}^F\gr_0^GE$ can be computed from data in the $v$-chart.
\end{remarque}

\subsection{Proof of Theorem \ref{th:5main}}\label{subsec:pf5main}
We have
\begin{multline*}
\DR_XE=\Big\{0\to\cO_X(*D)[u,u^{-1}]\To{\rd+\rd f/u}\\[-5pt]
\cdots\To{\rd+\rd f/u}\Omega^n_X(*D)[u,u^{-1}]\to0\Big\}.
\end{multline*}
It will be convenient to use the complex
\begin{multline*}
\ccDR_XE:=\Big\{0\to\cO_X(*D)[u,u^{-1}]\To{u\rd+\rd f}\\[-5pt]
\cdots\To{u\rd+\rd f}\Omega^n_X(*D)[u,u^{-1}]\to0\Big\}.
\end{multline*}
Both complexes are obviously isomorphic by multiplying the $k$th term of the first one by $u^k$, a morphism that we denote by $u^\cbbullet$.

The subcomplex $\DR_XG_0E$ of $\DR_XE$ is defined by
\begin{multline}\label{eq:DRG0}
\DR_XG_0E:=\Big\{0\to(R_F\cO_X(*H))(*P_\red)\To{\rd+\rd f/u}\\
\cdots\To{\rd+\rd f/u}\Omega^n_X\otimes(u^{-n}R_F\cO_X(*H))(*P_\red)\to0\Big\}.
\end{multline}
Similarly, the subcomplex $\ccDR_XG_0E$ of $\ccDR_XE$ is defined by
\begin{multline}\label{eq:ccDRG0}
\ccDR_XG_0E:=\Big\{0\to(R_F\cO_X(*H))(*P_\red)\To{u\rd+\rd f}\\
\cdots\To{u\rd+\rd f}\Omega^n_X\otimes(R_F\cO_X(*H))(*P_\red)\to0\Big\}.
\end{multline}
For example, if $H=\emptyset$, we obtain the complexes
\begin{multline*}
\Big\{0\ra\cO_X(*P_\red)[u]\To{\rd+\rd f/u}\cdots\To{\rd+\rd f/u}u^{-n}\Omega^n_X(*P_\red)[u]\ra0\Big\}\\
\overset{\;\,\textstyle u^\cbbullet}\simeq\Big\{0\ra\cO_X(*P_\red)[u]\To{u\rd+\rd f}\cdots\To{u\rd+\rd f}\Omega^n_X(*P_\red)[u]\ra0\Big\}.
\end{multline*}
The relative de~Rham complex $\ccDR_XG_0E$ is naturally filtered by
\begin{equation}\label{eq:FDRu}
F_{\alpha+p}\ccDR_XG_0E:=\Big\{0\to F_{\alpha+p}G_0\ccE\to\Omega^1_X\otimes F_{\alpha+p+1}G_0\ccE\to\cdots\Big\}.
\end{equation}

The proof of Theorem \ref{th:5main} is obtained by adapting the proofs of \cite[Cor.\,1.4.5 \& Prop.\,1.7.4]{E-S-Y13} to the present situation. We add the parameter $u$ and we consider the $u$-connection $u\rd+\rd f$. The natural inclusion morphism $\Omega_f^k(\alpha)[u]\to\Omega^k_X\otimes_{\cO_X}E$ factorizes through $\Omega^k_X\otimes_{\cO_X}F_\alpha G_0E$ since $F_\alpha G_0E=F_0\cO_X(*D)([\alpha P])[u]\cdot\refu$, and this shows that the filtered morphism of Theorem \ref{th:5main} is well-defined. To prove that it is a filtered quasi-isomorphism, we note that, for the analogue of \cite[Prop.\,1.7.4]{E-S-Y13} the ultimate step of the proof, after grading the complexes, is the same as in \loccit, since the graded differential is $\rd\log x^{-\bme}$ in both cases. Similarly, the arguments of \cite{Yu12} used in the proof of \cite[Prop.\,1.4.2\,\&\,Cor.\,1.4.5]{E-S-Y13} reduce the problem to proving a quasi-isomorphism with a graded differential which does not depend on~$u$.\qed

\subsection{Push-forward of the Brieskorn lattice}\label{subsec:qEu}
Let us consider the push-for\-ward~$\cH^k_u$ as obtained in the chart~$\Afu_u$, that is,
\[
\cH^k_u=R^kq_*\DR_{X\times\Afu_u/\Afu_u}(\ccE)\overset{\,\textstyle u^\cbbullet}\simeq R^kq_*\ccDR_{X\times\Afu_u/\Afu_u}(\ccE).
\]
We set $\bH^k_u:=\Gamma(\Afu_u,\cH^k_u)$, so that the above isomorphism becomes
\begin{equation}\label{eq:isouv}
\bH^k_u\simeq
\bH^k\big(X,(\Omega_X^\cbbullet(*D)[u,u^{-1}],u\rd+\rd f)\big)
\end{equation}
We obviously have $\bH^k_u=\bH^k_u[u^{-1}]=\bH^k_v[v^{-1}]$, and it is a free $\CC[u,u^{-1}]$-module with connection.

Let us consider the $\CC[u]$-module
\begin{equation}\label{eq:BrieskornHkv}
\begin{split}
G_0\bH^k_u&=\bH^k\big(X,(\Omega_X^\cbbullet\otimes_{\cO_X}\!\!(u^{-\cbbullet}R_F\cO_X(*H))(*P_\red),\rd+\rd f/u)\big)\\
&\simeq\bH^k\big(X,(\Omega_X^\cbbullet\otimes_{\cO_X}\!\!(R_F\cO_X(*H))(*P_\red),u\rd+\rd f)\big)\\
&=\bH^k\big(X,\ccDR_{X\times\Afu_u/\Afu_u}(G_0\ccE)\big)\quad\text{according to \eqref{eq:ccDRG0}.}
\end{split}
\end{equation}
For example, if $H=\emptyset$, we have
\begin{align*}
G_0\ccE^{f/u}&=\cO_X(*P_\red)[u]\cdot\refu,\\
\tag*{and}G_0\bH^k_u&=\bH^k\big(X,(u^{-\cbbullet}\Omega_X^\cbbullet(*P_\red)[u],\rd+\rd f/u)\big)\\
&\simeq\bH^k\big(X,(\Omega_X^\cbbullet(*P_\red)[u],u\rd+\rd f)\big).
\end{align*}
According to \cite{MSaito87}, we can apply the proposition in \cite[\S1]{Bibi97b} to $\ccE[v^{-1}]$ and get:

\begin{proposition}\label{prop:G0free}
For each $k$, $G_0\bH^k_u$ is a free $\CC[u]$-module, hence is a $\CC[u]$-lattice of $\bH^k_u$, and we have $\CC[u,u^{-1}]\otimes_{\CC[u]}\nobreak G_0\bH^k_u\simeq\bH^k_u=\bH^k_v[v^{-1}]$ by the isomorphism \eqref{eq:isouv}.\qed
\end{proposition}

\begin{remarque}[Stability under $u^2\partial_u$]\label{rem:actionud2u}
The natural action of $u^2\partial_u$ on $G_0E$ induces an action on $u^{-k}\Omega_X^k\otimes G_0E$ which defines an action of $u^2\partial_u$ on the complex $\DR_X(G_0E)$, hence on its cohomology $G_0\bH^k_u$ (equivalently, a shifted action by $u^2\partial_u-ku$ on $\Omega_X^k\otimes\nobreak G_0E$, hence on $\ccDR_X(G_0X)$ and on its cohomology). In other words, the action of $\partial_u$ on $\bH^k_u$ has a pole of order at most two when restricted to $G_0\bH^k_u$.
\end{remarque}

For each $k$, we have a natural morphism (\cf \eqref{eq:FDRu})
\begin{equation}\label{eq:RkFu}
\bH^k\big(X,F_{\alpha+p}\ccDR_XG_0E\big)\to \bH^k\big(X,\ccDR_XG_0E\big)=:G_0\bH^k_u,
\end{equation}
whose image is denoted by $F_{\alpha+p}G_0\bH^k_u$. The source of this morphism is a $\CC[u]$\nobreakdash-module of finite type because $q$ is proper and the terms of the complex \eqref{eq:FDRu} are $\cO_X[u]$-coherent. As already mentioned after Theorem \ref{th:5main}, \eqref{eq:RkFu} is injective for each $k$. The filtered $G_0\CC[u]\langle\partial_u\rangle$-module $(G_0\bH^k_u,F_{\alpha+\bbullet}G_0\bH^k_u)$ is the $(k-\dim X)$th push-forward of the filtered $G_0\cDXu$-module $(E,F_{\alpha+\bbullet}E)$.

\subsection{The case of cohomologically tame functions on affine varieties}
\let\Afu\oldAfu
In this subsection we use the Zariski topology on $U,X$ and $X\times\Afu_u$. We still denote by $\ccE$ the $\cD_{X\times\Afu_u}$-module $\cO_{X\times\Afu_u}(*\ccD)\cdot\refu$ and we make the abuse of identifying it with $\cO_X(*D)[u,u^{-1}]\cdot\refu$ (where $X$ has its Zariski topology).

Assume that $U$ is affine and that $f:U\to\Afu$ is a cohomologically tame function, in the sense of \cite[\S8]{Bibi96bb} (\cf also \cite[Prop\ptbl14.13.3(2)]{Katz90} for a weaker condition). In particular, $f$ has only isolated critical points. Then $H^k_{\dR}(U,\rd+\rd f)=0$ unless $k=n:=\dim X$, and $\dim H^n_{\dR}(U,\rd+\rd f)$ is equal to the sum of the Milnor numbers at the critical points. The Brieskorn lattice $G_0(f)$ is defined as $G_0(f):=\Omega^n(U)[u]\big/(u\rd+\nobreak\rd f)\Omega^{n-1}(U)[u]$.

\begin{proposition}
Under this tameness assumption on $f$, the natural morphism of complexes
\[
\bR q_*\ccDR_{X\times\Afu_u/\Afu_u}(G_0\ccE)\to\bR q_*(\Omega^\cbbullet_{U\times\Afu_u/\Afu_u},u\rd+\rd f)=(\Omega^\cbbullet(U)[u],u\rd+\rd f)
\]
is a quasi-isomorphism, from which one deduces, through $\bH^n(u^\cbbullet)$, an equality $G_0\bH^n_u=u^{-n}G_0(f)$ in
\[
\bH^n_u\simeq\Omega^n(U)[u,u^{-1}]\big/(u\rd+\rd f)\Omega^{n-1}(U)[u,u^{-1}].
\]
\end{proposition}

\begin{proof}
The natural morphism is induced by
\[
(F_k\cO_X(*H))(*P_\red)\to (F_k\cO_X(*H))(*D)=\cO_X(*D)=j_*\cO_U
\]
(where $j_*$ is taken here in the Zariski topology). Through this morphism, $\bH^k(u^\cbbullet)$ corresponds to $u^\cbbullet$ termwise on the right-hand complex. Since $H^k_{\dR}(U,\rd+\rd f)=0$ unless $k=\dim X=n$, we also have $\bH^k_u=0$ unless $k=n$, and since the $k$th cohomology of the left-hand complex is contained in $\bH^k_u$, we conclude that the left-hand complex has cohomology in degree $n$ at most. We therefore obtain a morphism
\begin{equation}\label{eq:G0G0}
G_0\bH^n_u\to u^{-n}G_0(f),
\end{equation}
whose localization with respect to $u$ is an isomorphism, because $R_F\cO_X(*H)[u^{-1}]=\cO_X(*H)[u,u^{-1}]$ and thus $(R_F\cO_X(*H))(*P_\red)[u^{-1}]=\cO_X(*D)[u,u^{-1}]$, hence
\begin{align*}
\bH^k_u&=\bH^k\big(X,(\Omega^\cbbullet_X(*D)[u,u^{-1}],\rd +\rd f/u)\big)\\
&=\bH^k\big(U,(\Omega^\cbbullet_U[u,u^{-1}],\rd +\rd f/u)\big)\\
&=H^k(\Omega^\cbbullet(U)[u,u^{-1}],\rd +\rd f/u) \quad\text{($U$ affine).}
\end{align*}
Both terms of \eqref{eq:G0G0} are $\CC[u]$ free of the same rank, hence the morphism \eqref{eq:G0G0} is injective, and we may regard it as an inclusion in $\bH^n_u$ through the previous identification. The conclusion follows from the lemma below.
\end{proof}

\begin{lemme}\label{lem:G0G0}
The morphism \eqref{eq:G0G0} is an isomorphism, in other words, $G_0\bH^n_u=u^{-n}G_0(f)$ in $\bH^n_u$.
\end{lemme}

\begin{proof}[Sketch of proof]
We will see in \S\ref{subsec:Brieskorndim>1} that the Brieskorn lattice $G_0\bH^n_u$ is identified with the Brieskorn lattice attached to the filtered $\cD_{\PP^1}$-module underlying the mixed Hodge module associated with $\cH^0f_+\cO_U$. On the other hand, it is shown in \cite[\S4.c]{Bibi05} that $G_0(f)$ is identified to the Brieskorn lattice of the Hodge filtration of $\cH^0f_+\cO_U$ shifted by $n$, which leads to the result.
\end{proof}

Recall (\cf \cite{Bibi96bb}) that the \emph{spectrum of $f$ at infinity} is defined as the set of pairs $(\gamma,\delta_\gamma)$, with $\gamma\in\QQ$ and $\delta_\gamma=\dim\gr_\gamma^VG_0(f)$. It is known (\cf \loccit) that $\delta_\gamma=0$ unless $\gamma\in[0,n]$ and that $\delta_\gamma=\delta_{n-\gamma}$ (\ie the spectrum is symmetric with respect to $n/2$).

\begin{corollaire}\label{cor:G0G0}
Under the previous assumptions, let us set $\gamma=\alpha+q$, with $\alpha\in[0,1)$ and $q\in\ZZ$. Then we have
\[
\delta_\gamma=\mu_\alpha^n(n-q)=h_\alpha^{n-q,q}=\dim\gr^{n-\gamma}_{F_\Yu^\cbbullet}H^n_{\dR}(U,\rd+\rd f).
\]
\end{corollaire}

\begin{proof}
We have isomorphisms
\[
\gr_\gamma^VG_0(f)\xrightarrow[\sim]{\textstyle~u^{-n}~}\gr_{\gamma-n}^VG_0\bH^n_u=\gr_{\alpha+q-n}^VG_0\bH^n_u\xrightarrow[\sim]{\textstyle~u^{n-q}~}\gr_\alpha^VG^{n-q}\bH^n_u.\qedhere
\]
\end{proof}

\begin{remarques}\label{rem:G0G0}\mbox{}
\begin{enumerate}
\item\label{rem:G0G01}
The duality $\delta_\gamma=\delta_{n-\gamma}$ implies, together with the general duality statement of \cite[Th.\,2.2]{Yu12}, that, if $U$ is affine and $f$ is cohomologically tame, we have
\[
\dim\gr^\lambda_{F_\Yu^\cbbullet}H^n_{\dR}(U,\rd+\rd f)= \dim\gr^\lambda_{F_\Yu^\cbbullet}H^n_{\dR,\mathrm{c}}(U,\rd+\rd f) \quad\forall\lambda.
\]
\item\label{rem:G0G02}
Assume that $U=(\CC^*)^n$ with coordinates $x_1,\dots,x_n$ and that $f$ is a convenient and non-degenerate Laurent polynomial (in the sense of Kouchnirenko \cite{Kouchnirenko76}). Then it is known that $f$ is cohomologically tame. Moreover,
\begin{align*}
G_0(f)/uG_0(f)&=\Omega^n(U)\big/\rd f\wedge\Omega^{n-1}(U)\\
&\simeq\CC[x_1^{\pm1},\dots,x_n^{\pm1}]\big/(x_1\partial f/\partial x_1,\dots,x_n\partial f/\partial x_n)=\CC[x^{\pm1}]/J(f),
\end{align*}
where the isomorphism is obtained by dividing by $\rd x_1/x_1\!\wedge\!\cdots\!\wedge\!\rd x_n/x_n$, and the filtration $V_\gamma\big(G_0(f)/uG_0(f)\big)$ is identified with the Newton filtration $N_\gamma(\CC[x^{\pm1}]/J(f))$ (\cf \cite[Th.\,4.5]{D-S02a}). Therefore,
\[
\dim\gr_{F_\Yu^\cbbullet}^\lambda H_{\dR}(U,\rd+\rd f)= \dim\gr_{F_\Yu^\cbbullet}^{n-\lambda} H_{\dR}(U,\rd+\rd f)=\dim\gr_\lambda^N(\CC[x^{\pm1}]/J(f)).
\]
\item\label{rem:G0G03}
Let $Y$ be a toric Fano manifold.\footnote{We thank É.\,Mann, Th.\,Reichelt and Ch.\,Sevenheck for providing us with the necessary arguments.} Mirror symmetry associates with it a convenient and non-degenerate Laurent polynomial $f$, and the cohomology $H^*(Y,\CC)$ is identified with $\CC[x^{\pm1}]/J(f)$ graded by the Newton filtration (\cf \cite{B-C-S05}). Since the cohomology is generated by divisor classes (\cf\eg\cite[\S5.2]{Fulton93}), it is of Hodge-Tate type and the Hodge filtration reduces to the filtration by the degree of the cohomology. It follows from the previous results that the Hodge numbers of $Y$ coincide with the irregular Hodge numbers associated to $f$. Such a mirror correspondence was one of the motivations of Kontsevich to introduce the complexes $(\Omega_f^\cbbullet,u\rd+v\rd f)$.
\end{enumerate}
\end{remarques}
\let\Afu\newAfu

\section{Relation with the irregular Hodge filtration of \texorpdfstring{$\ccE^{(v{:}u)f}(*\ccH)$}{E}}\label{sec:relwithFirr}

In this section, we set $\ccE:=\ccE^{(v{:}u)f}(*\ccH)$. We will compare the filtration $F_{\alpha+\bbullet}\ccE$ with the irregular Hodge filtration $F^\irr_{\alpha+\bbullet}\ccE$ as defined in \S\ref{sec:irregHodgefiltration}, namely, we consider the case where $\ccN=\cO_\ccX(*\ccD)$ (notation of \S\ref{subsec:setting}) with its differential $\rd$ twisted by the exponential of the rational function $vf:X\times\PP^1_v=\ccX\ratto\PP^1$. The module $\cFfcN$ considered in \S\ref{subsec:irregV} is obtained here by gluing $\cE^{\tau vf/\hb}[*\ccH]$ (notation of Proposition \ref{prop:EXf}) in the $v$-chart with $\cE^{\tau f/u\hb}[*\ccH]$ in the $u$-chart, and we will regard these modules algebraically with respect to $\tau$, $(v{:}u)$ and $\hb$. We will use the notation introduced in \S\ref{sec:MTEf}.

\begin{theoreme}\label{th:FFirr}
For each $\alpha\in[0,1)$, we have
\begin{align*}
F_{\alpha+\bbullet}\ccE^{vf}(*\ccH)&=F^\irr_{\alpha+\bbullet}\ccE^{vf}(*\ccH),\\
F_{\alpha+\bbullet}G_0\ccE^{f/u}(*\ccH)&=F^\irr_{\alpha+\bbullet}\ccE^{f/u}(*\ccH)\cap G_0\ccE^{f/u}(*\ccH).
\end{align*}
\end{theoreme}

The proof of the theorem will be done in various steps. For the sake of simplicity, we will only treat the case where $H=\emptyset$.

$\bbullet$ Firstly, one identifies $\cE^{\tau vf/\hb}$ as a submodule of $\cO_X(*P_\red)[v,\tau,\hb]\cdot\nobreak\mathrm{e}^{\tau vf/\hb}$ and $\cE^{\tau f/u\hb}$ as a submodule of $\cO_X(*P_\red)[u,u^{-1},\tau,\hb]\cdot\nobreak\mathrm{e}^{\tau f/u\hb}$. According to Proposition \ref{prop:Etaufhb}, we may have a strict inclusion only near points of $P_\red\times\{v=0\}$ and points of $\{f=0\}\times\{u=0\}$. For the latter set, the computation is much simplified because we only consider the intersection with $G_0\ccE$. For the former set, we will need explicit computations of the $V$-filtration entering the very definition of $\cE^{\tau vf/\hb}$ in Proposition~\ref{prop:EXf}.

$\bbullet$ Secondly, one computes the terms $V_\alpha^\tau\cE^{\tau vf/\hb}$ (\resp $V_\alpha^\tau\cE^{\tau f/u\hb}$) of the $V$-filtration relative to $\tau=0$, in order to apply Proposition \ref{prop:FirrValphatau}. We will work analytically with respect to the variables of~$X$ and algebraically with respect to $\tau$, $(u:v)$ and $\hb$.

\subsection{Computation in the \texorpdfstring{$v$}{v}-chart}
We use the algebraic version (with respect to $v,\tau,\hb$) $R_F(\cD_X[v,\tau]\langle\partial_v,\partial_\tau\rangle)$ of $\cR_{\cX\times\Afu_v\times\Afu_\tau}$. Recall that $\cE^{\tau vf/\hb}$ is a coherent $R_F(\cD_X[v,\tau]\langle\partial_v,\partial_\tau\rangle)$-submodule of $\cO_X(*P_\red)[v,\tau]\cdot\mathrm{e}^{\tau vf/\hb}$. We will set $\mathbf{e}=\mathrm{e}^{\tau vf/\hb}$.

\subsubsection*{Computation away from \texorpdfstring{$P_\red$}{P}}
Since $\tau vf$ is holomorphic, we have the equality $\cE^{\tau vf/\hb}=\cO_{X\moins P_\red}[v,\tau]\cdot\mathrm{e}^{\tau vf/\hb}$. Then, from the relation $\partiall_\tau\mathbf{e}=vf\mathbf{e}$, we conclude that $\cE^{\tau vf/\hb}$ is already $(R_F\cD_X[v,\tau]\langle\partial_v\rangle)$-coherent, and hence the $V^\tau$-filtration is given by $V_k^\tau\cE^{\tau vf/\hb}=\tau^{\max(-k,0)}\cE^{\tau vf/\hb}$: by uniqueness of the $V^\tau$-filtration, it is enough to check the strictness of the $\gr_k^{V^\tau}\cE^{\tau vf/\hb}$, which is clear. Therefore, only $\alpha=0$ is \hbox{relevant}. In particular, $V_0^\tau\cE^{\tau vf/\hb}=\cE^{\tau vf/\hb}$. Hence, the quotient modulo $(\tau-\hb)\cE^{\tau vf/\hb}$ is equal to $E^{vf}[\hb]$.

On the other hand, we have $F_{\alpha+p}E^{vf}=E^{vf}$ for any $\alpha\in[0,1)$ and $p\geq0$, and $F_{\alpha-1}E^{vf}=0$, that is, $R_FE^{vf}=E^{vf}[\hb]$.

\subsubsection*{Computation in a neighbourhood of $P_\red$}
Near a point of $P_\red$, let us set $g=1/f$, which is holomorphic in a neighbourhood of this point. In local coordinates we have $g=x^{\bme}$.

\subsubsection*{First step: computation of $\cE^{\tau v/g\hb}$}
By the very definition of Proposition \ref{prop:EXf}\eqref{prop:EXf1} we have, on this neighbourhood, $\cE^{\tau v/g\hb}=(\cO_X(*P_\red)[\tau,v,\hb]\cdot\mathrm{e}^{\tau v/g\hb})[*P_\red]$. Let \hbox{$i_g:X\hto X\times\CC_{t'}$} denote the graph inclusion of $g$ and let \hbox{$p:X\times\CC_{t'}\to X$} denote the projection. Then, by definition of $[*P_\red]$, $i_{g,+}\cE^{\tau v/g\hb}$ is the $R_F(\cD_X[t',v,\tau]\langle\partial_{t'},\partial_v,\partial_\tau\rangle)$-submodule of $(i_{g,+}\cO_X(*P_\red)[v,\tau,\hb])\cdot\mathrm{e}^{\tau v/t'\hb}$ generated by~$V_1^{t'}$.

\begin{lemme}\label{lem:calculEtauvfhb}
The submodule $\cE^{\tau v/g\hb}$ is generated by $x^{-\bf1}\mathbf{e}$ as an $R_F(\cD_X[v,\tau]\langle\partial_v,\partial_\tau\rangle)$-module. Moreover,
\begin{starequation}\label{eq:calculEtauvfhb}
\cE^{\tau v/g\hb}\cap(\tau-\hb)\cO_X(*P_\red)[v,\tau,\hb]\cdot\mathbf{e}=(\tau-\hb)\cE^{\tau v/g\hb}.
\end{starequation}%
\end{lemme}

\begin{proof}
Our first task is to compute the $V^{t'}$-filtration of $(i_{g,+}\cO_X(*P_\red)[v,\tau,\hb])\cdot\nobreak\mathrm{e}^{\tau v/t'\hb}=\bigoplus_{k\geq0}\cO_X(*P_\red)[v,\tau,\hb](\partiall_{t'}^k\delta)\otimes\mathrm{e}^{\tau v/t'\hb}$. For $\alpha\in[0,1)$, let us set
\begin{align*}
(\delta\otimes\mathbf{e})_{1+\alpha}&:=x^{-[\alpha\bme]-\bf1}\delta\otimes\mathrm{e}^{\tau v/t'\hb}\\
(\delta\otimes\mathbf{e})_{<1+\alpha}&:=x^{-\lceil\alpha\bme\rceil}\delta\otimes\mathrm{e}^{\tau v/t'\hb}.
\end{align*}
Then $(\delta\otimes\mathbf{e})_{1+\alpha}$ satisfies the following equations:
\begin{equation}\label{eq:xdeltae}
\left\{
\begin{aligned}
\partiall_v(\delta\otimes\mathbf{e})_{1+\alpha}&=\frac{\tau}{g}(\delta\otimes\mathbf{e})_{1+\alpha}\\
\partiall_\tau(\delta\otimes\mathbf{e})_{1+\alpha}&=\frac{v}{g}(\delta\otimes\mathbf{e})_{1+\alpha}\\
\partiall_{t'}(\delta\otimes\mathbf{e})_{1+\alpha}&=x^{-[\alpha\bme]-\bf1}(\partiall_{t'}\delta)\otimes\mathrm{e}^{\tau v/t'\hb}-\frac{\tau v}{g^2}(\delta\otimes\mathbf{e})_{1+\alpha}\\
\partiall_{x_i}(\delta\otimes\mathbf{e})_{1+\alpha}&=-\frac{e_i}{x_i}\Big(\partiall_{t'}t'+\frac{(1+[\alpha e_i])\hb}{e_i}+t'\partiall_\tau\partiall_v\Big)(\delta\otimes\mathbf{e})_{1+\alpha}.
\end{aligned}\right.
\end{equation}
As a consequence we have
\[
t'\partiall_x^{\bme}(\delta\otimes\mathbf{e})_{1+\alpha}=(-\bme)^\bme\prod_i\prod_{j=1}^{e_i}\Big(t'\partiall_{t'}+\frac{(j+[\alpha e_i])\hb}{e_i}+t'\partiall_\tau\partiall_v\Big)(\delta\otimes\mathbf{e})_{1+\alpha}.
\]
Similarly for $(\delta\otimes\mathbf{e})_{<1+\alpha}$ the last line of \eqref{eq:xdeltae} reads
\[
\partiall_{x_i}(\delta\otimes\mathbf{e})_{<1+\alpha}=-\frac{e_i}{x_i}\Big(\partiall_{t'}t'+\frac{\lceil\alpha e_i\rceil\hb}{e_i}+t'\partiall_\tau\partiall_v\Big)(\delta\otimes\mathbf{e})_{<1+\alpha}
\]
and we have
\[
t'\partiall_x^{\bme}(\delta\otimes\mathbf{e})_{<1+\alpha}=(-\bme)^\bme\prod_i\prod_{j=0}^{e_i-1}\Big(t'\partiall_{t'}+\frac{(j+\lceil\alpha e_i\rceil)\hb}{e_i}+t'\partiall_\tau\partiall_v\Big)(\delta\otimes\mathbf{e})_{<1+\alpha}.
\]
We then easily deduce a Bernstein relation for $(\delta\otimes\mathbf{e})_{1+\alpha}$ and for $(\delta\otimes\mathbf{e})_{<1+\alpha}$, showing that $(\delta\otimes\mathbf{e})_{1+\alpha}$ belongs to $V^{t'}_{1+\alpha}(i_{g,+}\cO_X[v,\tau,\hb])\cdot\mathrm{e}^{\tau v/t'\hb}$ and $(\delta\otimes\mathbf{e})_{<1+\alpha}$ to $V^{t'}_{<1+\alpha}(i_{g,+}\cO_X[v,\tau,\hb])\cdot\mathrm{e}^{\tau v/t'\hb}$. We will now give an explicit expression of these modules.

We have
\begin{align*}
(i_{g,+}\cO_X[v,\tau,&\hb])\cdot\mathrm{e}^{\tau v/t'\hb}=\bigoplus_{k\geq0}\cO_X(*P_\red)[v,\tau,\hb](\partiall_{t'}^k\delta)\otimes\mathrm{e}^{\tau v/t'\hb}\\
&=\bigoplus_{k\geq0}\cO_X(*P_\red)[v,\tau,\hb]\partiall_{t'}^k(\delta\otimes\mathbf{e})_{1+\alpha}\quad \text{(third line of \eqref{eq:xdeltae})}\\
&=\bigoplus_{k\geq0}\cO_X(*P_\red)[v,\tau,\hb]\partiall_{t'}^kt'^k(\delta\otimes\mathbf{e})_{1+\alpha}\quad \text{($t'(\delta\otimes\mathbf{e})_{1+\alpha})=g(\delta\otimes\mathbf{e})_{1+\alpha})$}\\
&=\bigoplus_{k\geq0}\cO_X(*P_\red)[v',\tau',\hb](\partiall_{t'}t')^k(\delta\otimes\mathbf{e})_{1+\alpha}\ \text{(setting $v'=v/g$, $\tau'=\tau/g$)}\\
&\simeq\cO_X(*P_\red)[v',\tau',\eta,\hb]\quad\text{(setting $\eta = \partiall_{t'}t'$, and $(\delta\otimes\mathbf{e})_{1+\alpha}\mto1$)}.
\end{align*}
We have a similar identification by using $(\delta\otimes\mathbf{e})_{<1+\alpha}$. Let us write the last line of \eqref{eq:xdeltae} as
\[
(\partiall_{x_i}+e_ix^{\bme-{\bf1}_i}\partiall_v\partiall_\tau)(\delta\otimes\mathbf{e})_{1+\alpha}=-\frac{e_i}{x_i}\Big(\partiall_{t'}t'+\frac{(1+[\alpha e_i])\hb}{e_i}\Big)(\delta\otimes\mathbf{e})_{1+\alpha}.
\]
For $\bma\in\NN^\ell$ and $\alpha\in[0,1)$, let us set (with the convention that a product indexed by the empty set is equal to one):
\begin{equation}\label{eq:paalpha}
\begin{split}
p_{\bma,\alpha}(s,\hb)&=\prod_i\prod_{j=1}^{a_i}\Big(s+\frac{(j+[\alpha e_i])\hb}{e_i}\Big),\\
p_{\bma,<\alpha}(s,\hb)&=\prod_i\prod_{j=0}^{a_i-1}\Big(s+\frac{(j+\lceil\alpha e_i\rceil)\hb}{e_i}\Big).
\end{split}
\end{equation}
We then have
\begin{align*}
(\partiall_{x_i}+e_ix^{\bme-{\bf1}_i}\partiall_v\partiall_\tau)^{\bma}(\delta\otimes\mathbf{e})_{1+\alpha}&=(-\bme)^{\bma}x^{-\bma}p_{\bma,\alpha}(\partiall_t't',\hb)(\delta\otimes\mathbf{e})_{1+\alpha}\\[3pt]
(\partiall_{x_i}+e_ix^{\bme-{\bf1}_i}\partiall_v\partiall_\tau)^{\bma}(\delta\otimes\mathbf{e})_{<1+\alpha}&=(-\bme)^{\bma}x^{-\bma}p_{\bma,<\alpha}(\partiall_t't',\hb)(\delta\otimes\mathbf{e})_{<1+\alpha}.
\end{align*}
Let us set
\begin{align*}
U_{1+\alpha}&=\sum_{\bma\geq0}\cO_X[v',\tau',\eta,\hb]x^{-\bma}p_{\bma,\alpha}(\eta,\hb)\subset\cO_X(*P_\red)[v',\tau',\eta,\hb]\\
U_{<1+\alpha}&=\sum_{\bma\geq0}\cO_X[v',\tau',\eta,\hb]x^{-\bma}p_{\bma,<\alpha}(\eta,\hb)\subset\cO_X(*P_\red)[v',\tau',\eta,\hb].
\end{align*}
We thus have isomorphisms, by sending $\eta$ to $\partial_{t'}t'$:
\begin{align*}
U_{1+\alpha}&\xrightarrow[\sim]{\textstyle~{}\cdot(\delta\otimes\mathbf{e})_{1+\alpha}~}V^{t'}_0R_F(\cD_X[t',v,\tau]\langle\partial_{t'},\partial_v,\partial_\tau\rangle)\cdot(\delta\otimes\mathbf{e})_{1+\alpha}\\
U_{<1+\alpha}&\xrightarrow[\sim]{\textstyle~{}\cdot(\delta\otimes\mathbf{e})_{<1+\alpha}~}V^{t'}_0R_F(\cD_X[t',v,\tau]\langle\partial_{t'},\partial_v,\partial_\tau\rangle)\cdot(\delta\otimes\mathbf{e})_{<1+\alpha}\end{align*}
If we set $I(\bma)=\{i\mid a_i=0\}$ for $\bma\in\NN^\ell$ and if $I(\bma)^c=\{i\mid a_i\geq1\}$ denotes its complement in $\{1,\dots,\ell\}$, then every element in $\cO_X(*P_\red)[v',\tau',\eta,\hb]$ can be written in a unique way~as
\begin{equation}\label{eq:decompO}
\sum_{\bma\geq0}\wt h_{\bma}(x_{I(\bma)},v',\tau',\eta,\hb)x^{-\bma}
\end{equation}
with $\wt h_{\bma}(x_{I(\bma)},v',\tau',\eta,\hb)\in\CC\{x_{I(\bma)}\}[v',\tau',\eta,\hb]$. Since $p_{\bma,\alpha}$ divides $p_{\bma',\alpha}$ if $\bma'\geq\bma$, we deduce that each element of $U_{1+\alpha}$ can be written as
\begin{equation}\label{eq:decompUa}
\sum_{\bma\geq0}h_{\bma,\alpha}(x_{I(\bma)},v',\tau',\eta,\hb)x^{-\bma}p_{\bma,\alpha}(\eta,\hb),
\end{equation}
and the coefficient $\wt h_{\bma}$ of $x^{-\bma}$ in its decomposition \eqref{eq:decompO} is $h_{\bma,\alpha}(x_{I(\bma)},v',\tau',\eta,\hb)p_{\bma,\alpha}(\eta,\hb)$. By uniqueness, we conclude that an element written as \eqref{eq:decompO} belongs to $U_{1+\alpha}$ if and only if $p_{\bma,\alpha}(\eta,\hb)$ divides $\wt h_{\bma}(x_{I(\bma)},v',\tau',\eta,\hb)$. In particular, the decomposition \eqref{eq:decompUa} is unique.

We wish to identify $U_{1+\alpha}\cdot(\delta\otimes\mathbf{e})_{1+\alpha}$ with $V^{t'}_{1+\alpha}(i_{g,+}\cO_X(*P_\red)[v,\tau,\hb])\cdot\mathrm{e}^{\tau v/t'\hb}$ and $U_{<1+\alpha}\cdot(\delta\otimes\mathbf{e})_{<1+\alpha}$ with $V^{t'}_{<1+\alpha}(i_{g,+}\cO_X(*P_\red)[v,\tau,\hb])\cdot\mathrm{e}^{\tau v/t'\hb}$. It is enough to check
\[
(t'\partiall_{t'}+(1+\alpha)\hb)^mU_{1+\alpha}\cdot(\delta\otimes\mathbf{e})_{1+\alpha}\subset U_{<1+\alpha}\cdot(\delta\otimes\mathbf{e})_{<1+\alpha} \quad\text{for $m$ big enough}
\]
and
\[
U_{1+\alpha}\cdot(\delta\otimes\mathbf{e})_{1+\alpha}/U_{<1+\alpha}\cdot(\delta\otimes\mathbf{e})_{<1+\alpha}\quad\text{has no $\hb$-torsion}
\]
(\cf\cite[Lem.\,3.3.4 \& \S3.4.a]{Bibi01c}). For the first point, we set
\[
I_\alpha=\{i\mid\alpha e_i\in\ZZ\}\quad\text{and, for $\bma\geq0$,}\quad I_\alpha(\bma)=I_\alpha\cap I(\bma)\text{ and }I_\alpha(\bma)^c=I_\alpha\cap I(\bma)^c.
\]
Then $(\delta\otimes\mathbf{e})_{1+\alpha}=x^{-{\bf1}_{I_\alpha}}(\delta\otimes\mathbf{e})_{<1+\alpha}$, and we have the relation
\begin{align*}
\prod_{i\in I_\alpha}(\partiall_{x_i}-e_ix^{\bme-{\bf1}_i}\partiall_v\partiall_\tau)\cdot(\delta\otimes\mathbf{e})_{<1+\alpha}&=(-\bme)^{{\bf1}_{I_\alpha}}(\partiall_{t'}t'+\alpha\hb)^{\#I_\alpha}(\delta\otimes\mathbf{e})_{1+\alpha}\\
&=(-\bme)^{{\bf1}_{I_\alpha}}\big(t'\partiall_{t'}+(1+\alpha)\hb\big)^{\#I_\alpha}(\delta\otimes\mathbf{e})_{1+\alpha}.
\end{align*}
For the torsion-free assertion, let us consider a section \eqref{eq:decompUa} of $U_{1+\alpha}$ and let us decompose (in~a~unique~way) $h_{\bma,\alpha}(x_{I(\bma)},v',\tau',\eta,\hb)$ as
\[
h_{\bma,\alpha}(x_{I(\bma)},v',\tau',\eta,\hb)=\sum_{\epsilong\in\{0,1\}^{I_\alpha(\bma)}}h_{\bma,\alpha,\epsilong}(x_{I(\bma+{\bf1}_{I_\alpha}-\epsilong)},v',\tau',\eta,\hb)x^{\epsilong},
\]
where $h_{\bma,\alpha,\epsilong}$ is holomorphic in its $x$-variables and polynomial in $v',\tau',\eta,\hb$. Then the decomposition \eqref{eq:decompUa} reads
\[
\sum_{\bma\geq0}\sum_{\epsilong\in\{0,1\}^{I_\alpha(\bma)}}h_{\bma,\alpha,\epsilong}(x_{I(\bma+{\bf1}_{I_\alpha}-\epsilong)},v',\tau',\eta,\hb)x^{-(\bma-\epsilong)}p_{\bma,\alpha}(\eta,\hb).
\]
We now note that, for $\epsilong\in\{0,1\}^{I_\alpha(\bma)}$, setting $\bmb=\bma+{\bf1}_{I_\alpha}-\epsilong$, we have
\[
p_{\bmb,<\alpha}(\eta,\hb)=(\eta+\alpha\hb)^{\#I_\alpha(\bmb)^c}\cdot p_{\bma,\alpha}(\eta,\hb).
\]
The unique decomposition \eqref{eq:decompUa} can thus also be written uniquely as
\begin{equation}\label{eq:decompUab}
\sum_{\bmb\geq0}h'_{\bmb,\alpha}(x_{I(\bmb)},v',\tau',\eta,\hb) x^{-\bmb}\frac{p_{\bmb,<\alpha}(\eta,\hb)}{(\eta+\alpha\hb)^{\#I_\alpha(\bmb)^c}}\cdot x^{{\bf1}_{I_\alpha}},
\end{equation}
with $h'_{\bmb,\alpha}=h_{\bma,\alpha,\epsilong}$, where $(\bma,\epsilong)$ is defined by the following conditions:
\begin{align*}
a_i&=b_i\quad\text{if }i\notin I_\alpha,\\
a_i&=b_i-1\text{ and }\epsilon_i=0\quad\text{if }i\in I_\alpha\text{ and }b_i\geq1,\\
a_i&=0\text{ and }\epsilon_i=1\quad\text{if }i\in I_\alpha\text{ and }b_i=0.
\end{align*}
The condition that a section $\eqref{eq:decompUab}\cdot(\delta\otimes\mathbf{e})_{1+\alpha}=\eqref{eq:decompUab}\cdot x^{-{\bf1}_{I_\alpha}}(\delta\otimes\mathbf{e})_{<1+\alpha}$ belongs to $U_{<1+\alpha}\cdot(\delta\otimes\mathbf{e})_{<1+\alpha}$ now reads
\[
\forall b\geq0,\quad(\eta+\alpha\hb)^{\#I_\alpha(\bmb)^c}\text{ divides }h'_{\bmb,\alpha}(x_{I(\bmb)},v',\tau',\eta,\hb).
\]
It is therefore clear that a section of $U_{1+\alpha}\cdot(\delta\otimes\mathbf{e})_{1+\alpha}$ belongs, when multiplied by~$\hb$, to $U_{<1+\alpha}\cdot(\delta\otimes\mathbf{e})_{<1+\alpha}$ if and only if it already belongs to $U_{<1+\alpha}\cdot(\delta\otimes\mathbf{e})_{<1+\alpha}$. In other words, $U_{1+\alpha}\cdot(\delta\otimes\mathbf{e})_{1+\alpha}\big/U_{<1+\alpha}\cdot(\delta\otimes\mathbf{e})_{<1+\alpha}$ has no $\hb$-torsion.

We conclude that
\begin{align*}
V_0^{t'}R_F(\cD_X[t',v,\tau]\langle\partial_{t'},\partial_v,\partial_\tau\rangle)\cdot(\delta\otimes\mathbf{e})_1&=U_1\cdot(\delta\otimes\mathbf{e})_1\\
&=V_1^{t'}(i_{g,+}\cO_X(*P_\red)[v,\tau,\hb])\cdot(\delta\otimes\mathbf{e})_1,
\end{align*}
hence $i_{g,+}\cE^{\tau v/g\hb}$ is generated by $(\delta\otimes\mathbf{e})_1$. It follows that $\cE^{\tau v/g\hb}$ is generated by~$x^{-\bf1}\mathbf{e}$.

We will prove the analogue of \eqref{eq:calculEtauvfhb} after $i_{g,+}$, from which one deduces similarly \eqref{eq:calculEtauvfhb}. We first notice that the equality
\[
V^{t'}_1(i_{g,+}\cE^{\tau v/g\hb})\cap(\tau-\hb)(i_{g,+}\cO_X(*P_\red)[v,\tau,\hb])\cdot\mathrm{e}^{\tau v/t'\hb}=(\tau-\hb)V^{t'}_1(i_{g,+}\cE^{\tau v/g\hb})
\]
immediately follows from the unique decomposition \eqref{eq:decompUa} of a local section of $V^{t'}_1(i_{g,+}\cE^{\tau v/g\hb})$. To end the proof, it therefore suffices to produce a similar unique decomposition of local sections of $V^{t'}_{1+k}(i_{g,+}\cE^{\tau v/g\hb}):=\sum_{j=0}^k\partiall_{t'}^jV^{t'}_1(i_{g,+}\cE^{\tau v/g\hb})$ for any $k\geq1$. This is obtained by writing
\[
\partiall^k_{t'}(\delta\otimes\mathbf{e})_1=x^{-k\bme}\partiall^k_{t'}t^{\prime k}(\delta\otimes\mathbf{e})_1=x^{-k\bme}\prod_{j=0}^{k-1}(\partiall_{t'}t'+j\hb)(\delta\otimes\mathbf{e})_1,
\]
giving rise to a formula similar to \eqref{eq:decompUa} for sections of $V^{t'}_{1+k}(i_{g,+}\cE^{\tau v/g\hb})$, which makes use of polynomials $p_{\bma,k}$ ($k\geq1$), derived from $p_{\bma,0}$ like in \cite[Lem.\,4.7]{Bibi96a}.
\end{proof}

\subsubsection*{Second step: computation of the $V^\tau$-filtration of $\cE^{\tau v/g\hb}$}
For $\alpha\in[0,1)$, let us set $\mathbf{e}_\alpha=\mathrm{e}^{\tau v/g\hb}/x^{[\alpha\bme]+{\bf1}}$.

\begin{lemme}\label{lem:Valphatauv}
The $V_\bbullet^\tau$-filtration of $\cE^{\tau v/g\hb}$ satisfies
\[
V_\alpha^\tau\cE^{\tau v/g\hb}=V_0^\tau R_F(\cD_X[v,\tau]\langle\partial_v,\partial_\tau\rangle)\cdot\mathbf{e}_\alpha\quad\forall\alpha\in[0,1).
\]
\end{lemme}

\begin{proof}
Since we are only interested in giving the formula for $V_\alpha^\tau\cE^{\tau v/g\hb}$, we can as well work with the localized module $\cE^{\tau v/g\hb}[\tau^{-1}]$ (see \cite[Lem.\,3.4.1]{Bibi01c}). In such a way, we can write
\[
\mathbf{e}_\alpha=x^{-[\alpha\bme]}(x^{-\bf1}\mathbf{e})=x^{\lceil(1-\alpha)\bme\rceil}\tau^{-1}\partiall_v(x^{-\bf1}\mathbf{e}),
\]
showing that $\mathbf{e}_\alpha$ is a section of $\cE^{\tau v/g\hb}[\tau^{-1}]$. For $\alpha\in[0,1)$ let us also set
\[
\mathbf{e}_{<\alpha}=\Big(\prod_{i\in I_\alpha}x_i\Big)\mathbf{e}_\alpha=:x^{{\bf1}_{I_\alpha}}\mathbf{e}_\alpha
\]
and, for $p\in\ZZ$,
\begin{align*}
U_{\alpha+p}^\tau(\cE^{\tau v/g\hb}[\tau^{-1}])&=\tau^{-p}V_0^\tau R_F(\cD_X[v,\tau]\langle\partial_v,\partial_\tau\rangle)\cdot\mathbf{e}_\alpha\\
U_{<\alpha+p}^\tau(\cE^{\tau v/g\hb}[\tau^{-1}])&=\tau^{-p}V_0^\tau R_F(\cD_X[v,\tau]\langle\partial_v,\partial_\tau\rangle)\cdot\mathbf{e}_{<\alpha},
\end{align*}
so that, clearly,
\begin{equation}\label{eq:UU<}
U_{<\alpha+p}^\tau(\cE^{\tau v/g\hb}[\tau^{-1}])\subset U_{\alpha+p}^\tau(\cE^{\tau v/g\hb}[\tau^{-1}]).
\end{equation}
For $p\leq0$, we will set $U_{\alpha+p}^\tau\cE^{\tau v/g\hb}=U_{\alpha+p}^\tau(\cE^{\tau v/g\hb}[\tau^{-1}])$ and $U_{<\alpha+p}^\tau\cE^{\tau v/g\hb}=U_{<\alpha+p}^\tau(\cE^{\tau v/g\hb}[\tau^{-1}])$. We will prove that $U^\tau_\bbullet(\cE^{\tau v/g\hb}[\tau^{-1}])$ is the good $V^\tau$-filtration of $\cE^{\tau v/g\hb}[\tau^{-1}]$. It is enough to prove that $U_\alpha^\tau\cE^{\tau v/g\hb}=V_\alpha^\tau\cE^{\tau v/g\hb}$ for $\alpha\in[0,1)$. The proof will be very similar to that of Lemma \ref{lem:calculEtauvfhb}, although with the variable $\tau$ instead of the variable $t'$.

By using \eqref{eq:UU<}, one first easily checks that $U_{\alpha-1}^\tau\cE^{\tau v/g\hb}\subset U_{<\alpha}^\tau\cE^{\tau v/g\hb}$ and
\[
(\tau\partiall_\tau+\alpha\hb)^{\# I_\alpha}U_\alpha^\tau\cE^{\tau v/g\hb}\subset U_{<\alpha}^\tau\cE^{\tau v/g\hb}.
\]
Indeed, the first point follows from the relation $\tau\mathbf{e}_\alpha=x^{\bme}\partiall_v\mathbf{e}_\alpha=x^{\bme-{\bf1}_{I_\alpha}}\partiall_v\mathbf{e}_{<\alpha}$, and the second one follows from the relation
\[
(\tau\partiall_\tau+\alpha\hb)^{\# I_\alpha}\mathbf{e}_\alpha=\Big(\frac{(-1)^{\#I_\alpha}}{\prod_{i\in I_\alpha}e_i}\Big)\prod_{i\in I_\alpha}\partiall_{x_i}\cdot\mathbf{e}_{<\alpha}.
\]
Due to the uniqueness of the $V^\tau$-filtration, the assertion of the lemma would follow from the property that $\gr_\alpha^{U^\tau}\cE^{\tau v/g\hb}$ has no $\hb$-torsion. We will argue in a way similar to that of Lemma \ref{lem:calculEtauvfhb} by finding a suitable expression for the sections of $U_\alpha^\tau\cE^{\tau v/g\hb}$.

Let us decompose $\cO_X[x^{-1},\hb][v,\tau]$ as $\cO_X[x^{-1},\hb][v,v\tau]\oplus\tau\cO_X[x^{-1},\hb][v\tau,\tau]$. Due to the relation $v\tau\mathbf{e}=x^{\bme}\tau\partiall_\tau\mathbf{e}$, we have isomorphisms of $\cO_X[x^{-1},\hb]$-modules
\begin{equation}\label{eq:isomathbfe}
\begin{split}
\cO_X[x^{-1},\hb][v,v\tau]\cdot\mathbf{e}&\isom\cO_X[x^{-1},\hb,v][\tau\partiall_\tau]\cdot\mathbf{e}\\
\cO_X[x^{-1},\hb][v\tau,\tau]\cdot\mathbf{e}&\isom \cO_X[x^{-1},\hb,\tau]\langle\tau\partiall_\tau\rangle\cdot\mathbf{e}
\end{split}
\end{equation}
given respectively by
\[
v^j(v\tau)^k\mathbf{e}\mto
x^{k\bme}v^j\prod_{i=0}^{k-1}\big(\tau\partiall_\tau-i\hb\big)\mathbf{e},\quad
(v\tau)^j\tau^k\mto x^{j\bme}\tau^k\prod_{i=0}^{j-1}\big(\tau\partiall_\tau-i\hb\big)\mathbf{e}.
\]
We thus obtain an isomorphism of free $\cO_X[x^{-1},\hb]$-modules:
\[
\cO_X[x^{-1},\hb][v,\tau]\cdot\mathbf{e}\isom\big(\cO_X[x^{-1},\hb,v][\tau\partiall_\tau]\oplus\tau\cO_X[x^{-1},\hb,\tau]\langle\tau\partiall_\tau\rangle\big)\cdot\mathbf{e}.
\]
We can replace $\mathbf{e}$ with $\mathbf{e}_\alpha$ or $\mathbf{e}_{<\alpha}$ in the above isomorphism. We will express $U_\alpha^\tau\cE^{\tau v/g\hb}$ and $U_{<\alpha}^\tau\cE^{\tau v/g\hb}$ as sub-$\cO_X[\hb]$-modules of the right-hand side, and by using the generator $\mathbf{e}_{<\alpha}$ in both cases, to make the computation of the quotient module easier.

We note first that $U_\alpha^\tau\cE^{\tau v/g\hb}=\cO_X[v,\hb]\langle\partiall_x,\partiall_v,\tau\partiall_\tau\rangle\cdot\mathbf{e}_\alpha$, \ie we can forget the action of $\tau$, since $\tau^k\mathbf{e}_\alpha=x^{k\bme}\partiall_v^k\mathbf{e}_\alpha$. We have a similar assertion for $U_{<\alpha}^\tau\cE^{\tau v/g\hb}$.

From the relation
\[
v^\ell\partiall_x^{\bma}\partiall_v^k(\tau\partiall_\tau)^j\mathbf{e}_\alpha=
\begin{cases}
{}\star x^{-(\bma+(k-\ell)\bme)}\tau^{k-\ell} p_{\bma,\alpha}(\tau\partiall_\tau+(k-\ell)\hb,\hb)\cdot{}\\
\hspace*{1.5cm}{}\cdot(\tau\partiall_\tau+(k-\ell)\hb)^j\prod_{i=0}^{\ell-1}(\tau\partiall_\tau-i\hb)\mathbf{e}_\alpha
&\text{if }k>\ell\geq0,\\[5pt]
{}\star x^{-\bma}v^{\ell-k}p_{\bma,\alpha}(\tau\partiall_\tau,\hb)(\tau\partiall_\tau)^j\prod_{i=0}^{k-1}(\tau\partiall_\tau-i\hb)\mathbf{e}_\alpha&\text{if }0\leq k\leq\ell,
\end{cases}
\]
for some nonzero constants $\star$ and with $p_{\bma,\alpha}(s,\hb)$ defined by \eqref{eq:paalpha}, we conclude that, through the isomorphism \eqref{eq:isomathbfe},
\begin{multline*}
U_\alpha^\tau\cE^{\tau v/g\hb}=\sum_{\bma\geq0}\cO_X[\hb][v,\tau\partiall_\tau]\cdot x^{-\bma}p_{\bma,\alpha}(\tau\partiall_\tau,\hb)\mathbf{e}_\alpha\\[-5pt]
{}+\sum_{\bma\geq0}\sum_{n>0}\tau^n\cO_X[\hb][\tau\partiall_\tau]x^{-(\bma+n\bme)}p_{\bma,\alpha}(\tau\partiall_\tau+n\hb,\hb)\mathbf{e}_\alpha
\end{multline*}

Formula \eqref{eq:paalpha} shows that, if $\bma\geq\bma'\geq0$, then $p_{\bma',\alpha}$ divides $p_{\bma,\alpha}$. It follows that any section of $U_\alpha^\tau\cE^{\tau v/g\hb}$ can be written as
\begin{multline}\label{eq:sectionUalpha}
\sum_{\bma\geq0}h_{\bma,\alpha}(x_{I(\bma)},v,\tau\partiall_\tau,\hb)x^{-\bma}p_{\bma,\alpha}(\tau\partiall_\tau,\hb)\cdot\mathbf{e}_\alpha\\[-5pt]
{}+\sum_{n>0}\tau^n\sum_{\bma\geq0}g_{\bma,\alpha,n}(x_{I(\bma+n\bme)},\tau\partiall_\tau,\hb)x^{-(\bma+n\bme)}p_{\bma,\alpha}(\tau\partiall_\tau+n\hb,\hb)\cdot\mathbf{e}_\alpha
\end{multline}
with $h_{\bma,\alpha}$ holomorphic in its $x$-variables and polynomial in $v,\tau\partiall_\tau,\hb$, and $g_{\bma,\alpha,n}$ holomorphic in its $x$-variables and polynomial in $\tau\partiall_\tau,\hb$.

Let us check that the decomposition \eqref{eq:sectionUalpha} is unique. The coefficient $h^{(n)}$~of~$\tau^n$~(\hbox{$n\!\geq\!0$}) is uniquely determined by the section. If $n=0$, the function $h^{(0)}\in\nobreak\cO_{X,0}[x^{-1},v,\eta,\hb]$ decomposes uniquely as $\sum_{\bma\geq0}h^{(0)}_{\bma}(x_{I(\bma)},v,\eta,\hb)x^{-\bma}$. Thus~$h^{(0)}_{\bma}$ must be divisible by $p_{\bma,\alpha}(\eta,\hb)$ and this determines uniquely $h_{\bma,\alpha}(x_{I(\bma)},v,\eta,\hb)$. We argue similarly for $n>0$ and $h^{(n)}\in\cO_{X,0}[x^{-1},\eta,\hb]$.

There is a similar decomposition for sections of $U_{<\alpha}^\tau\cE^{\tau v/g\hb}$, by replacing $p_{\bma,\alpha}(\tau\partiall_\tau,\hb)\cdot\mathbf{e}_\alpha$ with $p_{\bma,<\alpha}(\tau\partiall_\tau,\hb)\cdot\mathbf{e}_{<\alpha}$. In order to check whether a section \eqref{eq:sectionUalpha} belongs to $U_{<\alpha}^\tau\cE^{\tau v/g\hb}$, we replace $\mathbf{e}_\alpha$ with $x^{-{\bf1}_{I_\alpha}}\mathbf{e}_{<\alpha}$. Let us decompose (in~a~unique~way) $h_{\bma,\alpha}(x_{I(\bma)},v,\tau\partiall_\tau,\hb)$ as
\[
h_{\bma,\alpha}(x_{I(\bma)},v,\tau\partiall_\tau,\hb)=\sum_{\epsilong\in\{0,1\}^{I_\alpha(\bma)}}h_{\bma,\alpha,\epsilong}(x_{I(\bma+{\bf1}_{I_\alpha}-\epsilong)},v,\tau\partiall_\tau,\hb)x^{\epsilong},
\]
where $h_{\bma,\alpha,\epsilong}$ is holomorphic in its $x$-variables and polynomial in $v,\tau\partiall_\tau,\hb$. We have a similar decomposition for $g_{\bma,\alpha,n}(x_{I(\bma+n\bme)},\tau\partiall_\tau,\hb)$. Then the decomposition \eqref{eq:sectionUalpha} reads
\bgroup
\multlinegap0pt
\begin{multline*}
\sum_{\bma\geq0}\sum_{\epsilong\in\{0,1\}^{I_\alpha(\bma)}}\hspace*{-5mm}h_{\bma,\alpha,\epsilong}(x_{I(\bma+{\bf1}_{I_\alpha}-\epsilong)},v,\tau\partiall_\tau,\hb)x^{-(\bma+{\bf1}_{I_\alpha}-\epsilong)}p_{\bma,\alpha}(\tau\partiall_\tau,\hb)\cdot\mathbf{e}_{<\alpha}\\
{}+\sum_{\bma\geq0}\sum_{n>0}\sum_{\epsilong\in\{0,1\}^{I_\alpha(\bma+n\bme)}}\hspace*{-5mm}\tau^ng_{\bma,\alpha,n}(x_{I(\bma+n\bme+{\bf1}_{I_\alpha}-\epsilong)},\tau\partiall_\tau,\hb)\cdot{}\\[-5pt]
{}\cdot x^{-(\bma+n\bme+{\bf1}_{I_\alpha}-\epsilong)}p_{\bma,\alpha}(\tau\partiall_\tau+n\hb,\hb)\cdot\mathbf{e}_{<\alpha}.
\end{multline*}
\egroup
We now note that, for $\epsilong\in\{0,1\}^{I_\alpha(\bma)}$, setting $\bmb=\bma+{\bf1}_{I_\alpha}-\epsilong$, we have
\[
p_{\bmb,<\alpha}(s,\hb)=(s+\alpha\hb)^{\#I_\alpha(\bmb)^c}\cdot p_{\bma,\alpha}(s,\hb).
\]
The unique decomposition \eqref{eq:sectionUalpha} can thus also be written uniquely as
\begin{multline}\label{eq:sectionUalphab}
\sum_{\bmb\geq0}h'_{\bmb,\alpha}(x_{I(\bmb)},v,\tau\partiall_\tau,\hb) x^{-\bmb}\frac{p_{\bmb,<\alpha}(\tau\partiall_\tau,\hb)}{(\tau\partiall_\tau +\alpha\hb)^{\#I_\alpha(\bmb)^c}}\cdot\mathbf{e}_{<\alpha}\\
+\sum_{b\geq0}\sum_{n>0}\tau^ng'_{\bmb,\alpha,n}(x_{I(\bmb+n\bme)},v,\tau\partiall_\tau,\hb) x^{-(\bmb+n\bme)}\frac{p_{\bmb,<\alpha}(\tau\partiall_\tau+n\hb,\hb)}{(\tau\partiall_\tau +(n+\alpha)\hb)^{\#I_\alpha(\bmb)^c}}\cdot\mathbf{e}_{<\alpha}
\end{multline}
with $h'_{\bmb,\alpha}=h_{\bma,\alpha,\epsilong}$, where $(\bma,\epsilong)$ is defined by the following conditions:
\begin{align*}
a_i&=b_i\quad\text{if }i\notin I_\alpha,\\
a_i&=b_i-1\text{ and }\epsilon_i=0\quad\text{if }i\in I_\alpha\text{ and }b_i\geq1,\\
a_i&=0\text{ and }\epsilon_i=1\quad\text{if }i\in I_\alpha\text{ and }b_i=0,
\end{align*}
and similarly for $g'_{\bmb,\alpha,n}$. The condition that a section \eqref{eq:sectionUalphab} belongs to $U_{<\alpha}^\tau\cE^{\tau v/g\hb}$ now reads
\begin{align*}
\forall b\geq0,\quad(\tau\partiall_\tau +\alpha\hb)^{\#I_\alpha(\bmb)^c}&\text{ divides }h'_{\bmb,\alpha}(x_{I(\bmb)},v,\tau\partiall_\tau,\hb),\\
\tag*{and}
\forall b\geq0,\forall n>0,\quad(\tau\partiall_\tau +(n+\alpha)\hb)^{\#I_\alpha(\bmb)^c}&\text{ divides }g'_{\bmb,\alpha,n}(x_{I(\bmb+n\bme)},\tau\partiall_\tau,\hb).
\end{align*}
It is therefore clear that a section \eqref{eq:sectionUalphab} of $U_\alpha^\tau\cE^{\tau v/g\hb}$ belongs, when multiplied by~$\hb$, to $U_{<\alpha}^\tau\cE^{\tau v/g\hb}$ if and only if it already belongs to $U_{<\alpha}^\tau\cE^{\tau v/g\hb}$. In other words, $\gr_\alpha^{U^\tau}\cE^{\tau v/g\hb}$ has no $\hb$-torsion.
\end{proof}

\subsubsection*{Third step: End the proof of Theorem \ref{th:FFirr} in the $v$-chart}
We will now use the expression of Lemma \ref{lem:Valphatauv} to regard $V_\alpha^\tau\cE^{\tau v/g\hb}$ as an $\cO_X[v,\tau,\hb]$-submodule of $\cO_X[x^{-1},v,\tau,\hb]\cdot\nobreak\mathbf{e}_\alpha$. We can write (locally on~$X$ near $P$):
\[
V_0^\tau(R_F\cD_X[v,\tau]\langle\partial_v,\partial_\tau\rangle)=\cO_X[v,\tau,\hb]\langle\partiall_{x},\partiall_{y'},\partiall_v,\tau\partiall_\tau\rangle
\]
and we notice that the action of $\tau\partiall_\tau$ on $\mathbf{e}_\alpha$ is equal to that of $v\partiall_v$, so we can forget~$\tau\partiall_\tau$. We will also forget $(y',\partiall_{y'})$ which plays no significant role. Recall that the variables~$x$ are indexed as $x_1,\dots,x_\ell$. Working now within $\cO_X[x^{-1},v,\tau,\hb]\cdot\mathbf{e}_\alpha$, we have by induction on~$|\bma|$,
\[
(x\partiall_x)^{\bma}\partiall_v^c\mathbf{e}_\alpha\equiv\hb^{|\bma|+c}x^{-c\bme}\bigg(\sum_{j=0}^{|\bma|}q_{\bma,c,j}(x)x^{-j\bme}v^j\bigg)\mathbf{e}_\alpha\mod(\tau-\hb)\cO_X[x^{-1},v,\tau,\hb],
\]
where $q_{\bma,c,j}(x)$ is some polynomial and $q_{\bma,c,|\bma|}(x)$ is a nonzero constant. From this one concludes that there exist polynomials $r_{\bma,c,j}(x,\hb)$ with $r_{\bma,c,|\bma|}(x,\hb)$ constant, such that
\begin{multline*}
\partiall_x^{\bma}\partiall_v^c\mathbf{e}_\alpha\equiv\hb^{|\bma|+c}x^{-(|\bma|+c)\bme}\bigg(\sum_{j=0}^{|\bma|}r_{\bma,c,j}(x,\hb)x^{(|\bma|-j)\bme-\bma}v^j\bigg)x^{-[\alpha\bme]-{\bf1}}\mathbf{e}\\[-5pt]\mod(\tau-\hb)\cO_X[x^{-1},v,\tau,\hb],
\end{multline*}
and since for $0\leq j\leq|\bma|$ we have $|\bma-(|\bma|-j)\bme|_+\leq j$ (see the end of the proof of Lemma \ref{lem:FF'v}), the coefficient of $v^j$ belongs to $F_j\cO_X(*P_\red)([(\alpha+p)P])$ with $p=|\bma|+c$. Using that
\begin{align*}
(\tau-\hb)V_\alpha^\tau\cE^{\tau v/g\hb}&=(\tau-\hb)\cE^{\tau v/g\hb}\cap V_\alpha^\tau\cE^{\tau v/g\hb}\quad \text{(\cf \cite[Proof of Prop.\,3.1.2]{E-S-Y13})}\\
&= (\tau-\hb)\cO_X[x^{-1},v,\tau,\hb]\cap V_\alpha^\tau\cE^{\tau v/g\hb}\quad \text{(Lemma \ref{lem:calculEtauvfhb})},
\end{align*}
we conclude that the coefficient of $\hb^p$ in the $\gr\big(V_\alpha^\tau\cE^{\tau v/g\hb}/(\tau-\hb)V_\alpha^\tau\cE^{\tau v/g\hb}\big)$ (graded with respect to the $\hb$-adic filtration) is contained in $F_{\alpha+p}E^{vf}$, so $F^\irr_{\alpha+p}E^{vf}\subset F_{\alpha+p}E^{vf}$, according to Remark \ref{rem:FirrValphatau}.

In order to obtain the reverse inclusion, we remark that, for $|\bma|+c=p$ fixed, $x^{-\bma}v^{|\bma|}\mathrm{e}^{vf}/x^{[(\alpha+p)\bme]+\bf1}\in F_{|\bma|}\cO_X(*P_\red)([(\alpha+p)P])v^{|\bma|}\mathrm{e}^{vf}$ is equal, up to a nonzero constant and modulo $\sum_{j<|\bma|}F_j\cO_X(*P_\red)([(\alpha+p)P])v^j\mathrm{e}^{vf}$, to the class of $\partiall_x^{\bma}\partiall_v^c\mathbf{e}_\alpha$. We conclude by induction on $|\bma|$, the case where $|\bma|=0$ being clear.\qed

\Subsection{Computation in the \texorpdfstring{$u$}{u}-chart}\label{subsec:computuchart}

\subsubsection*{Computation away from $P_\red$}
We have $F_{\alpha+p}G_0E^{f/u}=G_0E^{f/u}=\cO_{X\moins P}[u]\mathrm{e}^{f/u}$ for $p\geq0$. We set similarly $\mathbf{e}=\mathrm{e}^{\tau f/u\hb}$.

\begin{lemme}
With respect to the inclusion $\cE^{\tau f/u\hb}\subset\cO_{X\moins P_\red}[u,u^{-1},\tau,\hb]\cdot\mathbf{e}$, we have $\mathbf{e}\in\cE^{\tau f/u\hb}$.
\end{lemme}

\begin{proof}
The question is local near a point of $\{f=0\}$, since otherwise we have equality in the previous inclusion, according to Proposition \ref{prop:Etaufhb}, and it amounts to proving that $\mathbf{e}\in V_1^u(\cO_{X\moins P_\red}[u,u^{-1},\tau,\hb]\cdot\mathbf{e})$, so we are reduced to computing the order of~$\mathbf{e}$ with respect to the $V^u$-filtration.

Let us first assume that the divisor $\{f=0\}$ has normal crossings. Let us choose local coordinates $x_1,\dots,x_n$ such that $f(x)=x^{\bmm}$ with $\bmm\in\NN^n$ (a local setting not to be confused with that of \S\ref{subsec:setting}). From the relation $u\partiall_u\mathbf{e}=-(\tau f/u)\mathbf{e}$ we obtain
\[
\partiall_{x_i}(f\mathbf{e})=\frac{m_i\hb}{x_i}\,x^{\bmm}\mathbf{e}+\frac{m_i}{x_i}\,\frac{\tau f}{u}\,x^{\bmm}\mathbf{e}=-m_i(u\partiall_u-\hb)x^{\bmm-{\bf1}_i}\mathbf{e},
\]
and iterating the process we find
\[
\partiall_x^{\bmm}(f\mathbf{e})=(-\bmm)^{\bmm}\prod_{i=1}^n\prod_{j=1}^{m_i}(u\partiall_u-j\hb/m_i)\cdot\mathbf{e}.
\]
Since $f\mathbf{e}=u\partiall_\tau\mathbf{e}$, this gives a Bernstein relation for $\mathbf{e}$ showing that $\mathbf{e}\in V_{<0}^u(\cO_{X\moins P_\red}[u,u^{-1},\tau,\hb]\cdot\mathbf{e})$.

When $\{f=0\}$ is arbitrary, the proof proceeds exactly like in \cite{Kashiwara76}. We work locally near a point of $\{f=0\}$ and we choose a projective birational morphism $\pi:X'\to X$ which is an isomorphism away from $\{f=0\}$ and such that $f':=f\circ\pi$ defines a normal crossing divisor. Using the global section $\mathrm{e}^{\tau f'/u\hb}$ of $V_{<0}^u\cE^{\tau f'/u\hb}$ (first part of the proof), one constructs a global section $\mathbf{e}'$ of $\cH^0\pi_+V_{<0}^u\cE^{\tau f'/u\hb}$ which coincides with $\mathbf{e}$ away from $\{f=0\}$. This is done by using the global section ${\bf1}_{\cX\from\cX'}$ of $\cR_{\cX\from\cX'}[u,\tau]\langle u\partiall_u\rangle$. Because $\cE^{\tau f'/u\hb}$ underlies a mixed twistor module, $\cH^0\pi_+\cE^{\tau f'/u\hb}$ is strictly specializable along $u=0$ and we have $V_{<0}^u\cH^0\pi_+\cE^{\tau f'/u\hb}=\cH^0\pi_+V_{<0}^u\cE^{\tau f'/u\hb}$. Therefore, $\mathbf{e}'$ is a section of $V_{<0}^u\cH^0\pi_+\cE^{\tau f'/u\hb}$, and thus satisfies a non-trivial Bernstein equation of the form
\[
\prod_{\beta<0}(u\partiall_u+\beta\hb)^{\nu_\beta}\cdot\mathbf{e}'=uP(x,u,\partiall_x,\partiall_\tau,u\partiall_u)\cdot\mathbf{e}'.
\]
We conclude that $\mathbf{e}$ satisfies the same equation away from $\{f=0\}$, hence everywhere, since $\cO_X[u,u^{-1},\tau,\hb]$ has no $\cO_X$-torsion.
\end{proof}

Due to the relation $f\tau\partiall_u\mathbf{e}=-\tau^2\partiall_\tau^2\mathbf{e}$ we conclude that
\[
V_0^\tau\cE^{\tau f/u\hb}\supset V_0^\tau R_F(\cD_{(X\moins P)}[u,\tau]\langle\partial_u,\partial_\tau\rangle)\cdot\mathbf{e}\supset \cO_{X\moins P}[u,\tau,\hb]\cdot\mathbf{e}.
\]
Then, computing modulo $(\tau-\hb)\cE^{\tau f/u\hb}$, $F_{\alpha+p}^\irr E^{f/u}$ contains $\cO_{X\moins P}[u,\hb]\cdot\mathrm{e}^{f/u}=G_0E^{f/u}$, hence $F_{\alpha+p}G_0E^{f/u}=F^\irr_{\alpha+p}G_0E^{f/u}$ away from $P_\red$.

\begin{remarque}\label{rem:FirrEu}
The explicit computation of $F_{\alpha+p}^\irr E^{f/u}$ in the neighbourhood of $f=\nobreak0$ would be more complicated, and restricting to $G_0E^{f/u}$ allows us to avoid this computation. Let us however note that, in the neighbourhood of the smooth locus of $f^{-1}(0)$, an explicit formula for $F_{\alpha+p}^\irr E^{f/u}$ can be obtained from Lemma \ref{lem:Valphatauv} by setting there $g=u$ and $v=f$. Since the order of the pole at $u=0$ is one, the only interesting $\alpha$ is zero, and the result is:
\[
F_p^\irr E^{f/u}=\frac{1}{u^{p+1}}\Big(\sum_{j=0}^p\cO_X[u]\frac{f^k}{u^k}\Big)\cdot\mathrm{e}^{f/u}\quad \text{($f$ smooth)}.
\]
This formula extends in a natural way to $F_{\alpha+p}^\irr E^{f/u}(*H)$, provided that moreover $f^{-1}(0)$ has no common component with $H$ and that $f^{-1}(0)\cup H$ has normal crossings.
\end{remarque}

\subsubsection*{Computation near \texorpdfstring{$P_\red$}{P}}

The computation is similar to, and even simpler than, the computation done in the $v$-chart. Indeed, due to Proposition \ref{prop:Etaufhb}, we have $\cE^{\tau/ux^{\bme}\hb}=\cO_X(*P_\red)[u,u^{-1},\tau,\hb]\cdot\mathrm{e}^{\tau/ux^{\bme}\hb}$, and there is no need for an analogue of Lemma \ref{lem:calculEtauvfhb}. We will consider the variable $u$ as part of the $x$-variables, and the divisor $u=0$ of~$\ccX$ (\cf \S\ref{subsec:rescaling}). For $\alpha\in[0,1)$, we set $\mathbf{e}_\alpha=\mathrm{e}^{\tau/u x^{\bme}\hb}/ux^{[\alpha\bme]+\bf1}$. The following lemma is similar to Lemma~\ref{lem:Valphatauv}.

\begin{lemme}\label{lem:Valphatauu}
The $V_\bbullet^\tau$-filtration of $\cE^{\tau/ux^{\bme}\hb}$ satisfies
$$
V_\alpha^\tau\cE^{\tau/ux^{\bme}\hb}=V_0^\tau R_F(\cD_X[u,\tau]\langle\partial_u,\partial_\tau\rangle)\cdot\mathbf{e}_\alpha\quad\forall\alpha\in[0,1).\eqno\qed
$$
\end{lemme}

We also obtain
\begin{multline*}
\partiall_x^{\bma}\partiall_u^b\mathbf{e}_\alpha\in \hb^{|\bma|+b}F_{|\bma|+b}\big(\cO_X(*P_\red)[u,u^{-1}]\big)\big([(\alpha+|\bma|+b)\ccP]\big)\cdot\mathrm{e}^{\tau/ux^{\bme}\hb}\\
\mod(\tau-\hb)\cE^{\tau/ux^{\bme}\hb},
\end{multline*}
where $F_\bbullet\big(\cO_X(*P_\red)[u,u^{-1}]\big)$ is the filtration by the order of the pole along $\ccP_\red$. Moreover, the coefficient of $(ux^{\bme})^{-(|\bma|+b)}\cdot (ux^{[\alpha\bme]+\bf1})^{-1}\cdot\mathrm{e}^{\tau/ux^{\bme}\hb}$ is a nonzero constant. It follows that
\[
F_{\alpha+p}^\irr E^{1/ux^{\bme}}=F_p\big(\cO_X(*P_\red)[u,u^{-1}]\big)\big([(\alpha+p)\ccP]\big)\cdot\mathrm{e}^{1/ux^{\bme}}.
\]
Intersecting with $G_0E^{1/ux^{\bme}}=\cO_X(*P_\red)[u]\mathrm{e}^{1/ux^{\bme}}$ gives
\[
F_{\alpha+p}^\irr E^{1/ux^{\bme}}\cap G_0E^{1/ux^{\bme}}=F_p\cO_X(*P_\red)\big([(\alpha+p)P]\big)[u]\cdot\mathrm{e}^{1/ux^{\bme}}=F_{\alpha+p}G_0E^{1/ux^{\bme}}.
\]

This ends the proof of Theorem \ref{th:FFirr}.\qed

\subsection{Another approach of Theorem \ref{th:FFirr} at the de~Rham level}
Let us assume that the zero divisor $f^{-1}(0)$ of $f:X\to\PP^1$ is smooth, that it has no component in common with $D$, and that $f^{-1}(0)\cup D$ still has normal crossings. We have a filtration $F_{\alpha+\bbullet}E$ (by using the formula given in Remark \ref{rem:FirrEu} in the $u$-chart). Then the proof of Theorem \ref{th:FFirr} gives in fact the equality $F_{\alpha+\bbullet}E=F^\irr_{\alpha+\bbullet}E$.

Let $\pi:\wt\ccX\to\ccX$ be a projective birational morphism such that $vf$ extend as a morphism $\wt{vf}:\wt\ccX\to\PP^1$ and $\wt\ccD:=\pi^{-1}(\ccD)$ is a normal crossing divisor in $\wt\ccX$. The pole divisor of $vf$ is $\ccP_\red$ and that of~$\wt{vf}$, that we denote by $\wt\ccP_\red$, is contained in $\pi^{-1}(\ccP_\red)$. We denote by $\wt\ccH$ the remaining components of $\wt\ccD$. The construction of \cite{E-S-Y13} produces a filtration $F_{\alpha+\bbullet}^\Del\ccE^{\wt{vf}}(*\wt\ccH)$. Note that
\[
\pi_+\ccE^{\wt{vf}}(*\wt\ccH)=\cH^0\pi_+\ccE^{\wt{vf}}(*\wt\ccH)=\ccE^{(v:u)f}(*\ccH)=:\ccE.
\]
By Theorem \ref{th:1main}, the push-forward $\pi_+\big(\ccE^{\wt{vf}}(*\wt\ccH),F_{\alpha+\bbullet}^\Del\ccE^{\wt{vf}}(*\wt\ccH)\big)$ is strict, since $F_{\alpha+\bbullet}^\Del\ccE^{\wt{vf}}(*\wt\ccH)=F_{\alpha+\bbullet}^\irr\ccE^{\wt{vf}}(*\wt\ccH)$, and produces the filtration $F_{\alpha+\bbullet}^\irr\ccE$, which is nothing but $F_{\alpha+\bbullet}\ccE$ by Theorem \ref{th:FFirr} in the present setting. The strictness of the push-forward implies a quasi-isomorphism at the de~Rham level:
\begin{equation}\label{eq:FFDelpi}
F_\alpha^p\DR\ccE\simeq\bR\pi_*F_{\Del,\alpha}^p\DR\ccE^{\wt{vf}}(*\wt\ccH),
\end{equation}
where, as usual, we set for a filtered $\cD$-module $(\ccM,F_\bbullet\ccM)$,
\[
F^p\DR\ccM=\{F_{-p}\ccM\to \Omega^1\otimes F_{-p+1}\ccM\to\cdots\}.
\]
We will show how to recover the quasi-isomorphism \eqref{eq:FFDelpi} for a suitable modification $\pi:\wt\ccX\to\ccX$ by a direct computation. This will give, in the present setting, a proof of the degeneration at $E_1$ of the spectral sequence attached to the hypercohomology of $\bH^k(\ccX,F_{\alpha+\bbullet}\DR\ccE)$ which only relies on \cite{E-S-Y13} for $\wt{vf}:\wt\ccX\to\PP^1$, and not on the finer results of Theorem \ref{th:1main}. However, the identification at the level of filtered $\cD$-modules, and not only at the level of filtered de~Rham complexes, is needed for the application to Kontsevich bundles given in Theorem \ref{th:2main}.

Let us set, for each $p$ and $\alpha\in[0,1)$,
\[
F_{\rN,\alpha}^p\DR\ccE=\Big\{\Omega_\ccX^p(\log\ccD)([\alpha\ccP])\To{\rd+\rd(vf)}\Omega_\ccX^{p+1}(\log\ccD)([(\alpha+1)\ccP])\to\cdots\Big\}[-p].
\]
Such a filtration already appeared in \cite{Yu12} in the study of the toric case, where the notation $F^\lambda_\mathrm{NP}(\nabla)$ was used (NP for Newton polygon).

\begin{lemme}\label{lem:Yu}
The natural morphism $F_{\rN,\alpha}^p\DR\ccE\to F_\alpha^p\DR\ccE$ is a quasi-isomorphism.
\end{lemme}

\begin{proof}
Let us prove the lemma in the $v$-chart for instance (the proof in the $u$-chart is similar), and let us assume that $H=\emptyset$ for the sake of simplicity, so that $\ccD=\ccP_\red$ (the general case is obtained by a Kunneth formula). Recall that $n=\dim X$. Everything below is thus only valid on $X\times\Afu_v$. Consider the following complexes, with differentials induced by $\rd+\rd(vf)$:
\begin{multline*}
\Phi_1(p):\Big\{\Omega^p_\ccX(\log\ccP_\red)\big([\alpha\ccP]\big)\to \Omega^{p+1}_\ccX(\log\ccP_\red)\big([(\alpha+1)\ccP]\big)\to\cdots\\
\to \Omega^{n+1}_\ccX(\log\ccP_\red)\big([(\alpha+n-p+1)\ccP]\big)\Big\}[-p]
\end{multline*}
and, for $k$ such that $1<k\leq n-p+2$,
\begin{multline*}
\Phi_k(p):\Big\{\Phi_1^{<n-k+2}\to F_0\cO_\ccX(*\ccP_\red)\Omega_\ccX^{n-k+2}\big([(\alpha+n-p-k+2)\ccP]\big)\to\cdots\\
\to\Big(\sum_{j=0}^{k-1}F_j\cO_\ccX(*\ccP_\red)v^j\Big)\Omega_\ccX^{n+1}\big([(\alpha+n-p+1)\ccP]\big)\Big\}.
\end{multline*}
Then $\Phi_1(p)=F_{\mathrm{NP},\alpha}^p\DR\ccE$ and $\Phi_{n-p+2}(p)=F_\alpha^p\DR\ccE$. On each successive quotient $\gr_k^\Phi=\Phi_k/\Phi_{k-1}$, the induced differential becomes $-(v/x^\bme)\sum e_i\rd x_i/x_i$. Except at the first non-zero term, the complex $\gr_k^\Phi$ decomposes into many parts of the Koszul complex associated with $-(v/x^\bme)\{e_1\rd x_1/x_1,\dots,e_\ell\rd x_\ell/x_\ell\}$. By a direct computation, the first non-zero chain map of $\gr_k^\Phi$ is injective. In particular, $\gr_k^\Phi$ is quasi-isomorphic to zero.
\end{proof}

Let us now end the direct proof of \eqref{eq:FFDelpi}. In the discussion of the toric case in \cite[\S4]{Yu12}, a specific resolution $\pi:\wt\ccX\to\ccX$ of $vf$ is constructed inductively by taking blowups along irreducible components of the intersection of the pole set $\ccP_\red$ of $vf$ with its zero set $(f^{-1}(0)\times\PP^1_v)\cup(X\times\{v=0\})$. Then it is shown in \loccit that \eqref{eq:FFDelpi} holds when we replace its left-hand side with $F_{\rN,\alpha}^p\DR\ccE$. Lemma \ref{lem:Yu} allows us to conclude.\qed

\appendix
\section*{Appendix. Brieskorn lattices and Hodge filtration}
\def\thesection{A}
\subsection{Brieskorn lattices in dimension one}\label{subsec:Brieskorndim1}
Let $(M,F_\bbullet M)$ be a holonomic $\Clt$-module equipped with a good filtration. We denote by $G$ the holonomic $\Clt$-module $\CC[\partial_t,\partial_t^{-1}]\otimes_{\CC[\partial_t]}M$. If we identify $\Clt$ with $\Clv$ by the Laplace correspondence $t\mto\partial_v$, $\partial_t\mto-v$, we also regard $G$ as a holonomic $\Clv$-module on which the multiplication by $v$ is bijective. It is therefore also a $\CC[v,v^{-1}]$-module. We will denote by $\wh\loc$ the natural morphism $M\to G$.

The Brieskorn lattice $G_0^{(F_\bbullet)}$ of the filtration $F_\bbullet M$ is defined as the saturation of the filtration by the operator $\partial_t^{-1}$, that is,
\[\tag{$*$}
G_0^{(F_\bbullet)}:=\sum_j\partial_t^{-j}\wh\loc(F_jM)\subset G.
\]
It is naturally a $\CC[\partial_t^{-1}]$-module (equivalently, a $\CC[v^{-1}]$-module). We will also set $G^p_{(F_\bbullet)}=v^{-p}G_0^{(F_\bbullet)}$ for any $p\in\ZZ$. Let us make the link with the definition in \cite[\S1.d]{Bibi05}. Let $p_o$ be an index of generation, so that $F_{p_o+\ell}M=F_{p_o}M+\cdots+\partial_t^\ell F_{p_o}M$ for any $\ell\geq0$. Then the definition in \loccit is
\[\tag{$**$}
G_0^{(F_\bbullet)}=\partial_t^{-p_o}\sum_{j\geq0}\partial_t^{-j}\wh\loc(F_{p_o}M).
\]
Let us check that both definitions give the same result. Let us write $(*)$ as
\[
G_0^{(F_\bbullet)}=\partial_t^{-p_o}\sum_j\partial_t^{-j}\wh\loc(F_{p_o+j}M).
\]
Firstly, for $j\leq0$, we have
\[
\partial_t^{-j}\wh\loc(F_{p_o+j}M)=\wh\loc(\partial_t^{-j}F_{p_o+j}M)\subset \wh\loc(F_{p_o}M),
\]
so we can also write
\begin{align*}
G_0^{(F_\bbullet)}&=\partial_t^{-p_o}\sum_{j\geq0}\partial_t^{-j}\wh\loc(F_{p_o+j}M)\\
&=\partial_t^{-p_o}\sum_{j\geq0}\partial_t^{-j}\big[\wh\loc(F_{p_o}M)+\cdots+\partial_t^j\wh\loc(F_{p_o}M)\big]\\
&=\partial_t^{-p_o}\sum_{j\geq0}\partial_t^{-j}\wh\loc(F_{p_o}M)=(**).
\end{align*}

We now express the Brieskorn lattice of the filtration as obtained by a push-forward operation. We consider the holonomic $\Cltv$-module $M[v,v^{-1}]\mathrm{e}^{vt}$. The $(t,v)$-Brieskorn lattice is the $\CC[t,v^{-1}]$-module defined by the following formula \hbox{(\cf\cite[\S1]{Bibi97b})}:
\[
\cG_0(M,F_\bbullet M)=\bigoplus_jF_jM\cdot v^{-j}\mathrm{e}^{vt}\subset M[v,v^{-1}]\mathrm{e}^{vt},\quad \cG^p(M,F_\bbullet M)=v^{-p}\cG_0(M,F_\bbullet M).
\]
We have $\partial_t\cG^p(M,F_\bbullet M)\!\subset\! \cG^{p-1}(M,F_\bbullet M)$ since $\partial_t F_jM\!\subset\! F_{j+1}M$ and $\partial_t\mathrm{e}^{vt}\!=\!v\mathrm{e}^{vt}$. The relative de~Rham complex
\[
\DR(M[v,v^{-1}]\mathrm{e}^{vt}):=\Big\{M[v,v^{-1}]\mathrm{e}^{vt}\To{\partial_t}M[v,v^{-1}]\mathrm{e}^{vt}\Big\}
\]
has cohomology in degree one only, and we have a natural identification as $\Clv$-modules
\[
\coker\Big[\partial_t:M[v,v^{-1}]\mathrm{e}^{vt}\to M[v,v^{-1}]\mathrm{e}^{vt}\Big]\simeq G
\]
by sending $\sum_jm_jv^{-j}\mathrm{e}^{vt}$ to $\sum_j(-\partial_t)^{-j}\wh\loc(m_j)$. This relative de~Rham complex is filtered by the subcomplexes
\begin{equation}\label{eq:GDR}
\cG^p\DR(M[v,v^{-1}]\mathrm{e}^{vt}):=\Big\{\cG^p(M,F_\bbullet M)\To{\partial_t}\cG^{p-1}(M,F_\bbullet M)\Big\}.
\end{equation}

\begin{lemme}[Push-forward]\label{lem:pushforwardG}
The relative de~Rham complex is strictly filtered by the $\cG^\bbullet$-filtration and, through the previous identification, the filtration on its $H^1\simeq G$ is equal to $G^{\cbbullet-1}_{(F_\bbullet)}$.
\end{lemme}

\begin{proof}
With respect to the previous identification, $\cG^p(M,F_\bbullet M)$ is sent onto $G^p_{(F_\bbullet)}$ according to the definition $(*)$. It remains to show that $(\image\partial_t)\cap \cG^p(M,F_\bbullet M)=\partial_t\cG^{p+1}(M,F_\bbullet M)$ and it is enough to check this for $p=0$.

We have $\partial_t\big(\sum_jm_jv^{-j}\mathrm{e}^{vt}\big)=\sum_j(\partial_tm_j+m_{j+1})v^{-j}\mathrm{e}^{vt}$ and by induction on $j$ we deduce that $(\partial_tm_j+m_{j+1})\in\ F_jM$ for all $j$ implies $m_j\in F_{j-1}M$ for all $j$.
\end{proof}

\begin{remarque}[Rees modules]\label{rem:Rees}
It will also be useful to have the following interpretation in terms of Rees modules (\cf Proof of Theorem \ref{th:Firr}), for which we use the variable~$u$ instead of $\hb$ here. We can twist the Rees module $R_FM$ by~$\mathrm{e}^{t/u}$ by changing the action of $u\partial_t$ to that of $u\partial_t+1$. We denote the corresponding $R_F\Clt$-module by $R_FM\cdot\mathrm{e}^{t/u}$. This is nothing but $\cG_0(M,F_\bbullet M)$ by the change of variable $u=v^{-1}$.

The push-forward of an $R_F\Clt$-module $\cM$ by the constant map $q:\Afu_t=\Spec\CC[t]\to\Spec\CC$ is nothing but the de~Rham complex $\partial_t:\cM\to u^{-1}\cM$, where the latter term is in degree zero. The push-forward $q_+(R_FM\cdot\mathrm{e}^{t/u})$ is thus equal to the complex (with the $\cbbullet$ in degree zero):
\[
R_FM\cdot\mathrm{e}^{t/u}\To{\partial_t}\underset{\bbullet}{u^{-1}R_FM\cdot\mathrm{e}^{t/u}}.
\]
Setting $u=v^{-1}$ we thus have an identification
\[
\xymatrix{
\cG_0(M,F_\bbullet M)\ar[r]^-{\partial_t}\ar@{=}[d]&\cG^{-1}(M,F_\bbullet M)\ar@{=}[d]\\
R_FM\cdot\mathrm{e}^{t/u}\ar[r]^-{\partial_t}&u^{-1}R_FM\cdot\mathrm{e}^{t/u}
}
\]
so we can interpret $G^{-1}_{(F_\bbullet)}$ as $H^0q_+(R_FM\cdot\mathrm{e}^{t/u})$, while $H^{-1}q_+(R_FM\cdot\mathrm{e}^{t/u})=0$.
\end{remarque}

\subsection{Brieskorn lattices in arbitrary dimension}\label{subsec:Brieskorndim>1}
We fix $k$ and we will apply the previous result to $(M,F_\bbullet M)=\cH^{k-\dim X}f_+(\cO_X(*D),(F_\bbullet\cO_X(*H))(*P_\red))$. Here, we identify filtered $\Clt$-modules and $\cD_{\PP^1}(*\infty)$-modules filtered by $\cO_{\PP^1}(*\infty)$-modules. We know that the latter underlies a mixed Hodge module (up to a shift of the filtration), according to \cite{MSaito87}. Working with Rees modules, the strictness property for the push-forward~$f_+$ of mixed Hodge modules can also be stated by saying that the push-forward $f_+\big[(R_F\cO_X(*H))(*P_\red)\big]$ is strict, and thus
\[
\cH^{k-\dim X}f_+\big((R_F\cO_X(*H))(*P_\red)\big)=R_FM.
\]
On the other hand, one checks that
\[
\cH^{k-\dim X}f_+\big((R_F\cO_X(*H))(*P_\red)\cdot\mathrm{e}^{f/u}\big)\simeq\big(\cH^{k-\dim X}f_+(R_F\cO_X(*H)(*P_\red)\big)\cdot\mathrm{e}^{t/u},
\]
and since $\cH^jq_+(R_FM\cdot\mathrm{e}^{t/u})=0$ for $j\neq0$, we conclude
\[
\cH^{k-\dim X}(q\circ f)_+\big((R_F\cO_X(*H))(*P_\red)\cdot\mathrm{e}^{f/u}\big)\simeq\cH^0q_+(R_FM\cdot\mathrm{e}^{t/u}).
\]
The left-hand term is by definition equal to
\[
\bH^k\big(X,(\Omega_X^\cbbullet\otimes_{\cO_X}\!\!(u^{-\cbbullet}R_F\cO_X(*H))(*P_\red),\rd+u^{-1}\rd f)\big),
\]
that is, to $G_0\bH^k_u$ as defined by \eqref{eq:BrieskornHkv}, while the right-hand term is equal to~$G^{-1}$ as defined above.

\backmatter
\def\og{}\def\fg{}
\def\smfedsname{eds.}
\def\smfedname{ed.}
\providecommand{\SortNoop}[1]{}\providecommand{\eprint}[1]{\href{http://arxiv.org/abs/#1}{\texttt{arXiv\string:\allowbreak#1}}}
\providecommand{\bysame}{\leavevmode ---\ }
\providecommand{\og}{``}
\providecommand{\fg}{''}
\providecommand{\smfandname}{\&}
\providecommand{\smfedsname}{\'eds.}
\providecommand{\smfedname}{\'ed.}
\providecommand{\smfmastersthesisname}{M\'emoire}
\providecommand{\smfphdthesisname}{Th\`ese}

\end{document}